\documentclass[a4paper,12pt]{article}

\usepackage{amsmath}
\usepackage{amsthm}
\usepackage{amssymb}
\usepackage{latexsym}
\usepackage{graphicx}
\usepackage{stmaryrd}
\usepackage{mathrsfs}
\usepackage{enumerate}
\usepackage{cases}
\usepackage{color}

\newcommand{\BLACK}{\color{black}}

\definecolor{dGREEN}{rgb}{0.0,0.5,0.5}

\newcommand{\glalign}[2]{\lower.6ex\vbox{
\baselineskip\lineskip\ialign{$#1\hfil##\hfil$\crcr#2\crcr=\crcr}}}

\newcommand{\del}{\partial}

\newcommand{\Om}{\Omega}

\newcommand{\domega}{\,{\rm d}\omega}

\newcommand{\dtau}{\,{\rm d}\tau}

\newcommand{\ds}{\,{\rm d}s}
\newcommand{\dt}{\,{\rm d}t}

\newcommand{\dx}{\,{\rm d}x}
\newcommand{\dy}{\,{\rm d}y}
\newcommand{\dz}{\,{\rm d}z}

\renewcommand{\div}{\mbox{\rm div}\,}

\newcommand{\supp}{\mbox{\rm supp}\,}

\newcommand{\IR}{\mathbb{R}}
\newcommand{\IC}{\mathbb{C}}
\newcommand{\IN}{\mathbb{N}}
\newcommand{\IZ}{\mathbb{Z}}

\newcommand{\D}{\mathcal{D}}
\newcommand{\HH}{\mathcal{H}}

\renewcommand{\L}{\mathrm{L}}

\DeclareMathOperator{\id}{id}

\newcommand{\dv}{{\rm div}\,}

\newcommand{\BB}{{\mathbb B}}
\newcommand{\BR}{{\mathbb R}}

\newcommand{\CA}{{\mathcal A}}
\newcommand{\CB}{{\mathcal B}}

\newcommand{\CR}{{\mathcal R}}

\newcommand{\CU}{{\mathcal U}}
\newcommand{\CQ}{{\mathcal Q}}

\newcommand{\CY}{{\mathcal Y}}
\newcommand{\CZ}{{\mathcal Z}}

\newcommand{\bff}{{\mathbf f}}

\def\eqn#1$$#2$${\begin{equation}\label#1#2\end{equation}}

\numberwithin{equation}{section}

\newtheorem{defi}{Definition}[section]
\newtheorem{thm}[defi]{Theorem}
\newtheorem{cor}[defi]{Corollary}
\newtheorem{prop}[defi]{Proposition}
\newtheorem{lem}[defi]{Lemma}
\newtheorem{rem}[defi]{Remark}
\newtheorem{assumption}[defi]{Assumption}

\def\eqn#1$$#2$${\begin{equation}\label#1#2\end{equation}}

\numberwithin{equation}{section}
\numberwithin{equation}{section}
\allowdisplaybreaks[4]

\begin{document}

\title{
\bf \large
Time periodic problem of the Navier-Stokes equations\\ in an exterior domain with  periodically moving boundary}
\author{{\normalsize
Reinhard Farwig\footnote{
Fachbereich Mathematik, Technische Universit\"at Darmstadt, Schlossgartenstr. 7, 64289 Darmstadt, \quad Germany, \texttt{farwig@mathematik.tu-darmstadt.de}} \;
and \;
Kazuyuki Tsuda\footnote{
Kyushu Sangyo University, 3-1 Matsukadai 2-chome,
Higashi-ku, Fukuoka,
813-8503 Japan, \texttt{k-tsuda@ip.}\texttt{kyusan-u.ac.jp}}   }\\[2ex]
}
\date{}
\maketitle
\vspace*{-7mm}
{\it To our academic teacher and colleague Hermann Sohr, an outstanding mathematician}\\

\begin{abstract}
\noindent
In this paper we consider the Navier-Stokes equations in exterior domains of $\IR^n$, $n\geq 3$, with a periodically in time moving boundary $\partial\Omega(t)$ and external force $f(t)$. For this case we prove the existence of a locally unique mild time periodic solution in weighted function spaces with radially symmetric Muckenhoupt weights. The solutions split into a stationary part controlled by potential theoretic estimates and a purely oscillatory part constructed as mild solution via analytic semigroup theory. To deal with perturbation terms of even second order - coming from a coordinate transform and the moving boundary - in weighted, homogeneous Sobolev spaces 
a maximal $L^1$ type regularity estimate  will be used in weighted Lorentz spaces.  
 To control the convective term an $\mathcal H^\infty$-calculus in weighted spaces of the Stokes operator, its $BIP$ property and embedding estimates of fractional powers are exploited, see a recent paper by the authors: {\em The Stokes operator on exterior domains in homogeneous weighted function spaces: From weak theory to $\mathscr H^\infty$-calculus to fractional domains} (2025).
\end{abstract}

\\


\section{Introduction}

For $t\in \IR$ let $\Omega(t)\subset \IR^n$, $n \geq 3$, be a family of exterior domains considered as perturbation of an exterior reference domain $\Omega_0 \subset\IR^n$ with $\partial \Omega(t)\in C^3$, and set $Q:= \bigcup_{t\in\IR} \Omega(t)\times\{t\}$. We assume that $\Omega(t)$ is time periodic with period $T>0$, {\em i.e.} $\Omega(t+T)=\Omega(t)$, and that for simplicity $\Omega(0) = \Omega_0$. 
Suppose that $\overline{\Omega(t)^c}$ is contained in an open ball $B_R$ of radius $R>0$ for all times $t$.  \BLACK

 The Navier-Stokes equations on the non-cylindrical space-time domain $Q$ are described by
\begin{align}\label{equ:ns}
	\begin{aligned}
		v_t - \nu\Delta v + v\cdot \nabla v + \nabla p & = f \quad\! \text{in }\;\, Q,\\
		 \div v & = 0\quad   \text{in }\;\, Q,\\
		v& = 0  \quad  \text{on }\, \bigcup_{t\in\IR} \partial \Omega(t)\times \{t\}.
	\end{aligned}
\end{align}
Here $v$ denotes the unknown velocity of an incompressible, viscous  fluid and $p$ denotes the pressure, respectively, at time $t\in \IR$ and position $x\in\Omega(t)$. Moreover, let $f=f(x,t)$ be a given external force with the same period in time, {\em i.e.}  $f(x,t+T)=f(x,t)$. For simplicity,  assume that $v=0$ on $\partial\Omega(t)$; the correct kinematic boundary condition will lead by classical trace arguments to a Navier-Stokes system with perturbation terms which are of order zero and one in $v$. For simplicity, the coefficient of viscosity $\nu$ equals $1$.

\begin{assumption}\label{ass}
	{\rm
There exists a family of unbounded domains $\Omega(t)\subset \IR^n$ with
	$\partial \Omega(t)\in C^3$, $t\in  J=\IR/(T\mathbb Z)$, 
    \BLACK and a map
	$$
		\psi \colon \overline{\Omega_0\times J \BLACK}\to \overline{Q},
		\; (\xi,t)\mapsto \psi(\xi,t) = (\phi(\xi,t),t),
	$$
where $\Omega_0\subset \IR^n$ is an unbounded reference domain with $\partial \Omega_0\in C^3$.
	Furthermore, the function $\phi$ has the following properties:
\begin{enumerate}
		\item[(i)] 
		For $t\in  J \BLACK$ the map
		$\phi(\cdot,t) \colon \overline{\Omega_0}\to \overline{\Omega(t)}$ is a
		$C^3$-diffeomorphism with $\phi(\cdot,0)=\id$.
		Let $\phi(\cdot,t)^{-1}$ denote its inverse for fixed $t$.
		\item[(ii)] 
		Concerning the function $\phi$ as a map on $\Omega_0\times \mathbb{R}$ we assume that
		\[\quad
			\phi\in \widehat C^{3,1}_b:=\big\{f\in C^0(\Omega_0\times J) :
			\partial^k_t\partial^{\alpha}_{\xi}f\in C^{0}_b,\, 1\leq 2k+|\alpha|\leq 3,\,k\in \IN_0,\,
			\alpha\in \IN_0^n\big\}.
		\]
		\item[(iii)] 
		The map $\phi$ is volume preserving, that is,
		$\det \nabla_{\xi}\phi(\cdot,t)=1$ for all $t$.
        

\item[(iv)] 
The map $\phi$ is $\mu$-H\"{o}lder continuous in time in the norm of $C^{3.1}$, {\em i.e.,} there exist constants $c_\phi>0$
and $\mu\in (0,1]$ such that
\begin{align}\label{phi-mu}
|\phi(t)-\phi(s)|_{C^{3,1}} \leq c_\phi |t-s|^{\mu}\quad\textrm{ for all }\; t, s\in J,
\end{align}
where $c_\phi$ is sufficiently small. 
In addition, 
\begin{align}\label{assumpt-phi}
\begin{aligned}
\|(1+ & |\xi|)^{n-1} \partial^k_t\partial^{\alpha}_{\xi} (\phi(\xi,t)-\phi(\xi,s))\|_{L^\infty} + \|(1+|\xi|)^{n} \partial^k_t\partial^{\alpha}_{\xi} (\phi(\xi,t)-\phi(\xi,s))\|_{L^{q_1,1}}\\
& \leq c_\phi |t-s|^{\mu}  
\end{aligned}
\end{align}
for $0\leq 2k+|\alpha|\leq 3$, where 
$\frac{1}{q_1}+\frac{1}{q}< \frac{1}{n}$ and $q>n$ is given in the solution spaces $X_s$ and $X_\perp$, see Subsect. \ref{S2.1} below. \BLACK
 
	\end{enumerate}
	}
\end{assumption}
By this assumption we may change the variable $x\in\Omega(t)$ to the variable $\xi$ on the reference domain $\Omega_0$ explicitly by the mapping $\phi$ and will consider the velocity on $\Omega_0$. 
Such a coordinate transform is similar to the {\em Arbitrary Lagrangian–Eulerian} (ALE) method in numerical simulation where the problem is transferred from a moving domain to a fixed or an even moving reference domain, see {\em e.g.}~Duarte, Gormaz and Natesan \cite{DGN-ALE}, Benselama and Monnier \cite{BM-ALE} for the Navier-Stokes equations. 
For the ALE method, a domain velocity  of the computational mesh \BLACK is used. On the other hand, we define a velocity $u$ on $\Omega_0$ by \eqref{equ:trafo} below to preserve the divergence free condition in mathematical analysis. 
Using the coordinate transform $\phi(t)$ the Navier-Stokes system \eqref{equ:ns} will be rewritten as a modified Navier-Stokes system on the fixed exterior reference domain $\Omega_0$ and has the form 
\begin{align}\label{equ:ns-new}
\begin{aligned}
	u_t - \Delta u - \sum_{|\alpha|\leq 2 } a_{\alpha} \partial^{\alpha} u + \sum_{|\beta|\leq 1 } b_{\beta}\partial^{\beta} u + \nabla^{\phi(t)} \tilde p + u\cdot \nabla^{\phi(t)} u & =  \Phi f\BLACK,\\
	\div u = 0,\quad u|_{\partial\Omega_0} & =0, 
\end{aligned} 
\end{align}
where the perturbation terms $a_\alpha, b_\beta$, the modified gradients $\nabla^{\phi(t)}$ and $\Phi$ will be explained in Subsect. \ref{S2.3}.

To solve \eqref{equ:ns-new} we split $(u, p) $ into a stationary part $(u_s,p_s)$ defined by $u_s=\frac{1}{T} \int_0^T u\dtau$ and a purely oscillatory unsteady part $u_\perp = u-u_s$; similarly, we decompose $p$ into $p_s$ and $p_\perp$. In view of this splitting the solution space consists of two parts, $X_s$ for the stationary part of the solution and $X_\perp$ for its instationary part:
\begin{align*}
X_s = & \Big\{(u,p)  \in L^\infty_{n-2}(\Omega_0) \cap \widehat H^{1,\infty}_{n-1}(\Omega_0) \cap H^{1,q}(\Omega_0) \cap \widehat{H}^{2;q,\infty}(\Omega_0) \times \widehat{H}^{1;q,\infty}(\Omega_0) : \\
& \quad \div u=0 \textrm{ in } \Omega_0,\; u|_{\partial\Omega_0}=0\Big\}, \\
X_\perp = & \Big\{ (u,p) \in L^\infty_{per}(\mathbb{R}; H^{1,q}_{n-1}\BLACK (\Omega_0))\times L^\infty_{per} (\mathbb{R};\widehat{H}^{1;q,\infty}(\Omega_0)):\\ 
& \quad
 \nabla^2 u \in  L^\infty_{per} (\mathbb{R};L^{q,\infty}_{n-1\BLACK}(\Omega_0)),\, \div u=0 \textrm{ in } \Omega_0,\, u=0\textrm{ on } \partial\Omega_0\Big\}
\end{align*} 
 Both spaces are defined by weighted Lebesgue spaces $L^q$, weighted Lorentz spaces $L^{q,\infty}$ and nonhomogeneous as well as homogeneous Sobolev spaces $H^{1,q}, \widehat H^{1;q,\infty}$; the subscript $n-1$ indicates the weight $(1+|x|^2)^{\frac{n-1}{2}}$, {\em etc.} For details we refer to Subsect. \ref{S2.1}.\\

Then the main result on existence of $T$-periodic solutions to \eqref{equ:ns-new} reads as follows:

\begin{thm}\label{main}
Let $n<q<\infty$, $\Omega_0\subset\IR^n$ be an exterior reference domain of class $C^3$ and $\phi$ be a $T$-periodic coordinate transform as in Assumption \ref{ass}. Moreover, let $f=f_s+f_\perp$ be a $T$-periodic external force such that
\begin{align}\label{f-ass}
\begin{aligned}
& f_s\in L^q_n(\Omega_0),\\
& f_\perp \in L^\infty_{per}(\IR;L^q_{n-1+\delta}(\Omega_0)) 
\BLACK ,
\end{aligned}
\end{align}
where $0<\delta< 1-\frac{n}{q}$.  \BLACK Then there exists a closed ball $\mathcal B \subset X_s\times X_\perp$ and  a positive constant $\epsilon_f$ such that  under the assumption 
\begin{equation}\label{f-small}
\|f_s\|_{L^q_n(\Omega_0)}+\|f_\perp\|_{L^\infty_{per}(\IR;L^q_{n-1+\delta}(\Omega_0))} \leq \epsilon_f
\end{equation}
\BLACK 
the modified Navier-Stokes system \eqref{equ:ns-new} possesses a unique $T$-periodic solution $(u,p) \cong (u_s, p_s, u_\perp, p_\perp) \in \mathcal B$ such that $(u_s,p_s)\in X_s$ and $(u_\perp,p_\perp)\in X_\perp$.
\end{thm}

We refer to Remark \ref{rem-norms-f} (ii) below for other admissible conditions on $f$.

\BLACK

 \vspace{1ex}
Time periodic flow is one of the basic phenomena in fluid mechanics and 
has been discussed extensively. 
Serrin \cite{Serrin} started to study the time periodic problem on fixed bounded domains. Kaniel and Shinbrot \cite{Kaniel-Shinbrot} continued the work  by emphasizing the so-called reproductive property. \BLACK
On a fixed unbounded domain including exterior domains, Kozono and Nakao \cite{Kozono-Nakao} solved the time periodic problem under small time periodic external forces for dimension $n \geq 4$ by a semigroup approach. 
Then, working in weak $L^n$ spaces, Yamazaki \cite{Yamazaki} included also the three dimensional case. Okabe and Tsutsui \cite{OkabeTsutsui} removed the divergence form condition of the external force assumed in \cite{Kozono-Nakao} on the whole space case. 


For the time periodic problem with periodically moving boundary
Morimoto \cite{Morimoto} as well as Miyakawa and Teramoto \cite{Miyakawa-Teramoto} show existence of weak time periodic $L^2$ solutions. 
Then Teramoto \cite{Teramoto} proves existence of strong periodic $L^2$ solutions in bounded domains with the help of regularity theory of weak solutions provided that the periodic motion of the domain is sufficiently small.
The authors of the present paper, jointly with Kozono and Wegmann \cite{FKTW}, obtained  for bounded domains \BLACK existence of mild periodic $L^q$ solutions under suitable smallness conditions on the external forces and the periodic motion of the boundary. 
Furthermore, the authors \cite{FT-initial-value-prob} show that the periodic solutions are strong ones in a critical scale invariant space.  
Finally, the authors \cite{FT-lin-half, FT-lin-half2} prove existence of time periodic mild solutions in critical scale invariant spaces on half spaces with moving boundary. 

On the other hand, the case of exterior domains with moving boundary is more involved. In the three dimensional case the first author, Kozono and Wegmann \cite{Farwig-Kozono-Wegmann} and Eiter and Shibata \cite{Eiter-Shibata} \BLACK 
assume smallness of the motion of the boundary to get global-in-time  solutions; 
in addition, they considered the locally interacting case, {\em i.e.} $\nabla \phi-I\BLACK$ has compact support.  \BLACK
The authors of \cite{Farwig-Kozono-Wegmann} show global-in-time maximal regularity for the initial value problem and existence of strong solutions. Eiter, Kyed and Shibata \cite{Eiter-Kyed-Shibata21} study the
time periodic problem  with a periodic free boundary condition, using a partial Lagrangian coordinate,  {\em i.e.,} due to a  cut-off function 
the coordinate transform leads to genuine Lagrangian coordinates only on some ball.
%
A fundamental idea is time periodic maximal regularity established by the transference theorem of de Leeuw adapted to operator valued multipliers.  
Eiter and Shibata \cite{Eiter-Shibata} solve the time periodic moving boundary problem \eqref{equ:ns},
including the Oseen case in which a non-vanishing velocity $u_\infty$ at space infinity is allowed; then they exploit a partial Lagrangian coordinate similar as above and the transference theorem. 

{\bf New approach:} In the present paper, we solve the time periodic moving boundary problem  without any localizing effect in the coordinate transform, \BLACK {\em i.e.,} $\phi(\xi,t) \neq \xi $ for all $\xi\in\Omega_0$ is allowed.  
Using the semigroup approach we exploit a completely new method, {\em viz.} the weighted theory approach  with radially symmetric weights of type $(1+|x|^2)^{s/2}$, $s\in\IR$. \BLACK
This approach makes the moving boundary problem easier to handle because it is based on  well known potential theoretical estimates and Stokes semigroup theory.  
Moreover, we believe that this approach will be applicable to other unbounded domains,  especially with non-compact boundary like the half space. Indeed, in \cite{FT-lin-half,FT-lin-half2}, we assumed that the boundary is moving only on a compact subset of $\partial\IR^n_+$  
in order to establish weighted estimates of the Stokes semigroup on half spaces, {\em cf.} Kobayashi and Kubo \cite{Kobayashi-Kubo-half} who announced such estimates for a half space case. 
It is revealed that the weighted theory is effective also for perturbed systems of the Navier-Stokes equations; for example, we mention the moving boundary problem in this article and the stability problem of stationary solutions in \cite{Farwig-Tsuda-stab}.  
In the proof, we make use of
the $\mathscr H^\infty$-calculus of the Stokes operator in weighted spaces on exterior domains, see \cite{Farwig-Tsuda-Hinfty} and Subsect. \ref{S2.2}; for this result in spaces 
without weights we refer to  Noll and Saal \cite{NollSaal}. The $\mathcal H^\infty$-calculus not only derives the characterization of domains of fractional powers in weighted norms, but is also   useful for perturbation theory of the Stokes operator as considered in our paper \cite{FT-unif} for families of perturbations and applied in \cite{FT-lin-half}.  \BLACK

\BLACK

{\bf Difficulties:} The  non-localized impact of  the coordinate transform seems to be an essential problem to get global-in-time solutions. 
In the moving boundary problem, handling a quasi-linear term arising from the boundary condition is most difficult, {\em i.e.,}~$g \Delta u$, where $g= a_\alpha$ (the $a_\alpha$'s are given functions determined by $\phi$, see \eqref{equ:trafodgl}) in our case and $g = \int_{0}^t \nabla u(\tau) \dtau $ in the free boundary case.  
Shibata \cite[Sect. 8]{Shibata-lecturer} shows existence of global solutions to the initial value problem in an exterior domain with free boundary, using a partial Lagrangian coordinate to control this term. 
In view of this point, in the higher dimensional case {\em i.e.,}~$4\leq n$,  Oishi and Shibata \cite{Oishi-Shibata} study to remove such a localized condition and get  global-in-time solutions. 
Saal \cite{Saal, Saal6thMiss} shows existence of local-in-time solutions to the initial value problem in a exterior domain with a given moving boundary. The result is extended globally-in-time by \cite{Farwig-Kozono-Wegmann} in which such a localized condition to control perturbation terms including $g \Delta u$ is used. 
Concerning the free boundary problem globally-in-time in the half space case, a Besov space approach has been applied which does not need such localized effects, {\em cf.} Ogawa and Shimizu  \cite{OgawaShimizu},  \BLACK  Danchin, Hieber, Mucha and Tolksdorf \cite{DHMT}. 
We note that this approach can be applied to both the whole space and half space, but not  directly to our exterior domain case. In addition, we consider the time periodic problem, not the initial value problem. 
\BLACK

The authors solved the time periodic problem for the periodically moving half space in \cite{FT-lin-half, FT-lin-half2}, 
assuming that coefficients in the perturbation terms have compact support.  
The crucial point is to construct the evolution operators  of a modified $t$-dependent Stokes operator and of its adjoint, 
\BLACK  and derive global-in-time estimates of both operators. However, the evolution operator approach seems to pose new difficulties in the exterior domain case.  
Indeed, we need to use the reasonable estimate 
$$
\|\nabla^2 u\|_{L^p} \leq C\|A u\|_{L^p},
$$ 
to get uniform resolvent estimates which are required for the construction of the evolution operator.  
Thus we have to take the strong restriction $p < \frac{n}{2}$ in the exterior domain case.  This restriction makes further estimates more  difficult. 
Indeed, to control the $L^\infty$ norm of solutions, we need $L^q$ norms with $q>n$. 
Moreover, the time $L^q$-$L^p$ decay estimates of the evolution operator would be slower by the restriction $p< \frac{n}{2}$. 
This implies that in our case we can apply neither the method of Kozono and Nakao  \cite{Kozono-Nakao} nor Yamazaki type estimates \cite{Yamazaki}, which are based on suitable decay estimates, directly to the evolution operator.

{\bf Details of the strategy:} 
We solve the time periodic moving boundary problem by a classical semigroup approach.  
However, the Stokes semigroup is considered in weighted $L^q$ spaces with radial weights in the Muckenhoupt weight class.  
We adopt the following strategy: \BLACK

\vspace{1ex}

(1) Formulation 

Benefiting from recent studies about time periodic problems as in \cite{Celik-K, Eiter-Kyed-Shibata21}, we decompose solutions into a stationary part, $u_s$, and a purely oscillating part, $u_\perp$. 
The advantage of this decomposition is the choice of adequate function spaces for each part. 
In the stationary part, we  use weighted $L^\infty$ norms proposed by potential theoretical estimates, while we exploit weighted $L^q$ spaces for the oscillating part by applying $L^q$-$L^p$ estimates of the Stokes semigroup. 
We note that in our case the decomposition  technique works well in combining the potential theoretical estimates with the classical formulation used in  Kozono and Nakao \cite{Kozono-Nakao} to solve the time periodic problem. 
In \cite{Eiter-Kyed-Shibata21} it is used for the maximal time periodic regularity based on the transference theorem.  

We consider the pressure term in the formulation directly. In previous studies, for example \cite{Saal, Saal6thMiss}, Saal uses the non-local $t$-dependent family of projections $\mathbb P(t)$ to eliminate the pressure term.
However, it requires some restriction on weights which does not match the potential theoretical estimates. 
Moreover, $\mathbb P(t)$ and the projection $P_\perp$ do not commute, where $P_\perp$ is related to the decomposition method and defined in Sect. 2 below. 
Finally, we rewrite the oscillating part according to Kozono-Nakao to solve the time periodic problem by weighted $L^q$-$L^p$ decay estimates of the Stokes semigroup.  

\vspace{1ex}

(2) Analysis of the stationary part 

On the stationary part, we establish new potential theoretical estimates on exterior domains to handle the non-local terms. 
We extend the potential theoretical estimate used in \cite{Eiter-Kyed-Shibata21} to the one containing non-local terms and including the weak $L^q$ norm. 
The estimate is applicable to terms which do not have any divergence structure. 
Therefore, we can estimate nonlinear terms of the moving boundary problem in weighted $L^\infty$, $L^q$ and weak $L^q$ spaces.   

\vspace{1ex}

(3) Analysis of the pure oscillation part

Concerning the purely oscillating part, we apply $L^q$-$L^p$ decay estimates of the Stokes semigroup in weighted classes. 
This estimate is derived by uniform resolvent estimates of the Stokes operator on the Muckenhoupt weight class. 
To estimate quasi-linear terms which appear from the boundary condition, we consider the weak Lorentz class for second order derivatives of solutions with weights, and we apply a Yamazaki type estimate, in other words, a maximal $L^1$ type estimate in  Lorentz spaces ``partially'' with radial weights and negative powers, {\em i.e.,} 
\begin{equation}\label{loss}\
\int_0^\infty \tau^{\frac{n}{2}(\frac{1}{p}-\frac{1}{q})} \| A  e^{-\tau A}\BLACK u\|_{L^{q,1}_{-(n-1)}} \dtau  \leq C\|u\|_{L^{p,1}_{-(n-1)}}
\end{equation}
for $p<q$.
We apply this estimate only for the crucial quasi-linear terms in the duality argument to get closed estimates.  
The reason to exploit the duality argument is to transform $\nabla^2$ to the Stokes operator $A$ which commutes with the Stokes semigroup and allows for negative weight exponents $s$ in the uniform resolvent estimate.    
Recall that duality in Muckenhoupt classes with power $s>0$ leads to negative powers. Due to the duality argument, the weak $L^q$ norm is used in the class of solution. 
We also note that negative powers $s$ can be handled by the property of bounded purely imaginary powers ($BIP$) in weighted spaces established by the $\mathcal H^\infty$-calculus in \cite {Farwig-Tsuda-Hinfty}, see also Subsect.~\ref{S2.2}. 
On the other hand, negative powers cannot be handled in the usual $L^p$-$L^q$ decay estimate of the Stokes semigroup 
because they rely on local decay estimate or the uniform resolvent estimate which require positive exponents, see Remark \ref{why-we-need-BIP} below. 

\vspace{1ex}

 (4) Loss of regularity \BLACK

We note that there is some loss $ \tau^{\frac{n}{2}(\frac{1}{p}-\frac{1}{q})}$ for $p<q$  in \eqref{loss}. \BLACK 
Hence Yamazaki did not use this maximal $L^1$ type estimate to solve the time periodic problem in \cite{Yamazaki}. 
On the other hand, in our case this loss is compensated  by the H\"{o}lder condition in time of $\phi$,  see \eqref{phi-mu}, \eqref{assumpt-phi}, which allows the estimate to work for terms of type
$a_2 \nabla^2 u$. 
Due to this loss, we have to apply another method to estimate the convective term. 
Indeed, the $\mathcal H^\infty$-calculus plays an essential role. 
The $BIP$ property based on the $\mathcal H^\infty$-calculus  helps to \BLACK weaken the singularity of the $\nabla^2$ estimate in the time integral by regarding $\nabla^2 \sim A$ in weighted norms 
and splitting $A$ into the product of $A^{1-\theta}$ with $A^{\theta}$; here $A$ is the Stokes operator and $\theta$ is a small positive number so that $A^\theta$ can be applied to vector fields not vanishing on $\partial\Omega$; \BLACK see the proof of Proposition \ref{result-oscillatory-part-2nd-derivative} below. 
Other terms are estimated directly by the $L^q$-$L^p$ decay estimates of the Stokes semigroup. 
As for the pressure term, we apply the regularity structure of $\nabla^{\phi(t)}-\nabla^{\phi(0)}$ as shown in \cite{FT-lin-half, FT-lin-half2}.   

\vspace{1ex}

The paper is organized as follows. In Sect.~2 we introduce several function spaces, the Muckenhoupt class of $L^q$ weights, properties of the Helmholtz projection and the Stokes operator in weighted spaces, describe the coordinate transform applied to the Navier-Stokes system and reformulate the problem of time periodic solutions by an integral equation. 
Moreover, in Subsect.~\ref{S2.2}  we discuss properties of the Stokes semigroup on weighted spaces and cite results on the $\mathscr H^\infty$-calculus of the Stokes operator in weighted spaces, the property of bounded imaginary powers and further consequences, derived in \cite{Farwig-Tsuda-Hinfty}.  Further, Subsect.~\ref{S2.3} deals with the analysis of the domain of $A^\theta$ for small $\theta>0$ to prove Proposition \ref{result-oscillatory-part-2nd-derivative}. Finally Subsect.~\ref{S2.4}, \ref{S2.5} introduce the modified Navier-Stokes system and formulate a fixed point setting for the time-periodic solution. \BLACK  Sect.~3 is dedicated to the analysis of the stationary part $u_s$ in weighted $L^q$ and weak $L^q$ spaces. 
Next Sect.~4 concerns the analysis of the purely oscillatory part $u_\perp$ via semigroup theory and estimates of Yamazaki type in weighted Lorentz spaces. Based on  these tools, we estimate all nonlinear terms in Sect.~5.

\section{Preliminaries}

\subsection{Notations}\label{S2.1} 

Let $\Omega \subset \IR^n$ be either a bounded or an unbounded standard domain, {\em i.e.} either a whole or (perturbed) half space, or an exterior domain. Then 
$L^q(\Omega)$, $1\leq q \leq \infty$, denotes the usual Lebesgue space with norm $\|\cdot\|_{L^q(\Omega)} = \|\cdot\|_{L^q}$.
 Note that we frequently omit the symbol $\Omega$ in $L^q$, $\|\cdot\|_{L^q}$, {\em etc.}, of the underlying domain when it is known from the context. \BLACK 
In addition, $W^{k,q}(\Omega)$, $k\in\IN_0$, is the standard Sobolev space with norm $\|\cdot\|_{W^{k,q}(\Omega)}$.
We denote the spaces of test functions and of solenoidal test functions by 
$C^{\infty}_0(\Omega)$  and $C^{\infty}_{0,\sigma}(\Omega):=\{\varphi \in C^{\infty}_{0}(\Omega): \div \varphi=0\}$, respectively. Moreover,
$L^q_{\sigma}(\Omega):= \overline{C^{\infty}_{0,\sigma}(\Omega)}^{\|\cdot\|_{L^q(\Omega)}}$, $1<q<\infty$, is the $L^q$ space of weakly solenoidal vector fields, $u$, with vanishing normal component, $u\cdot \textsl{n}$, on $\partial\Om$.
The same notation $L^q(\Omega)$ {\em etc.} is used for both functions, vector fields and matrix-valued functions.

For a standard domain $\Omega\subset\IR^n$  with boundary of class $C^1$ we recall the Helmholtz projection, {\em i.e.,} the projection $\mathbb{P}_q\colon L^q(\Omega)\to L^q_{\sigma}(\Omega) \subset L^q(\Omega)$, $1<q<\infty$, such that the kernel of $\mathbb{P}_q$ equals the space $G_q(\Omega)$ of all weak gradient fields in $L^q(\Omega)$. We omit the index $q$ in the projection if no confusion will occur; actually, $\mathbb{P}_q u = \mathbb{P}_ru$ for all vector fields $u\in L^q(\Omega)\cap L^r(\Omega)$, $1<q,r<\infty$. Moreover, the adjoint of $\mathbb P_q$ coincides with the Helmholtz projection $\mathbb{P}_{q'}: L^{q'}(\Omega) \to L^{q'}(\Omega)$ where $\frac1q+\frac1{q'}=1$. 


\vspace*{1ex} 

Next we define the Muckenhoupt class as follows.

\begin{defi}
Let $1<q <\infty$. A weight function $0 \leq w \in L^1_{\rm loc}(\mathbb{R}^n)$ belongs to the Muckenhoupt class $\mathscr{A}_q(\IR^n) $ if $w$ satisfies 
$$
\sup_Q \bigg(\frac{1}{|Q|}\displaystyle\int_Q w \dx\bigg)\bigg(\frac{1}{|Q|}\displaystyle\int_Q w^{-1/(q-1)}\dx\bigg)^{q-1} \leq C <\infty
$$
for all cubes $Q \subset \mathbb{R}^n$, where $|Q|$ denotes the Lebesgue measure of $Q$.  
\end{defi}



Then the weighted $L^q$ space with Muckenhoupt weight $w \in \mathscr{A}_q$, $1<q < \infty$, is defined by  
\begin{equation}\label{mathbb-Lqw}
\mathbb{L}^q_w (\Omega) = \Bigg\{ u \in L^1_{{\rm loc}} (\bar{\Omega}): \|u\|_{\mathbb{L}^q_w (\Omega)}= \Bigg(\displaystyle\int_{\Omega} |u|^q w \dx\Bigg)^{1/q}<\infty\Bigg\}.  
\end{equation}
Furthermore, we denote weighted non-homogeneous and homogeneous Sobolev spaces by 
\begin{align*}
& \mathbb{W}^{k,q}_{w}(\Omega) = \{u \in \mathbb{L}^q_{w}(\Omega): \nabla^\alpha u\in \mathbb{L}^q_{w}(\Omega),\, 
|\alpha | \leq k\}, \\
& \widehat{\mathbb{W}}^{k,q}_{w}(\Omega) = \{u \in W^{k,1}_{\rm loc}(\Omega): \nabla^\alpha u\in \mathbb{L}^q_{w}(\Omega),\, 
|\alpha | = k\}
\end{align*}
with the norms 
\begin{align*}
\|u\|_{\mathbb{W}^{k,q}_{w}(\Omega)}= \Bigg(\sum_{|\alpha|\leq k} \|\nabla^\alpha u\|_{\mathbb{L}^q_{w}(\Omega)}^q\Bigg)^{1/q}, \\
\|u\|_{\widehat{\mathbb{W}}^{k,q}_{w}(\Omega)}= \Bigg(\sum_{|\alpha|= k} \|\nabla^\alpha u\|_{\mathbb{L}^q_{w}(\Omega)}^q\Bigg)^{1/q} 
\end{align*}
for $1<q < \infty$, $k \in \mathbb{N}_0$ and $w\in \mathscr{A}_q$, respectively; here $\alpha\in \IN_0^n$ is a multi-index, and $|\alpha|$ is the length of $\alpha$. 
Moreover, let $\mathbb{W}^{1,q}_{0,w}(\Omega)= \{u\in \mathbb{W}^{1,q}_{w}(\Omega): u|_{\partial\Omega}=0\}$; by standard arguments it can be proved that $C^\infty_0(\Omega)$ is dense in $\mathbb{W}^{1,q}_{0,w}(\Omega)$. \BLACK


\BLACK For weight functions of radially symmetric type 
$$ w(x) = \langle x\rangle^{\ell q} = \big(1+|x|^2\big)^{\ell q/2},\quad\ell\in\IR,$$ 
we introduce the weighted space 
$$
L^q_\ell (\Omega) := \bigg\{ u \in L^1_{\rm  loc} (\bar{\Omega}): \|u\|_{L^q_\ell (\Omega)}= \Big(\displaystyle\int_{\Omega} |u|^q \langle x\rangle^{\ell q} \dx\Big)^{1/q}<\infty\bigg\}  
$$
for $1<q < \infty$,  {\em i.e.,} 
we replace the index $w$ by $\ell$ and use the weight $\langle x\rangle^{\ell q}$ rather than $\langle x\rangle^{\ell}$ in the above integrand. 
Similarly, weighted Sobolev spaces $W^{k,p}_{\ell}$, $W^{k,p}_{0,\ell}$ are defined. 
Concerning function spaces of solenoidal vector fields we denote by $L^q_{\sigma,\ell}(\Omega)$  the closure of $C^\infty_{0,\sigma}(\Omega)$ with respect to the norm of $L^q_\ell(\Omega)$. 
Recall that $w =\langle x\rangle^{\ell q}\BLACK \in \mathscr A_q(\IR^n)$ if and only if 
$-\frac{n}{q} < \ell < n\big(1-\frac{1}{q}\big) = \frac{n}{q'}$;  see {\em e.g.} \cite[Lemma 2.3]{FS}.

Weighted $L^p$ spaces are immediately extended to weighted Lorentz spaces on $\Omega$. Let $1 < p < \infty $ and $1 \leq q\leq \infty$, and let $\ell\in\IR$ be an arbitrary weight exponent. Then  
$$ {L}^{p,q}_\ell(\Omega) = \{u \textrm{ measurable on }\Omega:  \langle x\rangle^\ell u \in L^{p.q}\} $$
is equipped with the norm $\|u\|_{{L}^{p,q}_\ell(\Omega)} = \|\langle x\rangle^\ell u\|_{L^{p,q}(\Omega)}$. 
For $1< q_0,q_1 <  \infty$ and $1 \leq r \leq \infty$ there holds the
interpolation identity 
\begin{equation}\label{Lps-interp}
\big(L^{q_0}_{\ell}(\Omega), L^{q_1}_{\ell}(\Omega) \big)_{\theta,r} = L^{q,r}_\ell(\Omega),\quad \frac1q = \frac{1-\theta}{q_0} + \frac{\theta}{q_1}, \;\;0<\theta<1,
\end{equation}
see \cite[Theorem 2]{Freitag}. 
Then the reiteration theorem of real interpolation \cite[Theorem 1.10.2]{Triebel} implies for any $1<p_0<p_1<\infty$ and $1\leq r_0,r_1\leq \infty$ that
\begin{equation}\label{Lps-reiterate}
\big(L^{p_0,r_0}_\ell(\Omega), L^{p_1,r_1}_\ell(\Omega)\big)_{\theta,q} = L^{p,q}_\ell(\Omega), \quad \frac{1}{p} = \frac{1-\theta}{p_0} + \frac{\theta}{p_1},\;\; 0<\theta<1.\end{equation}

A crucial tool for inequalities in radially  \BLACK weighted Sobolev spaces are the general Gagliardo-Nirenberg estimates as follows.

\vspace*{1ex}

\begin{prop}\label{DS} {\em (Duarte and Silva \cite[Theorem 1.7]{Du-Si})} 
Let $1 < p, q, r <  \infty$, $0\leq k_0 <k$, $\frac{k_0}{k}\leq \theta\leq 1$, and $\alpha\in (-\frac{n}{p}, \frac{n}{p'})$, $\beta\in (-\frac{n}{q}, \frac{n}{q'})$, $\gamma >-\frac{n}{r}$, 
satisfy  
\begin{equation}\label{theta-pqr-alpha-beta-N}
\theta\Big(\frac{1}{p}-\frac{k}{n}\Big) + (1-\theta)  \frac{1}{q} +\frac{\gamma}{n} \leq \frac{1}{r} -\frac{k_0-\gamma}{n} \leq  \theta\Big(\frac{1}{p}-\frac{k-\alpha}{n}\Big) + (1-\theta)\Big(\frac{1}{q}+\frac{\beta}{n}\Big)
\end{equation} 
and 
$$ \gamma \leq \theta \alpha + (1-\theta)\beta $$
with 
$$
\frac{1}{r} \leq \frac{\theta}{p}+\frac{1-\theta}{q}.
$$
Then there holds for any Schwartz function $f \in \mathcal S(\mathbb{R}^n)$ 
\begin{align}\label{GNI-weighted-N}  
\|(-\Delta)^{k_0/2} f \|_{L^r_\gamma} 
\leq C\|(-\Delta)^{k/2} f\|_{L^p_\alpha}^\theta \, \|f\|_{L^q_\beta}^{1-\theta}.
\end{align}   
\end{prop}

\vspace{1ex}

To apply Proposition \ref{DS} also to exterior domains we need special extension operators that take into account estimates in homogeneous function spaces. 
Chua \cite{Chua} proved the extension property for homogeneous weighted Sobolev spaces for general $(\varepsilon,\infty)$ domains, including smooth exterior domains as follows. Note that in  Lemma \ref{extention} below $w$ denotes an arbitrary Muckenhoupt weight in $\mathscr{A}_q$. \BLACK

\vspace{1ex}

\begin{lem}\label{extention} {\rm(\cite[Theorem 1.5]{Chua}, \cite[Theorem 2.2]{FroehII})}
Let $1<q <\infty$, $w\in \mathscr{A}_q$, and $k_1<\ldots<k_N\in\mathbb N_0$. Further let $\Omega\subset\IR^n$ be an exterior Lipschitz domain. \BLACK Then there exists a linear extension operator $E: \bigcap_{i=1}^N \widehat{\mathbb{W}}^{k_i,q}_{w}(\Omega)\rightarrow  \bigcap_{i=1}^N \widehat{\mathbb{W}}^{k_i,q}_{w}(\mathbb{R}^n)$ such that 
$$
\|\nabla^{k_i} Eu \|_{\mathbb{L}^q_w(\mathbb{R}^n)} \leq C_i \|\nabla^{k_i} u\|_{\mathbb{L}^q_w(\Omega)}
$$
for all $i=1, \ldots,N$ and $u \in \bigcap_{i=1}^N \widehat{\mathbb{W}}^{k_i,q}_{w}(\Omega)$. 
\end{lem}

\begin{lem}\label{Helmh}{\rm \cite[Theorem 1.3, Corollary 4.4]{FS}}
For an exterior domain $\Omega_0 \subset\IR^n$ with boundary of class $C^1$, all $1<q<\infty$ and weight exponents $-\frac{n}{q}<\ell<\frac{n}{q'}$ there exists the bounded Helmholtz projection 
$$ \mathbb P = \mathbb P_{q,\ell,\Omega_0}: L^q_\ell(\Omega_0) \to L^q_{\sigma,\ell}(\Omega_0) $$
with null space $\nabla \widehat H^{1,q}_{\ell}(\Omega_0)$ and range $L^q_{\sigma,\ell}(\Omega_0)$. The adjoint of $\mathbb P_{q,\ell}$ equals $\mathbb P_{q',-\ell}$ on $L^{q'}_{-\ell}(\Omega_0)$.

Similar results hold for the case of the whole space $\IR^n$. 
\end{lem}

In view of the unboundedness of domains  weighted \BLACK homogeneous function spaces are introduced with their density  and interpolation properties. For the whole space $\IR^n$ homogeneous Riesz potential spaces (Sobolev spaces) $\widehat{\mathbb{H}}^{\kappa,q}_w(\IR^n)$ are defined as distributions in the space $Z'(\IR^n) = \mathcal S'(\IR^n)/\Pi(\IR^n)$ where $\mathcal S'(\IR^n)$ denotes the space of Schwartz' tempered distributions and $\Pi(\IR^n)$ the set of all polynomials over $\IR^n$. Then  for any $\kappa\in\IR$, $1<q<\infty$ and $w\in\mathscr A_q$
$$
\widehat{\mathbb{H}}^{\kappa,q}_w(\IR^n) = \Big\{f\in Z'(\IR^n): \|f\|_{\widehat H^{\kappa,q}_w} = \Big(\int_{\IR^n} \big|\mathcal F^{-1}(|\xi|^\kappa \hat f(\xi))(x)\big|^q \,w(x) \dx \Big)^{1/q} <\infty\Big\}. 
$$
See \cite[Chapter 5]{Triebel2010} for details, where these spaces are introduced  - without weights - as the class $\dot F^{\kappa;q,2}(\IR^n)$. 
If $\kappa=k\in\IN$,  $\widehat{\mathbb{H}}^{k,q}_w(\IR^n)$ coincides with the classical homogeneous Sobolev space of functions $f$ such that - with equivalent norms - $\sum_{|\alpha|=k} \| \nabla^\alpha f\|_{q,w}<\infty$. 
The drawback of the spaces 
$\widehat{\mathbb{H}}^{\kappa,q}_w(\IR^n)$  
is the fact that elements are uniquely determined only up to polynomials. 
In view of interpolation of spaces $\mathbb L^r_v(\IR^n)$ and $\widehat{\mathbb{H}}^{\kappa,q}_w(\IR^n)$ we identify $\mathbb L^r_v(\IR^n)$, $1\leq r<\infty$, with $(\mathbb L^r_v(\IR^n) + \Pi(\IR^n))/\Pi(\IR^n)$.

For weighted \BLACK homogeneous spaces defined on an  exterior domain $\Omega \subset\IR^n$ similar results are stated.  Recall the definition
$$ \widehat{\mathbb{H}}^{k,q}_w(\Omega) := \big\{u: u=v\big|_{\Omega},\, v\in \widehat{\mathbb{H}}^{k,q}_w(\IR^n)\big\}, \quad 
\|u\|_{\widehat{\mathbb{H}}^{k,q}_w(\Omega)}: = \inf_{v,\, u=v|_{\Omega}} \|v\|_{\widehat{\mathbb{H}}^{k,q}_w(\IR^n)}, $$
for $k\in\IN$ and, in case of vanishing boundary values,
$$  \widehat{\mathbb{H}}^{k,q}_{w,0}(\Omega) := \overline{ C^\infty_0(\Omega)}^{\widehat{\mathbb H}^{k,q}_{w}(\Omega)}. $$
Moreover,  for $\kappa=k+\theta$ where $k\in \IN_0$ and $0<\theta<1$ let 
\begin{equation*}
\widehat{\mathbb H}^{\theta,q}_{w,0}(\Omega) := \big[\widehat{\mathbb H}^{k,q}_{w,0}(\Omega), \widehat{\mathbb H}^{k+1,q}_{w,0}(\Omega)\big]_\theta.
\end{equation*}  
Finally, for $k \in \mathbb{N}$, we define weighted \BLACK homogeneous $L^{q,\infty}$-Sobolev spaces by 
$$ \widehat{\mathbb{H}}^{k,q;\infty}_{w}(\Omega) = \{u \in H^{k,1}_{\rm loc}(\Omega): \nabla^\alpha u\in \mathbb{L}^{q,\infty}_{w}(\Omega),\, 
|\alpha | = k\} $$
%
%

Concerning spaces of solenoidal vector fields  $\mathbb L^q_{\sigma,w}(\Omega)$ and $L^q_{\sigma,\ell}(\Omega)$ are defined as the closure of $C^\infty_{0,\sigma}(\Omega)$ with respect to the norm of $\mathbb L^q_{w}(\Omega)$ and $L^q_\ell(\Omega)$, respectively.
The subspace of solenoidal vector fields in $\widehat{\mathbb{H}}^{1,q}_{w,0}(\Omega)$ is similarly introduced by
\begin{align}\label{spaces}
 \widehat{\mathbb{H}}^{1,q}_{\sigma,w,0}(\Omega) & :=  \overline{C^\infty_{0,\sigma}(\Omega)}^{\|\nabla \cdot\|_{L^q_w(\Omega)}},
\end{align} 
 and, by complex interpolation, for $0<\theta<1$, 
\begin{equation}\label{spaces2}
\widehat{\mathbb{H}}^{\theta,q}_{\sigma,w,0}(\Omega) := [\mathbb L^q_{\sigma,w}(\Omega), \widehat{\mathbb{H}}^{1,q}_{\sigma,w,0}(\Omega)]_\theta.
\end{equation}
Obviously there holds the identity
$$
\mathbb L^q_{\sigma,w}(\Omega) \cap \widehat{\mathbb H}^{1,q}_{\sigma,w,0}(\Omega)  = {\mathbb H}^{1,q}_{\sigma,w,0}(\Omega).$$
%
Homogeneous spaces with weights $\langle\cdot\rangle^{\ell  q}$ are defined analogously. For example 
$$\widehat H^{1,q}_{\sigma,\ell,0}(\Omega) := \overline{C^\infty_{0,\sigma}(\Omega)}^{\|\nabla\cdot\|_{L^q_\ell(\Omega)}}. $$

Finally, we define  weighted \BLACK inhomogeneous fractional Sobolev spaces of solenoidal vector fields for $0< \theta \leq 1$ by 
\begin{align}\label{inhomogeneous space}
H^{\theta,q}_{\sigma,\ell}(\Omega) := H^{\theta,q}_{\ell}(\Omega) \cap L^q_{\sigma,\ell}(\Omega).
\end{align}
ignoring the usual boundary condition $u=0$, but keeping $u\cdot \textsl{n} =0$ on $\partial\Omega$. Obviously, the standard Sobolev spaces of solenoidal vector fields vanishing on the boundary satisfy \BLACK $
H^{\theta,q}_{\sigma,\ell,0}(\Omega) = H^{\theta,q}_{\ell,0}(\Omega) \cap L^q_{\sigma,\ell}(\Omega)$.

\begin{prop}{\rm \cite[Proposition 2.4]{Farwig-Tsuda-Hinfty}}\label{dense Htheta}
Let $\Omega\subset \IR^n$ be an exterior domain as above. Then, for each $1<q<\infty$, $0<\theta<1$ and $w\in\mathscr A(\IR^n)$,  the set $C^\infty_{0,\sigma}(\Omega)$ is dense in the interpolation space $\big[\mathbb L^q_{\sigma,w}(\Omega), \widehat{\mathbb{H}}^{1,q}_{\sigma,w,0}(\Omega)\big]_\theta$, {\em i.e.} by definition in $\widehat{\mathbb{H}}^{\theta,q}_{\sigma,w,0}(\Omega)$. 

By analogy, $C^\infty_{0}(\Omega)$ is dense in $\big[\mathbb L^q_{w}(\Omega), \widehat{\mathbb{H}}^{1,q}_{w,0}(\Omega)\big]_\theta = \widehat{\mathbb{H}}^{\theta,q}_{w,0}(\Omega)$.  
\end{prop}

%
\begin{lem}{\rm \cite[Lemma 2.5]{Farwig-Tsuda-Hinfty}}\label{hatH0-dense}
Let $1<q<\infty$, $w\in\mathscr  A_q(\IR^n)$.

(i) The set $C^\infty_{0}(\IR^n)$ is dense in $\widehat{\mathbb{H}}^{1,q}_{w}(\IR^n)$.

(ii) The set $\Delta C_0^\infty(\IR^n)$ is dense in $\mathbb L^q_w(\IR^n)$.

(iii) The set $C^\infty_{0}(\IR^n)$ is dense in $\widehat{\mathbb{H}}^{\kappa,q}_{w}(\IR^n)$ for each $0\leq\kappa\leq 2$.  
\end{lem}

\vspace*{1ex}

Finally, we define weighted Sobolev-Lorentz spaces $\widehat{\mathbb{H}}^{\kappa;p,q}_w(\Omega)$, $\kappa>0$, via real interpolation of homogeneous spaces $\widehat{\mathbb{H}}^{k,p}_w(\Omega)$. 
\vspace*{2ex}

For a time-periodic function, $u$, with period $T>0$ we introduce the splitting 
$$ u =  u_s+u_\perp $$
where 
\begin{equation}\label{u_s(t)} 
u_s = \frac{1}{T} \int_0^T  u(\tau)\dtau
\end{equation}
is the {\em steady-state part} so that for the {\em purely oscillatory part} $u_\perp = u-u_s$ the integral mean
$\frac{1}{T} \int_0^T u_\perp(\tau)\dtau $ vanishes. 
We define projections, $P_s$ and $P_\perp$, as 
\begin{equation}\label{PsPperp}
P_s u := u_s = \int_0^T \hspace{-7mm} -\hspace{2mm} u(\tau)\dtau=\frac{1}{T} \int_0^T  u(\tau)\dtau,  \ \  P_\perp u := u_\perp = (I-P_s)u, 
\end{equation}
{\em cf.} Celik and Kyed \cite{Celik-K}. Note that  
$\div u_s=0$, if $\div u=0$, 
and thus $\div u_\perp= \div u-\div u_s=0$. 

\vspace{1ex}

Given $q>n$ we consider the function spaces
\begin{align*}
X_s = & \Big\{(u,p)=(P_s u,P_s p) \in W^{1,\infty}_{\rm loc} \cap H^{1,q} \cap \widehat{H}^{2;q,\infty} \times \widehat{H}^{1;q,\infty} :  \div u=0 \textrm{ in } \Omega_0,\; u|_{\partial\Omega_0}=0,\\
& \quad \|(u,p)\|_{X_s}:=\|u\|_{  L^\infty_{n-2} } + \| \nabla u\|_{L^\infty_{n-1}}+ \|u\|_{H^{1,q}} +\|\nabla^2 u\|_{L^{q,\infty}}
+\|\nabla p\|_{L^{q,\infty}}< \infty  \Big\}
\end{align*}
and 
\begin{align*}
X_\perp = & \Big\{ (u,p)=(P_\perp u, P_\perp p) \in L^\infty_{per}(\mathbb{R};  W^{1,q}_{n-1}\BLACK (\Omega_0)) \times L^\infty_{per} (\mathbb{R};\widehat{H}^{1;q,\infty}(\Omega_0)): \\ 
& \quad 
\nabla^2 u \in  L^\infty_{per} (\mathbb{R};L^{q,\infty}_{n-1\BLACK}(\Omega_0)),\; \div u=0,\, u=0\textrm{ on } \partial\Omega_0,\BLACK\\
& \quad \|(u,p)\|_{X_\perp}:= \|u\|_{ L^\infty_{per} (\mathbb{R}; W^{1,q}_{n-1} \BLACK )} 
+ \|\nabla^2 u\|_{L^\infty_{per}(\mathbb{R}; L^{q,\infty}_{ n-1})} + \|\nabla p\|_{L^\infty_{per} (\mathbb{R}; L^{q,\infty})}\BLACK  \Big\}.
\end{align*} 

\subsection{Properties of the Stokes semigroup on weighted spaces}\label{S2.2}

Given the Helmholtz projection $\mathbb P_{q,w}$, the Stokes operator $A_{q,w}$ is defined by 
$$
A_{q,w}=-\mathbb P_{q,w}\Delta_{\Omega}\colon \D(A_{q,w}) = \mathbb{W}^{2,q}_w(\Omega)\cap \mathbb{W}^{1,q}_{0,w}(\Omega)\cap \mathbb{L}^q_{\sigma,w}(\Omega)\subset \mathbb{L}^q_{\sigma,w}(\Omega)\to \mathbb{L}^q_{\sigma,w}(\Omega).
$$
Since $A_{q,w}u = A_{q,\tilde w}u$ for $u\in \D(A_{q,w}) \cap \D(A_{q,\tilde w})$, a dense subset of $\D(A_q)$, we will also write $A$ for $A_{q,w};$ here $w,\tilde w$ are any weights in $\mathscr{A}_q$. 
Further, we introduce the sector $ \Sigma_\omega = \{\lambda\in\IC: |\arg \lambda| < \omega, \lambda\neq 0\}$ for any $0<\omega<\pi$.

\begin{thm}\label{res-weighted} {\em \cite[Theorems 1.5 and 5.5]{FS}}
Let $n \geq 3$, $\Omega \subset \mathbb{R}^n$ be an exterior domain with boundary of class $C^2$, let $q \in (1, \infty)$ and $w\in \mathscr{A}_q$. 
\begin{enumerate}
\item[{\rm (i)}]
For $\lambda\in\Sigma_\omega$, $0<\omega<\pi$, satisfying $|\lambda|\geq\delta>0$ the Stokes resolvent problem
$\lambda u + A_{w,q} u =f$ 
has a unique solution $u\in {\D}(A_{w,q})$ satisfying the resolvent estimate
\begin{align}\label{equ:rse-w}
	\|\lambda u\|_{\mathbb{L}^q_w(\Omega)}  + \|A_{q,w} u\|_{\mathbb{L}^q_w(\Omega)}  \leq c_{\omega,\delta} \|f\|_{\mathbb{L}^q_w(\Omega)}.
\end{align}
\item[{\rm (ii)}]
If $w(x) = \langle x\rangle^\alpha$  
with $-n < \alpha < n(q-1)$ and $n\geq 3$ 
\BLACK, {\em i.e.} $w\in \mathscr{A}_q$, the constant $c_{\omega,\delta}$ can be replaced by $c_{\omega}$, a constant independent of $\delta$. 
\item[{\rm (iii)}]
Let 
$ w^{s/q} \langle x\rangle^{-\gamma s} \in  \mathscr{A}_s,$ 
where 
$\gamma=n\big(\frac{2}{n}+\frac{1}{s}-\frac{1}{q}\big) \geq 0 $, $n\geq 3$ 
\BLACK and $s \geq q$. Then 
the resolvent estimate 
\begin{align}\label{equ:rse-w2}
	\|\lambda u\|_{\mathbb{L}^q_w(\Omega)}  +  |\lambda|^\frac12 \|\nabla u\|_{\mathbb{L}^q_w(\Omega)}  +\|\nabla^2 u\|_{\mathbb{L}^q_w(\Omega)}  \leq c \|(\lambda+A_{q,w})u\|_{\mathbb{L}^q_w(\Omega)}
\end{align}
holds for $u\in\mathcal D(A_{w,q})$ and all $\lambda\in \Sigma_\omega$ uniformly.  
In particular, if  $w=\langle x\rangle^\alpha$ \BLACK and
$$
n\geq 3,  \ \ 2q-n < \alpha < n(q-1), $$
then \eqref{equ:rse-w2} is satisfied. 
\end{enumerate}
\end{thm}

\vspace{1ex}




We recall weighted $L^p$-$L^q$ estimates on a smooth exterior domain $\Omega$ from  
\cite[Proposition 2.11, Theorem 2.10]{Farwig-Tsuda-stab}: 

\begin{thm}\label{weightedLpLq^decay-estimates}
Let $\Omega \subset\IR^n$, $n \geq 3$, be an exterior domain with boundary of class $C^2$, and let $1< p \leq q <\infty$.

(i) Assume $2-\frac{n}{p} < \ell < \frac{n}{p'}$. 
Then for $u_0 \in L^p_{\sigma,\ell}(\Omega)$ and any multi-index $\alpha\in\IN_0$ with $|\alpha|=0,1$  such that $\frac{n}{2}\big(\frac{1}{p}-\frac{1}{q}\big)+\frac{|\alpha|}{2}<1$ it holds for $0<t<\infty$ that 
\begin{equation}\label{etA<22}
\big\|\nabla^{\alpha} e^{-tA} u_0 \big\|_{L^q_{\ell}(\Omega) } \leq C t^{-\frac{n}{2}\big(\frac{1}{p}-\frac{1}{q}\big)-\frac{|\alpha|}{2}}\|u_0\|_{L^p_\ell(\Omega)}.
\end{equation}

(ii) Assume $-\frac{n}{q} \leq \ell_0 \leq \ell < \frac{n}{p'}$ and $\ell\geq 0$.   
Then for $u_0 \in L^p_{\sigma,\ell}(\Omega)$ the following estimates hold: For $t>1$, 
\begin{equation}\label{etA>1}
\big\|e^{-t A} u_0\big\|_{L^q_{\ell_0}(\Omega) } \leq C t^{-\frac{n}{2}\big(\frac{1}{p}-\frac{1}{q}\big) - \frac{\ell-\ell_0}{2} } \|u_0\|_{L^p_\ell(\Omega) }. \end{equation}
If in addition 
$0\leq \ell_0$, we have that 
\begin{equation}\label{etA>1nabla}
\big\|\nabla e^{-t A} u_0\big\|_{L^q_{\ell_0}(\Omega) } \leq C \big(t^{-\frac{n}{2}\big(\frac{1}{p}-\frac{1}{q}\big)-\frac{1}{2} - \frac{\ell-\ell_0}{2} } +t^{-\frac{n}{2p}}t^{-\frac{\ell}{2}}\big) \|u_0\|_{L^p_\ell(\Omega) }. 
\end{equation}
For $0<t<2$ and any $\alpha\in\IN_0$, $|\alpha|=0,1$  such that $\frac{n}{2}\big(\frac{1}{p}-\frac{1}{q}\big)+\frac{|\alpha|}{2}<1$ it holds that 
\begin{equation}\label{etA<}
\big\|\nabla^{\alpha} e^{-tA} u_0 \big\|_{L^q_{\ell_0}(\Omega) } \leq C t^{-\frac{n}{2}\big(\frac{1}{p}-\frac{1}{q}\big)-\frac{|\alpha|}{2}}\|u_0\|_{L^p_\ell(\Omega)}.   
\end{equation} 
\end{thm}

\vspace{1ex}

\begin{rem}\label{constant}
{\rm Fixing $t_0>0$, we can replace 
$t>1$ by $t>t_0$ and 
$0<t<2$ by $t<t_0$ in \eqref{etA>1}, \eqref{etA>1nabla} 
and \eqref{etA<}, respectively, 
with constants $C_{t_0}$. 
}
\end{rem}


\begin{rem}\label{why-we-need-BIP}
{\rm 
It is difficult to ignore the restriction $0\leq \ell$ in (ii) because it comes from the local decay estimate, 
\begin{align*}
  \|  e^{-tA} u_0\BLACK \|_{L^{p}_{\ell_0}(\Omega_{R})} \leq C \| e^{-tA} u_0\BLACK\|_{L^{p}(\Omega_{R})} 
\leq C t^{-\frac{n}{2p}-\frac{s}{2}} \|u_0\|_{ L^p(\Omega) } \leq Ct^{-\frac{n}{2p}-\frac{s}{2}} \|u_0\|_{ L^p_\ell(\Omega_0)}
\end{align*}
on a bounded domain $\Omega_{R}=\Omega \cap B_R$, see \cite[Proposition 3.5]{Farwig-Tsuda-exterior}. 
\ On the other hand, in (ii) the condition $2-\frac{n}{p} < \ell$ from (i) is relaxed to $-\frac{n}{q}<\ell_0$.  
}
\end{rem}

\vspace{1ex}

To handle radial weights with negative exponents in $L^p$-$L^q$ decay estimate of the Stokes semigroup, the following estimate plays an important role. 

\vspace{1ex}

\begin{prop}\label{Lp-Lq At-negative-weight} 
Let $1<p \leq q<\infty$ with $3 \leq n$ and $- \frac{n}{q}
< \ell < \frac{n}{p'}-2$ with $\ell \leq 0$. Further assume that
 $\frac{1}{q} = \frac{1}{p} - \frac{2\theta}{n}$ where $0\leq\theta\leq 1$. 
Then for $u\in L^p_{\sigma,\ell}(\Omega)$ there holds the $L^p$-$L^q$-estimate 
$$ \|e^{- tA}u\|_{L^q_{\ell }(\Omega)} \leq C t^{-\frac{n}{2}\big(\frac{1}{p}-\frac{1}{q}\big)} \|u\|_{L^p_{\ell }(\Omega)}$$
with a constant $C>0$ independent of $t>0$.
\end{prop}

\vspace{1ex}

\begin{proof} 
We make use of a duality argument and take $\varphi \in C^{\infty}_{0,\sigma}(\Omega)$. By Theorem \ref{weightedLpLq^decay-estimates} (i), applied to $\|e^{-tA} \varphi\|_{L^{p'}_{-\ell}}$ there holds the estimate
\begin{align*}
\big|(e^{-tA} u_0, \varphi)\big| & = \big|( u_0, e^{-tA} \varphi)\big| \leq C\|u_0 \|_{L^{p}_\ell} \|e^{-tA} \varphi\|_{L^{p'}_{-\ell}} \\
&\leq C \|u_0 \|_{L^{p}_\ell}\; t^{-\delta} \,\|\varphi\|_{L^{q'}_{-\ell}},  
\end{align*}
provided that $-\ell\geq 0$ and $2-\frac{n}{q'} < -\ell < \frac{n}{(q')'} = \frac{n}{q}$. These conditions are equivalent to $\ell\leq 0$ and 
$-\frac{n}{q} < \ell < \frac{n}{q'} -2$.
\end{proof}



\vspace{2ex}

Then, by Proposition \ref{Lp-Lq At-negative-weight}, 
we have the following $L^p$-$L^q$ decay estimate of the Stokes semigroup in weighted spaces with negative exponent. 

\vspace{1ex}

\begin{prop}
Let $1<p< \infty$ with $3 \leq n$ and $-\frac{n}{q} < \ell < \frac{n}{p'}-2$ with $\ell \leq 0$.
Let $1<p \leq q<\infty$, and assume $\delta := \frac{n}{2}\big(\frac{1}{p}-\frac{1}{q}\big) \in[0,1]$. 
Then it holds that 
\begin{equation}\label{A-1-delta-0}
\|A e^{-tA}u\|_{L^{q}_\ell(\Omega)} \leq Ct^{-1-\delta} \|u\|_{L^{p}_\ell(\Omega)} 
\end{equation}
In addition, for $t>0$, 
\begin{equation}\label{A-1-delta}
\|A e^{-tA}u\|_{L^{q,1}_\ell(\Omega)} \leq Ct^{-1-\delta} \|u\|_{L^{p,1}_\ell(\Omega)}. 
\end{equation}
\end{prop}

\vspace{1ex}

\BLACK 

\begin{proof}
We recall from \eqref{equ:rse-w} for $u=(\lambda+A)^{-1}f$, $f\in L^q_{\sigma,\ell}(\Omega)$, $\lambda\in \Sigma_\omega$, the uniform resolvent estimate
$$
|\lambda|\|u\|_{L^{q}_\ell(\Omega)} + \|A u\|_{L^q_\ell } \leq C\|f\|_{L^q_\ell }
$$
for $1<q<\infty$ with $-\frac{n}{q}< \ell < \frac{n}{q'}$. Hence 
\begin{align}\label{Ae-est}
\|A e^{-tA}u\|_{L^q_\ell } \leq Ct^{-1}\|u\|_{L^q_\ell },
\end{align}
which easily follows from the representation of the Stokes  semigroup via Dunford's calculus and the Stokes resolvent.  

In addition, the $L^p_\ell$-$L^q_\ell$ decay estimate of the semigroup, Proposition \ref{Lp-Lq At-negative-weight}, holds for $1< p \leq q <\infty $ with $-\frac{n}{q} < \ell < \frac{n}{p'}-2$,  $\ell\leq 0$. 
Therefore, the semigroup property implies \eqref{A-1-delta-0}.
This together with real interpolation derives \eqref{A-1-delta} as well.
\end{proof}

 For more sophisticated estimates of the Stokes operator, its embeddings  and decay results of its semigroup we need that $A_{q,\ell}$ has a bounded $\mathscr H^\infty$-calculus and hence possesses the property of bounded imaginary powers $BIP$. We refer to \cite{Farwig-Tsuda-Hinfty} for precise details and proofs; recall that $\HH_0(\Sigma_\phi)$ denotes a space of bounded holomorphic functions on the sector $\Sigma_\phi$ with a concrete control of functions as $\lambda \to0$ and $|\lambda|\to\infty$. \BLACK Note that similar results hold for $A_{q,\ell}$ on the whole space; in that case, some proofs are much easier. 
 The focus of \cite{Farwig-Tsuda-Hinfty} is put on homogeneous spaces. Actually, the proofs for nonhomogenous spaces, {\em e.g.,} of Theorem \ref{theorem-H-infty-all} (iii)  which is not mentioned in \cite{Farwig-Tsuda-Hinfty}, \BLACK are much shorter than in the homogeneous case.

\begin{thm}\label{theorem-H-infty-all}
{\rm \cite[Theorem 1.1]{Farwig-Tsuda-Hinfty}}
Let $\Omega\subset\IR^n$, $n\geq 3$, be an exterior domain with boundary of class $C^3$. 
Further let $1<q< \infty$, and $A=A_{q,\ell} = -\mathbb P \Delta$ be the Stokes operator on $L^q_{\sigma, \ell }(\Omega)$ for $-\frac{n}{q} < \ell < \frac{n}{q'}$.

(i) There exists a constant $C=C(q,\ell,\phi)>0$ such that for any $\phi\in (0,\frac{\pi}{2})$, $f \in L^q_{\sigma, \ell }(\Omega)$ and all $h\in \HH_0(\Sigma_\phi)$
 \begin{align}\label{H-infty-all}
\|h(A)\|_{\mathcal{L}(L^q_{\sigma,\ell}(\Omega))} \leq C |h|_{\infty,\phi\BLACK}.    
\end{align}
In particular, $A$ possesses a bounded $\mathscr H^\infty$-calculus on $L^q_{\sigma.\ell}(\Omega)$ with $\mathscr H^\infty$-angle $\Phi_ {A}^\infty=0$. 

(ii) The Stokes operator has the property $BIP$ with power angle $\Theta_A=0$. Thus for any $\theta>0$ there exists $C=C(\theta)$ suoch that 
$$ \|A^{it}\|_{\mathcal L(L^q_{\sigma,\ell}(\Omega))} \leq C e^{\theta|t|},\quad t\in\IR. $$

(iii)
There holds for $0<\theta<1$ the identity 
$$ \D(A^{\theta}) = [L^q_{\sigma,\ell}(\Omega), \D(A)]_\theta. $$
\end{thm}

\begin{thm}\label{A1/2-nabla_A1/2-all}
{\rm \cite[Theorem 1.2]{Farwig-Tsuda-Hinfty}}
(i) For $1<q<\infty$ and $-\frac{n}{q}<\ell<\frac{n}{q'}$ there holds
\begin{equation}\label{A1/2-nabla-all}
    \big\|A_{q,\ell}^{1/2}\big\|_{L^q_\ell(\Omega)} \leq C\|\nabla u \|_{L^q_\ell(\Omega)}, \quad u\in \widehat H^{1,q}_{\sigma,\ell}(\Omega),
\end{equation}
with a constant $C>0$ independent of $u$.

(ii) Let $1<q<\infty$ and $1-\frac{n}{q} <\ell< \frac{n}{q'}$. Then there exists $C>0$ independent of $u$ such that
\begin{equation}\label{nabla-A1/2-all}
    \|\nabla u \|_{L^q_\ell(\Omega)} \leq C\big\|A_{q,\ell}^{1/2} u\big\|_{L^q_\ell(\Omega)}, \quad u\in \widehat\D(A_{q,\ell}^{1/2}), 
    \end{equation}
 where 
$\widehat{\D}(A^{1/2}_{q,\ell}$ denotes the completion of $\D(A^{1/2}_{q,\ell})$  with respect to $\|A^{1/2}_{q,\ell}\,\cdot\|_{L^q_\ell(\Omega)}.$
\end{thm}

 As applications we mention the $L^q_\ell$-$L^r_{\ell'}$ decay of the Stokes semigroup $e^{- tA_{q,\ell}}$ and the $L^p$-maximal regularity of $A_{q,\ell}$.

\begin{cor}\label{Lp-Lq At-all} 
{\rm \cite[Corollary 1.4]{Farwig-Tsuda-Hinfty} }
Let $1< q\leq r<\infty$, let $\kappa-\frac{n}{q} < \ell < \frac{n}{q'}$  with  $\kappa\in(0,n)$, and assume that
$$ \ell-\ell' = \kappa + \frac{n}{r} -\frac{n}{q}\geq 0. $$
Then there holds the weighted $L^q$-$L^r$-estimate
$$ \|e^{- tA_{q,\ell}}u_0\|_{L^r_{\ell'}(\Omega)} \leq C t^{-\frac{n}{2}\big(\frac{1}{q}-\frac{1}{r}\big) -\frac{\ell-\ell'}{2}} \|u_0\|_{L^q_{\ell}(\Omega)}, \quad u_0\in L^q_{\sigma,\ell }(\Omega),$$
with a constant $C>0$ independent of $t>0$. 
\end{cor}




\subsection{The Stokes operator $A^\alpha$ for small $\alpha\geq 0$ \label{S2.3}}
In this subsection we will prove that in contrast to the identity $\widehat\D(A^{1/2}_{q,\ell}) = \hat H^{1,q}_{\ell,0}(\Omega)$, the closure of $C^\infty_{0,\sigma}(\Omega)$, see Theorem \ref{A1/2-nabla_A1/2-all}, and a similar result for nonhomogeneous spaces, the domain $\D(A_{q,\ell}^{\alpha})$ for $0<\alpha<\frac{1}{2q}$ does not require any boundary condition except for $u\cdot \textsl{n}=0$ on $\partial\Omega$.  The main idea is an analogous result for $H^{\alpha,q}_{\ell}$ and a retraction-coretraction argument for $\D(A)$  in the nonhomogeneous setting. \BLACK

\BLACK 

\begin{prop}\label{H-alpha-Omega-IRn} 
Let $\Omega$ be an exterior domain and $\omega\subset\IR^n$ denote any bounded domain of boundary class $C^1$ or $C^2$, respectively. Finally, let $1<q<\infty$, $0<\alpha<1$ and $-\frac{n}{q} <\ell < \frac{n}{q'}$. 

(i) 
There hold the complex interpolations $[L^q_\ell(\Omega), H^{k,q}_\ell(\Omega)]_\alpha = H^{k\alpha,q}_{\ell}(\Omega)$, $k=1,2$.
\begin{align*}
[L^q_\ell(\Omega), H^{1,q}_\ell(\Omega)]_\alpha = H^{\alpha,q}_{\ell}(\Omega),\quad &  
 [L^q_\ell(\Omega), H^{2,q}_\ell(\Omega)]_\alpha  = H^{2\alpha,q}_{\ell}(\Omega),\\
 [L^q_\ell(\omega), H^{1,q}_\ell(\omega)]_\alpha  = H^{\alpha,q}_{\ell}(\omega),\quad & 
 [L^q_\ell(\omega), H^{2,q}_\ell(\omega)]_\alpha = H^{2\alpha,q}_{\ell}(\omega),
\end{align*}
respectively.

(ii)
The family of spaces $H^{\alpha,q}_{\ell}(\omega)$, $0\leq \alpha\leq 1$, is independent of the weight exponent $\ell$ and has - uniformly with respect to $\ell$ - equivalent norms $\|\cdot\|_{H^{\alpha,q}_{\ell}(\omega)}$.
\end{prop}

\begin{proof}
(i) 
 We construct \BLACK retractions $R: H^{k,q}_\ell(\IR^n) \to H^{k,q}_\ell(\Omega)$ and $R: H^{k,q}_\ell(\IR^n) \to H^{k,q}_\ell(\omega)$, $k=1,2$. Indeed, $R$ is defined by the trivial restriction for functions on $\IR^n$ to $\Omega$ and $\omega$, respectively. For the construction of corresponding coretractions $S$ we follow \cite[Subsect. \!4.2.2]{Triebel} and will consider only the case of a bounded domain $\omega$ when $k=1$. 
By a partition of unity \BLACK the problem is decomposed into a problem on a thin boundary layer covered by finitely many balls and further finitely many balls covering the remaining complement in $\omega$. 
For a typical open ball covering a part of $\partial\omega$ we flatten the boundary part by a coordinate transform and transfer the problem into a problem on the half space $\IR^n_+$. 
Given any function $f\in H^{1,q}_\ell(\IR^n_+)$ we use a standard extension method for $f$ to a function on $\IR^n$ vanishing for $x_n\leq -\epsilon$ for a sufficiently small $\epsilon>0$, together with the backward coordinate transform. 
The composition of these operators defines a local coretraction. Following the partition of unity and adding all local coretractions \BLACK we get a coretraction from $H^{1,q}_\ell(\omega)$ to $H^{1,q}_\ell(\IR^n)$ related to $R$. 

Now the method of retraction-coretraction proves (i).

(ii) A standard argument, see \cite[Subsect.\! 4.2.2, Step 2 of Lemma 4.2.2]{Triebel}, shows that the norm of $u\in H^{1,q}_\ell(\omega)$ is equivalent to the norm 
 defined by \BLACK the infimum of norms of extensions $U \in H^{1,q}_\ell(\IR^n)$ having support in $\omega^\epsilon = \omega\cup\{x: {\rm dist}(x,\partial\omega) < \epsilon\}$ where $\epsilon>0$ is chosen sufficiently small. 
Indeed, it suffices to consider a multiplication operator on $H^{1,q}_\ell(\IR^n)$ defined by an adequate cut off function supported in $\omega^\epsilon$. By complex interpolation and (i) this property holds for $H^{\alpha,q}_\ell(\omega)$ as well. 
Since the weight function $w(x)=\langle x\rangle^{q\ell}$ is almost constant on $\omega^\epsilon$, the weight plays no role to characterize the space $H^{1,q}_\ell(\omega)$ and its norm.
\end{proof}

\begin{lem}\label{commutator estimate} 
Let  $u \in H^{\alpha, q}_{\ell}(\Omega)$ with $\alpha \in (0,1)$ and $-\frac{n}{q} < \ell < \frac{n}{q'}$. Further let $\eta $ be a smooth cut off function such that 
$\eta(x)=1$ for $x\in B_1$ and $\supp \eta\subset B_{2}$, and define
$\eta_R(x) = \eta(x/R)$ when $R>1$. 
Then it holds that 
$$
\eta_R u \rightarrow u \mbox{ in }H^{\alpha, q}_\ell(\Omega) \quad \textrm{ as } R\rightarrow \infty.
$$ 
\end{lem}

\begin{proof} 
Given $u\in H^{\alpha,q}_{\ell}(\Omega)$ and $\epsilon>0$ we find by \cite[Theorem 1.9.3 (c)]{Triebel} and Proposition \ref{H-alpha-Omega-IRn} (i) an element
$u_1\in H^{1,q}_{\ell}(\Omega)$ such that 
$$ \|u-u_1\|_{H^{\alpha,q}_{\ell}(\Omega)} < \frac{\epsilon}{3}. $$
Working in $H^{1,q}_{\ell}(\Omega)$ with the equivalent norm $\|u_1\|_{L^q_\ell(\Omega)} + \|\nabla u_1\|_{L^q_\ell(\Omega)}$, see \eqref{equivnorm} below, it is immediate to see that there exists $R>0$ such that 
$$ \|\eta_R u_1 - u_1\|_{H^{1,q}_{\ell}(\Omega)} <\frac{\epsilon}{3}. $$
Thus
$$ \|u-\eta_R u\|_{H^{\alpha,q}_{\ell}(\Omega)} \leq \|u-u_1\|_{H^{\alpha,q}_{\ell}(\Omega)} + \|u_1-\eta_R u_1\|_{H^{\alpha,q}_{\ell}(\Omega)} + \|\eta_R(u_1-u)\|_{H^{\alpha,q}_{\ell}(\Omega)}. $$
Since the multiplication by $\eta_R$ is a bounded operator on both $L^q_\ell$ and $H^{1,q}_\ell$ and hence, by complex interpolation, also on $H^{\alpha,q}_{\ell}(\Omega)$, uniformly with respect to $R$, the last term in the above inequality can be estimated 
by a constant $C(\eta)$ times $\|u_1-u\|_{H^{\alpha,q}_{\ell}(\Omega)}<\frac{\epsilon}{3}$.
\end{proof}
\BLACK


\begin{prop}\label{H-alpha-0-0} 
Let $\Omega$ be a smooth exterior domain with 
$\Omega^c \subset B_{R}$. Then for $1<q<\infty$, $ -\frac{n}{q} < \ell<\frac{n}{q'}$ and $0<\alpha < \frac{1}{2q}$ the space $C^\infty_0(\Omega)$ is dense in $H^{\alpha, q}_{\ell}(\Omega)$, {\em i.e.,}
$$
H^{\alpha,q}_{\ell, 0}(\Omega) = H^{\alpha, q}_{\ell}(\Omega). 
$$
\end{prop}

\begin{proof}
Let $\eta_R $ be a smooth cut off function as in Lemma \ref{commutator estimate} satisfying $\eta_R(x)=1$ for $x\in B_R$,  ${\rm supp}\,\eta_R \subset B_{2R}$. 
Given $u \in H^{\alpha, q}_{\ell}(\Omega)$ we set  $u_R:= \eta_R u \in H^{\alpha, q}_{\ell}(\Omega)$. 
Then 
$$
\|u_R-u\|_{H^{\alpha, q}_{\ell}(\Omega)} = \|\eta_R u- u \BLACK\|_{H^{\alpha, q}_\ell(\Omega)} \rightarrow 0 
$$
as $R \rightarrow \infty$ by Lemma \ref{commutator estimate}.

For the next step we need that for the domains $\omega=\Omega_{2R}:=\Omega \cap B_{2R}$ and $\omega=\Omega$ 
the spaces $H^{1,q}_\ell(\omega)$, $ -\frac{n}{q} < \ell < \frac{n}{q'}$, can be equipped with the equivalent norms 
\begin{equation}\label{equivnorm}
\|u\|_{H^{1,q}_\ell(\omega)} \sim \|u\|_{L^q_\ell(\omega)} + \|\nabla u\|_{L^q_\ell(\omega)}. 
\end{equation}
Indeed, the ideas to prove \eqref{equivnorm} by \cite[Subsect. 4.2.3, Theorem 4.2.4]{Triebel} for $\ell=0$  can easily be extended to the case $\ell\neq 0$. In particular, by Proposition \ref{H-alpha-Omega-IRn} (ii)  
$H^{\alpha,q}(\Omega_{2R}) = H^{\alpha,q}_\ell(\Omega_{2R})$ with equivalent norms for $0<\alpha<1$.
   
Moreover, by the same idea, for any function $v_R\in H^{1,q}_\ell(\Omega)$ with compact support in $\overline{\Omega}_{2R}$ there holds $\|v_R\|_{L^q_\ell(\Omega)} = \|v_R\|_{L^q_\ell(\Omega_{2R})}$, $\|v_R\|_{H^{1,q}_\ell(\Omega)} = \|v_R\|_{H^{1,q}_\ell(\Omega_{2R})}$ 
and hence by the  exactness of \BLACK complex interpolation 
\begin{equation}\label{Halpha-R-ell}
\|v_R\|_{H^{\alpha,q}_\ell(\Omega)}   =  \|v_R\|_{H^{\alpha,q}_\ell(\Omega_{2R})},\quad \supp v_R\subset \overline{\Omega}_{2R}.
\end{equation}
\BLACK  

By virtue of \cite[Theorem 4.3.2]{Triebel} we see that 
$$
H^{\alpha,q}_{0}(\Omega_{2R})=H^{\alpha,q}(\Omega_{2R}) 
$$
and $C^\infty_0 (\Omega_{2R})$ is dense in $H^{\alpha,q}(\Omega_{2R}) = H^{\alpha,q}_\ell(\Omega_{2R})$. Hence for any sufficiently large $R>0$, there exists $\varphi_{R,k} \in C^\infty_0 (\Omega_{2R}) \subset C^\infty_0 (\Omega)$ such that 
\begin{equation}\label{Halpha-R-ell2}
\|\varphi_{R,k}- u_R \|_{H^{\alpha,q}_\ell(\Omega_{2R})} \sim \|\BLACK\varphi_{R,k}- u_R \|_{H^{\alpha,q}(\Omega_{2R})}  \rightarrow 0 \;\; {\rm as }\; k\to\infty. 
\end{equation}
In the estimate
$$
\|u-\varphi_{R,k}\|_{H^{\alpha,q}_{\ell}(\Omega)} \leq 
\|u-u_R\|_{H^{\alpha,q}_{\ell}(\Omega)} + \|\varphi_{R,k}-u_R\|_{H^{\alpha,q}_{\ell}(\Omega)} 
$$
the first term on the right hand side goes to $0$ as $R\rightarrow \infty$ due to Lemma \ref{commutator estimate}. 
Moreover, fixing large $R$, \eqref{Halpha-R-ell}, \eqref{Halpha-R-ell2} show that 
$\|\varphi_{R,k}-u_R\|_{H^{\alpha,q}_{\ell}(\Omega)} = \|\varphi_{R,k}-u_R\|_{H^{\alpha,q}_{\ell}(\Omega_{2R})} $ 
tends to $0$ as $k \rightarrow \infty$. 
Therefore, for each $\epsilon>0$ there exists $\varphi_\epsilon \in C^\infty_0(\Omega)$ such that 
$$
\|u-\varphi_\epsilon\|_{H^{\alpha,q}_{\ell}(\Omega)} <\epsilon. 
$$
Consequently $C^\infty_0(\Omega)$ is dense in $H^{\alpha,q}_{\ell}(\Omega)$.
\end{proof}

\vspace{2ex}
Extending the results of \cite[Subsect. 4.3.1, 4.3.2]{Triebel} to exterior domains $\Omega$ with weights $\langle x\rangle^{q\ell}$ we also introduce the intermediate spaces 
\begin{align}\label{def.tildeH}
\widetilde H^{s,q}_\ell(\omega) = \{u\in H^{s,q}_\ell(\IR^n): \supp u\subset\overline\omega\},\quad s\in\IR,
\end{align}
{\em cf.} \cite[Definition 4.3.2]{Triebel}),  where elements $u\in \widetilde H^{s,q}_\ell(\omega)$ are considered as functions (or functionals) on $\IR^n$. The space is equipped with the norm 
$$ \|u\|_{\widetilde H^{s,q}_\ell(\omega)} := \|u\|_{H^{s,q}_\ell(\IR^n)}.$$

In the following, this definition will be used only for $s\in[0,2]$. Then we get for $0<\alpha <1$ and $s_0<s<s_1$, $s=(1-\alpha)s_0 +\alpha s_1$, that 
\begin{equation}\label{Htilde-interpol}
[\widetilde H^{s_0,q}_\ell, \widetilde H^{s_1,q}_\ell]_\alpha = \widetilde H^{s,q}_\ell
\end{equation}
as well as $[H^{s_0,q}_{\ell,0}, H^{s_1,q}_{\ell,0}]_\alpha = H^{s,q}_{\ell,0}$.  The interpolation result \eqref{Htilde-interpol} is proved via the method of local coordinates described in the proof of Proposition \ref{H-alpha-Omega-IRn} and the retraction-coretraction principle, {\em cf.} \cite[Proof of Theorem 4.3.2/2]{Triebel}. \BLACK
\vspace{2ex}

\begin{lem}\label{2.23}
For any smooth bounded or unbounded domain $\omega\subset\IR^n$ of class $C^2$ the space
$C^\infty_0(\omega)$ is dense in $\widetilde H^{s,q}_\ell(\omega)$, $s\in[0,2]$, and hence
\begin{equation}\label{H_0-tildeH}
 \widetilde H^{s,q}_\ell(\omega) \subset \overline{C_0^\infty(\omega)}^{\|\cdot\|_{H^{s,q}_{\ell}(\IR^n)}} \subset \overline{C_0^\infty(\omega)}^{\|\cdot\|_{H^{s,q}_{\ell}(\omega)}} = H^{s,q}_{\ell,0}(\omega). 
\end{equation} 
\end{lem}

\begin{proof}
We give a sketch of the proof.
Following the proof of \cite[Theorem 4.3.2 (b)]{Triebel}  
it is enough to consider \cite[Step 2, p. 318]{Triebel}:   
Let a smooth cut off function $\psi$ with compact support and $f\in \widetilde H^{s,q}_\ell(\omega)$ be given. 
Then we may assume that $\Psi := \psi f \in \widetilde H^{s,q}_\ell(\omega) \hookrightarrow  H^{s,q}_\ell(\IR^n)$. Since $C^\infty_0(\IR^n)$ is dense in $H^{s,q}_\ell(\IR^n)$, we restrict the analysis to $\Psi\in C^\infty_0(\IR^n)$. 
If $s=1$, then obviously $\Psi(\cdot+h)\to \Psi$ in $H^{s,q}_\ell(\IR^n)$ as $h\to 0$ due to $L^q$-continuity in the mean; here it is  helpful that $\Psi$ has compact support so that the weight $w(x) = \langle x\rangle^{q\ell}$ does not disturb the convergence. 
If $0<s<1$, the moment inequality implies that 
\begin{equation}\label{Psi+h}
\|\Psi(\cdot+h) - \Psi\|_{H^{s,q}_\ell} \leq \|\Psi(\cdot+h) - \Psi\|_{H^{1,q}_\ell}^s \|\Psi(\cdot+h) - \Psi\|_{L^q_\ell}^{1-s} \end{equation} 
converges to $0$ as $h\to 0$.  Indeed, the property of compact support of $\Psi$ is not necessary since the weight $w$ admits the smoothness estimate $w(x)\leq c w(x+th)$ for all $x\in\IR^n$, $h\in\IR^n$ with $|h|=1$ and all $0<t<t_0$ with a constant $c=c(t_0)$. 

Using a partition of unity on $\omega$ we have to consider the convergence \eqref{Psi+h} for functions $\Psi = \psi f$ for finitely many functions $\psi\in C^\infty_0(\omega)$ with support in a neighborhood of $\partial\omega$ covering $\partial\omega$; to be more precise, $\supp \psi f\subset \bar\omega$. 
Since $\partial\omega$ is of class $C^1$ at least and has the cone property, \eqref{Psi+h} holds for certain $h$ such that the compact support of $\Psi(\cdot+h)$ is contained in $\omega$. 
Summarizing, we conclude that $C^\infty_0(\omega)$ is dense in $\widetilde H^{s,q}_\ell(\omega)$ and hence $\widetilde H^{s,q}_\ell(\omega) \subset \overline{C_0^\infty(\omega)}^{\|\cdot\|_{H^{s,q}_{\ell}(\IR^n)}}$.
This completes the proof of \cite[Step 2, p. 318]{Triebel}.  

By definition of the norms of $H^{s,q}_{\ell}(\IR^n)$ and $H^{s,q}_{\ell}(\omega)$ the embedding $\overline{C_0^\infty(\omega)}^{\|\cdot\|_{H^{s,q}_{\ell}(\IR^n)}} \hookrightarrow \overline{C_0^\infty(\omega)}^{\|\cdot\|_{H^{s,q}_{\ell}(\omega)}}$ is obvious. Finally, by definition, the latter space equals $H^{s,q}_{\ell,0}(\omega)$. \BLACK
\end{proof}

\begin{prop}\label{char-fractD(Delta)}
Let $\Omega\subset\mathbb R^n$ be a $C^2$ exterior domain and assume $ 1<q<\infty$, $-\frac{n}{q} < \ell < \frac{n}{q'}$.
Let $-\Delta=-\Delta_{q,\ell}$ be the Dirichlet Laplacian on $L^q_\ell(\Omega)$  with domain 
$\mathcal D(-\Delta) = H^{2,q}_\ell(\Omega) \cap H^{1,q}_{\ell,0}(\Omega)$.
If $0< 2\alpha < 1/q$, then
\begin{equation}\label{H2thata}
\D(-\Delta^\alpha) = [L^q_\ell(\Omega),\mathcal D(-\Delta)]_\alpha = H^{2\alpha,q}_\ell(\Omega) = H^{2\alpha,q}_{\ell,0}(\Omega)
\end{equation}
with equivalent norms.
\end{prop}

\begin{proof}
We put $s := 2\alpha$.
First, since $\mathcal D(-\Delta)\hookrightarrow H^{2,q}_\ell(\Omega),$
complex interpolation implies that
$$
[L^q_\ell(\Omega),\mathcal D(-\Delta)]_\alpha
\hookrightarrow
[L^q_\ell(\Omega),H^{2,q}_\ell(\Omega)]_\alpha = H^{s,q}_\ell(\Omega) 
$$
Moreover, by Propositions \ref{H-alpha-Omega-IRn} and \ref{H-alpha-0-0},
$$
[L^q_\ell(\Omega),H^{2,q}_\ell(\Omega)]_\alpha
=
H^{s,q}_\ell(\Omega) = H^{2\theta,q}_{\ell,0}(\Omega).
$$
Hence
\begin{equation}\label{L,DtoHs}
[L^q_\ell(\Omega),\mathcal D(-\Delta)]_\alpha
\hookrightarrow  H^{s,q}_\ell(\Omega).\;  
\end{equation}
Conversely, since $H^{2,q}_{\ell,0}(\Omega)
\subset
H^{2,q}_\ell(\Omega)\cap H^{1,q}_{\ell,0}(\Omega)
= \mathcal D(-\Delta),$
we have
$$
[L^q_\ell(\Omega),H^{2,q}_{\ell,0}(\Omega)]_\alpha
\hookrightarrow
[L^q_\ell(\Omega),\mathcal D(-\Delta)]_\alpha.
$$

We now identify the left-hand side. To this aim we consider for $0\leq s\leq2$ the spaces
$\widetilde H^{s,q}_\ell(\Omega)$, see \eqref{def.tildeH}
Although elements $U\in \widetilde H^{s,q}_\ell(\Omega)$ are considered as functions on $\IR^n$, they can be identified by a restriction operator $R$, see \eqref{R-op} below, with  functions on $\Omega$. In the special cases $s=k=0,1,2$ where the - locally defined - norm  $\|u\|_{W^{k,q}_\ell(\IR^n)} = \sum_{|\alpha|\leq k}\|\nabla^\alpha u\|_{L^q_\ell(\IR^n)}$ is equivalent to the norm of $H^{k,q}_\ell(\IR^n)$ we obviously get that
\begin{align*}
\widetilde H^{k,q}_\ell(\Omega) = \overline{C^\infty_0(\Omega)}^{\,H^{k,q}_\ell(\mathbb R^n)} & = \overline{C^\infty_0(\Omega)}^{\,H^{k,q}_\ell(\Omega)} = H^{k,q}_{\ell,0}(\Omega),\\
 \|U\|_{\widetilde H^{k,q}_\ell(\Omega)} & = \|U\|_{H^{k,q}_\ell(\IR^n)} \;\;\textrm{ for }\; U\in \widetilde H^{k,q}_\ell(\Omega).
\end{align*}
In particular, for $k=0$ we get that
$$
\widetilde L^q_\ell(\Omega) = \widetilde H^{0,q}_\ell(\Omega)
:=
\{U\in L^q_\ell(\mathbb R^n): \operatorname{supp} U \subset\overline\Omega\},
$$
and $ \|U\|_{\widetilde L^q_\ell(\Omega)} = \|U\|_{L^q_\ell(\mathbb R^n)}.$

Next we will show that
\begin{equation}\label{tH=clC}
\widetilde H^{s,q}_\ell(\Omega) = \overline{C^\infty_0(\Omega)}^{\,H^{s,q}_\ell(\mathbb R^n)},\quad 0\leq s\leq 2; 
\end{equation}
note that in \eqref{tH=clC} the norm $\|\cdot\|_{H^{s,q}_\ell(\mathbb R^n)}$ rather than $\|\cdot\|_{H^{s,q}_\ell(\Omega)}$ has been used. 
Obviously, $C^\infty_0(\Omega) \subset \widetilde H^{s,q}_\ell(\Omega)$ and thus  
$$
\overline{C^\infty_0(\Omega)}^{\,H^{s,q}_\ell(\mathbb R^n)}\subset  \overline{\widetilde H^{s,q}_\ell(\Omega)}^{\,H^{s,q}_\ell(\mathbb R^n)} = \widetilde H^{s,q}_\ell(\Omega),  \quad  s\geq 0.
$$
On the other hand, by Lemma \ref{2.23}, 
$\widetilde H^{s,q}_\ell(\Omega) \subset \overline{C^\infty_0(\Omega)}^{\,H^{s,q}_\ell(\mathbb R^n)} $.  Hence we proved \eqref{tH=clC}.  

The zero extension operator
$$
E_0u(x):=
\begin{cases}
u(x),& x\in\Omega,\\  
0,& x\notin\Omega,
\end{cases}
$$
defines an isomorphism
\begin{align*}
E_0:L^q_\ell(\Omega) & \longrightarrow \widetilde L^q_\ell( \Omega),\\
E_0:H^{2,q}_{\ell,0}(\Omega) & \longrightarrow
\widetilde H^{2,q}_\ell(\Omega).
\end{align*}
Indeed, on $C^\infty_0(\Omega)$,  zero extension preserves the norm, and both spaces are defined as the corresponding completions. 
The inverse map is the restriction operator 
$$ R: U\mapsto U|_\Omega $$
which satisfies  
\begin{align}\label{R-op}
	\begin{aligned}
	R:\widetilde L^q_\ell( \Omega)  & \longrightarrow L^q_\ell( \Omega),\\
	R: \widetilde H^{2,q}_{\ell}(\Omega)
&\longrightarrow H^{2,q}_{\ell,0}(\Omega).
\end{aligned}
\end{align}
Thus $R E_0=I.$

By the exactness of complex interpolation, for each $0<\alpha<1$,
$$
[L^q_\ell(\Omega),H^{2,q}_{\ell,0}(\Omega)]_\alpha
\cong
[\widetilde L^q_\ell(\Omega),\widetilde H^{2,q}_\ell(\Omega)]_\alpha.
$$
Furthermore, we see from \eqref{Htilde-interpol} that
\begin{equation}\label{LtHt=Ht}
[\widetilde L^q_\ell(\Omega),\widetilde H^{2,q}_\ell(\Omega)]_\alpha
=
\widetilde H^{s,q}_\ell(\Omega),
\end{equation}

Next we will prove that for $0<s<1/q$ 
\begin{equation}\label{RtH=Hs}
R(\widetilde H^{s,q}_\ell(\Omega))=  H^{s,q}_{\ell,0}(\Omega).
\end{equation}
 Indeed, for $u \in \widetilde H^{s,q}_\ell(\Omega)$, there exists due to \eqref{tH=clC} a sequence $(\phi_j)_j \subset C^\infty_0 (\Omega) $ such that 
$\phi_j \rightarrow u \mbox{ in } H^{s,q}_\ell(\mathbb{R}^n)$
as $j \rightarrow \infty$. Hence, by restriction on $\Omega$ 
and the definition of the norm of $H^{s,q}_\ell(\Omega)$ as an infimum of norms in $H^{s,q}_\ell(\mathbb{R}^n)$, also
 $\phi_j \rightarrow u|_{\Omega} \mbox{ in } H^{s,q}_\ell(\Omega)$,
thus $u|_{\Omega} \in H^{s,q}_{\ell,0}(\Omega)$. 
This implies that 
$$
R(\widetilde H^{s,q}_\ell(\Omega)) \subset H^{s,q}_{\ell,0}(\Omega). 
$$
Conversely, for $u \in H^{s,q}_{\ell,0}(\Omega)$, there  exists a sequence $(\phi_j)_j \subset C^\infty_0 (\Omega) $ such that 
$\phi_j \rightarrow u$ in $H^{s,q}_\ell(\Omega)$ as $j\to\infty$.
Note that the zero extension $E_0$ yields 
$ E_0 \phi_j = \phi_j \in C^\infty_0 (\Omega)$
and hence $E_0 \phi_j \in \widetilde H^{s,q}_w(\Omega)$. In addition, since $RE_0 =I$, also $R(E_0 u)=u$ for $u$ as well as for $\phi_j$. Therefore, as $j\to\infty$, the inclusion \BLACK
\begin{equation*}  
H^{s,q}_{\ell,0}(\Omega) \subset R(\widetilde H^{s,q}_\ell(\Omega))
\end{equation*}
proves \eqref{RtH=Hs}.

Hence, applying the isomorphism $R$ to \eqref{LtHt=Ht}, we conclude from \eqref{RtH=Hs} and \eqref{R-op} that
\begin{align*}
H^{s,q}_{\ell,0}(\Omega) & = R(\widetilde H^{s,q}_\ell(\Omega)) = R[\widetilde L^q_\ell(\Omega),\widetilde H^{2,q}_\ell(\Omega)]_\alpha\\[1ex]
& =[R(\widetilde L^q_\ell(\Omega)), R(\widetilde H^{2,q}_\ell(\Omega))]_\alpha  = [L^q_\ell(\Omega),H^{2,q}_{\ell,0}(\Omega)]_\alpha,
\end{align*}
{\em i.e.,}
$$
H^{s,q}_{\ell,0}(\Omega) = [L^q_\ell(\Omega),H^{2,q}_{\ell,0}(\Omega)]_\alpha
$$

Since $0<s<1/q$, we already showed in Proposition \ref{H-alpha-0-0}  that
$$
H^{s,q}_{\ell,0}(\Omega)=H^{s,q}_\ell(\Omega).
$$
Therefore
$$
H^{s,q}_\ell(\Omega)
=
[L^q_\ell(\Omega),H^{2,q}_{\ell,0}(\Omega)]_\alpha
\hookrightarrow
[L^q_\ell(\Omega),\mathcal D(-\Delta)]_\alpha.
$$

Combining this with the first embedding \eqref{L,DtoHs}, we obtain
$$
[L^q_\ell(\Omega),\mathcal D(-\Delta)]_\alpha
=
H^{s,q}_\ell(\Omega).
$$
Since $s=2\alpha$, this proves \eqref{H2thata}.
\end{proof}

\BLACK
\begin{prop}\label{fractional-power-without-boundary-condition}
Let $1<q<\infty$,  $-\frac{n}{q}< \ell < \frac{n}{q'}$ \BLACK and $0< \alpha < \frac{1}{2q}$. Then the Stokes operator $A_{q,\ell}$ on $L^q_{\sigma,\ell}(\Omega)$ satisfies 
\begin{align}\label{D(Aalpha)-small}
\D(A_{q,\ell}^\alpha) = H^{2\alpha,q}_{\sigma, \ell}(\Omega), 
\end{align}
where the inhomogeneous space $H^{2\alpha,q}_{\sigma, \ell}(\Omega)$ is defined by \eqref{inhomogeneous space}. 

\end{prop}

\begin{proof}
By virtue of Theorem \ref{theorem-H-infty-all} (iii) there holds for $0<\alpha<1$ and $-\frac{n}{q}< \ell < \frac{n}{q'}$, $1<q<\infty$,
$$
\D(A^\alpha)= [L^q_{\sigma, \ell}(\Omega), \D(A)]_\alpha;
$$
here $A=A_{q,\ell}$. By analogy, the Dirichlet-Laplacian $-\Delta=-\Delta_{q,\ell}$ with domain $\D(-\Delta_{\ell,q}) = H^{2,q}_{\ell}(\Omega) \cap H^{1,q}_{\ell,0}(\Omega) \supset H^{2,q}_{\ell,0}(\Omega)$ satisfies
$$\D((-\Delta)^\alpha)= [L^q_{\ell}(\Omega), \D((-\Delta))]_\alpha.
$$

In order to conclude \eqref{D(Aalpha)-small}
we construct
a projection (retraction) $R:L^q_\ell\to L^q_{\sigma,\ell}$  
such that $RL^q_\ell(\Omega) = L^q_{\sigma,\ell}(\Omega)$ and $R\D(-\Delta)=\D(A)= \D(-\Delta) \cap L^q_{\sigma,\ell}(\Omega)$. 
In this situation, \cite[Theorem 1.17.1.1]{Triebel}, Theorem \ref{theorem-H-infty-all} (iii), Propositions \ref{H-alpha-0-0} and  \ref{char-fractD(Delta)} 
\BLACK
and \eqref{inhomogeneous space}, {\em i.e.} $H^{2\alpha,q}_{\sigma,\ell} = H^{2\alpha,q}_{\ell} \cap L^q_{\sigma,\ell}$, 
imply that
\begin{align*}
\D(A^\alpha) & = [L^q_{\sigma,\ell},\D(A)]_\alpha = [L^q_\ell,\D(-\Delta)]_\alpha \cap L^q_{\sigma, \ell}\\
&= H^{2\alpha, q}_{\ell}(\Omega) \cap  L^q_{\sigma,\ell} = H^{2\alpha,q}_{\sigma, \ell}(\Omega).
\end{align*}

Actually, modifying \cite[Sect.~7]{Giga} and choosing $R = (I+A)^{-1} P(I-\Delta)$ we  have to prove that $R:\D(-\Delta) \to \D(A)$ and $R:L^q_\ell\to L^q_{\sigma,\ell}$ are bounded.
The first assertion is obvious by resolvent estimates of $A$  with resolvent parameter $\lambda=1$, \BLACK see Theorem \ref{res-weighted}. 
For the second statement we refer to a duality argument, write for simplicity $A^*$ for $A_{q',-\ell}= (A_{q,\ell})^*$ {\em etc.} and consider $R^*= (I-\Delta^*)(I+A^*)^{-1}$ as map from $L^{q'}_{\sigma,-\ell}$ to $L^{q'}_{-\ell}$. 
Then for any \BLACK  $f\in C^\infty_{0,\sigma}$  
let $u = (I+A^*)^{-1}f$ and $(u,p)$ be a solution to the resolvent problem \begin{align*}
   u  -\Delta  u + \nabla p =f,  \ \ \div u = 0\mbox{ in } \Omega, \quad u=0 \mbox{ on } \partial\Omega.
\end{align*}  
Now Theorem \ref{res-weighted} 
implies that for $-\frac{n}{q'}< -\ell < \frac{n}{q}$ 
$$
\| u\|_{L^{q'}_{-\ell}(\Omega)}+\| \nabla^2 u\|_{L^{q'}_{-\ell}(\Omega)} \leq  C\|f\|_{L^{q'}_{-\ell}(\Omega)}, 
$$
which derives the boundedness of $R^*$ from  $L^{q'}_{\sigma,-\ell}(\Omega)$ to $L^{q'}_{-\ell}(\Omega)$. 
\end{proof}
\BLACK


\subsection{The modified Navier-Stokes system} \label{S2.4}

System \eqref{equ:ns} will be reformulated as a problem in the cylindrical space-time domain $\Omega_0\times J$ where $J=\IR/(T\IZ)$ and $\Omega_0$ is the exterior reference domain. Here we follow the procedure described in \cite{Farwig-Kozono-Wegmann, FKTW, Saal}.

Let $v\colon \Omega(t)\to \IR^n$ be a function with parameter $t$. Given $\phi$ as in Assumption \ref{ass} define the map $\Phi=\Phi(\cdot)$ for fixed $t\in J$ by
\begin{align}\label{equ:trafo}
	\Phi(t) \colon v \mapsto (\Phi(t) v)(\cdot,t) := u(\cdot,t) := ((\nabla \phi)^{-1}(\cdot,t))\, v(\phi(\cdot,t),t),
\end{align}
which maps solenoidal vector fields $v$ on $\Omega(t)$ to solenoidal vector fields $u$ on $\Omega_0$.
Note that for each $1<q<\infty$ and $t\in J$ 
$$\Phi(t): L^q(\Omega(t)) \to L^q(\Omega_0)$$
is an isomorphism; the same holds for the spaces $W^{k,q}$, $k=1, 2$, $W^{1,q}_0$, and $L^q_{\sigma}$ by \cite{Saal}.

We see  from \cite[(13),(14)]{Saal} that
\begin{align}\label{equ:trafodgl}
\begin{aligned}
	(\Phi(t) \Delta_x \Phi(t)^{-1} u )(\xi,t) & =  \sum_{i,j,k,l,m=1}^3 \big(\partial_{x_k} \phi^{-1}\big)\big(\partial_{x_j}\phi^{-1}\big)_i\big(\partial_{x_j}\phi^{-1}\big)_l(\phi(\xi,t),t)\\
 &\qquad\times \Big [  \big(\partial_{\xi_l}\partial_{\xi_i}\partial_{\xi_m}\phi_k\big)u_m
  	+ \big(\partial_{\xi_i}\partial_{\xi_m}\phi_k\big) \partial_{\xi_l}u_m\\
 &\qquad\quad + \big(\partial_{\xi_l}\partial_{\xi_m}\phi_k\big) \partial_{\xi_i}u_m + \big(\partial_{\xi_m}\phi_k\big)
	\partial_{\xi_l}\partial_{\xi_i}u_m\Big]\\
	& =: \Delta_\xi u + \sum_{|\alpha|\leq 2} a_{\alpha}(\xi,t) \partial^{\alpha} u,
\end{aligned}
\end{align}
where the $a_{\alpha}$ denote matrix-valued coefficients, and
\begin{align}\label{equ:trafobeta}
	\Phi(t) \partial_t \Phi(t)^{-1}u & = \partial_t u + (\nabla\phi)^{-1} \!\Bigg(\!\sum_{i,j} (\partial_t\phi^{-1})_j[\partial_{\xi_i}\partial_{\xi_j}\phi\,u_i + \partial_{\xi_i}\phi\,\partial_{\xi_j}u_i] + \sum_i (\partial_{\xi_i}\partial_t\phi)u_i  \!\Bigg)\nonumber\\
& =: \partial_t u + \sum_{|\beta|\leq 1} b_{\beta}(\cdot,t)\partial^{\beta}u
\end{align}
with matrix-valued coefficients $b_\beta(\xi,t)$.
Then the following inequalities follow from Assumption \ref{ass} (iv):
\begin{align}\label{equ:alpha-beta=0}
	\|a_{\alpha}\|_{ L^\infty_{n-1} \cap L^{q_1,1}_n\BLACK} + \|b_{\beta}\|_{ L^\infty_{n-1} \cap L^{q_1,1}_n\BLACK}  & \leq c  \mathfrak{C}_\phi,\\[1ex] 
\|a_{\alpha}(\cdot,t) - a_{\alpha}(\cdot,\tau)\|_{\infty} +  \|b_{\beta}(\cdot,t) -b_{\beta}(\cdot,\tau)\|_{\infty} & \lesssim |\phi(t)-\phi(\tau)|_{\widehat C^{3,1}},
\label{equ:alpha-beta}
\end{align}
where in \eqref{equ:alpha-beta=0} $\mathfrak{C}_\phi = \mathcal O(c_\phi)$,
see {\em e.g.} \cite[(2.14), (2.15)]{FT-lin-half} for details,
and \eqref{equ:alpha-beta} is mainly due to 
\cite{Saal}. Since by an inspection of the proof in \eqref{equ:alpha-beta} only derivatives of $\phi$ are needed, we may take the seminorm $\widehat C^{3,1}$ for the proof. We also note that $a_\alpha(\cdot,0)\equiv 0$, 
but in general $b_\beta(\cdot,0)\neq 0$. 
\BLACK

Now \eqref{equ:ns} is reformulated as a problem
on $\Omega_0 \BLACK \times J$ by using the transformation $\Phi(t)$. To be more precise, we are led to the nonlinear system 
\begin{align}\label{equ:ns2}
\begin{aligned}
	u_t + \sum_{|\beta|\leq 1 } b_{\beta}\partial^{\beta} u  - \Delta u - \sum_{|\alpha|\leq 2 } a_{\alpha} \partial^{\alpha} u + \nabla^{\phi(t)} \tilde p + u\cdot \nabla^{\phi(t)} u & = \Phi f\BLACK,\\
	\div u = 0,\quad u_{|\partial\Omega_0} & =0, 
\end{aligned} 
\end{align}
where $\tilde p= p\circ \psi$ and 
$$ \nabla^{\phi(t)} \tilde p = (\nabla \phi(t))^{-1}  \big((\nabla \phi(t))^{-1}\big)^{\top} \nabla\tilde p; $$
moreover, $(u,\tilde p)$ solves \eqref{equ:ns2} iff $(v,p)$ solves \eqref{equ:ns}.  For simplicity, we will write $p=p_s+p_\perp$ for $\tilde p$ and $f=f_s+f_\perp$ for $\Phi f$.

We decompose \eqref{equ:ns2} by the projections $P_s$ and $P_\perp$  into the equations \BLACK 
\begin{equation}\label{u-stat-m} 
 -\Delta u_s + \nabla p_s = K(u,p), \quad \div u_s = 0\; \textrm{in }\;\Omega_0,\; u_s\big|_{\partial\Omega_0} = 0, 
\end{equation}
 and 
\begin{equation}\label{u-perp-m} 
\del_t u_\perp- \Delta u_\perp + \nabla p_\perp = N(u, p), \quad \div u_\perp = 0 \; \textrm{in }\;\Omega_0,\; u_\perp\big|_{\partial\Omega_0} = 0, 
\end{equation}
 where 
\begin{align}\label{KNM}
\begin{aligned}
K(u, p) & = -(u_s\cdot\nabla u_s)  -
P_s \{u_s \cdot (\nabla^{\phi(t)}-\nabla^{\phi(0)})u_s\}
+ P_s (\nabla^{\phi(0)}-\nabla^{\phi(t)}) p\\
& \quad -  P_s( u_s \cdot\nabla^{\phi(t)} u_\perp) - P_s ( u_\perp\cdot\nabla^{\phi(t)} u_s)-  P_s(u_\perp\cdot\nabla^{\phi(t)} u_\perp) + f_s+P_s M(u), \\[1ex]
N(u, p) & =  P_\perp(\nabla^{\phi(0)}-\nabla^{\phi(t)}) p + f_\perp(t)-P_\perp (u_s\cdot\nabla^{\phi(t)} u_\perp) -P_\perp( u_\perp(t)\cdot\nabla^{\phi(t)} u_s) \\
& \quad - P_\perp (u_\perp\cdot\nabla^{\phi(t)} u_\perp) + P_\perp M(u) - P_\perp (u_s\cdot\nabla^{\phi(t)} u_s), \\[1ex]
M(u) & = \sum_{|\alpha|\leq 2 } a_{\alpha} \partial^{\alpha} u- \sum_{|\beta|\leq 1 } b_{\beta}\partial^{\beta} u 
\end{aligned}
\end{align}
for $u=u_s +u_\perp$, $p_s=P_s p$ and 
$p_\perp=P_\perp p$. Note that $K(u,p)$ and $N(u,p)$ contain the external force term $f$. \BLACK Applying  $(I-\mathbb{P})$ to \eqref{u-perp-m} we derive that 
\begin{equation}\label{u-perp-eq-p} 
\nabla p_\perp = (I- \mathbb{P})\Delta u_\perp + (I-\mathbb{P})N(u,p\BLACK).  
\end{equation}

By the formula of Kozono-Nakao \cite{Kozono-Nakao}, we obtain the following formula from \eqref{u-perp-m}  to get a time periodic solution $u_\perp$: 
\begin{equation}\label{u-perp-int}
u_\perp (t) = H[u,p](t):= \displaystyle\int_{-\infty}^t e^{-(t-\tau)A}  \mathbb P \BLACK N(u,p)(\tau) \dtau.  
\end{equation}
Therefore,  we will find a $T$-periodic solution $u=u_s+u_\perp$, $p= p_s+ p_\perp$ as a fixed point of a coupled nonlinear system containing the stationary Stokes problem \eqref{u-stat-m} in $(u_s,p_s)$ and the instationary system \eqref{u-perp-m} in $u_\perp$ - with associated pressure $p_\perp$ - solved by \eqref{u-perp-int}; then $\nabla p_\perp$ is defined via \eqref{u-perp-eq-p}.
\BLACK

\BLACK

\BLACK 

\vspace{1ex}

\subsection{Formulation as Integral Equation}\label{S2.5}

We look for a fixed point $(u_s, u_\perp, p_s, p_\perp)$ satisfying the coupled nonlinear system  
\begin{align} \label{equ:ns-us-per}
\left\{
\begin{array}{ll}
-\Delta u_s + \nabla p_s = K(u,p), \quad \div u_s = 0\; \textrm{in }\;\Omega_0,\; u_s\big|_{\partial\Omega_0} = 0, \\
	u_\perp (t) = H[u,p](t)= \displaystyle\int_{-\infty}^t e^{-(t-\tau)A}  \mathbb P\BLACK N(u,p)(\tau) \dtau, \\
\hspace*{3pt} \nabla p_\perp = (I- \mathbb{P})\Delta u_\perp + (I-\mathbb{P}) N(u,p\BLACK),  
    \end{array}
\right. 
\end{align}
where $A = -\mathbb P\Delta$ denotes the Stokes operator in a weighted space $L^q_{\sigma,n-1}(\Omega)$.
Note that the integral mean of $N(u,p)$ on $(0,T)$ vanishes, {\em i.e.,} $P_s N(u,p)=0$. 

Decomposing the integral on $(0,\infty)$ into consecutive parts on $(kT,(k+1)T)$, $k\in\IN_0,$ we get that
\begin{align}\label{equ:ns-per2}
\begin{aligned}
u_\perp(t) = & \sum_{k=0}^\infty \int_{kT}^{(k+1)T} e^{-\tau A} \mathbb P N(u, p)(t-\tau)\dtau\\
= & \sum_{k=0}^\infty \int_0^T e^{-(\tau+kT) A} \mathbb P N(u, p)(t-\tau-kT) \dtau\\
= & \sum_{k=0}^\infty e^{-kT A} \int_0^T e^{-\tau A}\mathbb P N(u, p)(t-\tau-kT) \dtau.
\end{aligned}
\end{align} 
Assuming that $f, \phi$ and $u,p$ are $T$-periodic, we simplify 
\eqref{equ:ns-per2} to 
\begin{align}\label{u(t)} 
u_\perp(t) = \Big(\sum_{k=0}^\infty e^{-kT A}\Big) \int_0^T e^{-\tau A} \mathbb P N(u, p)(t-\tau) \dtau. 
\end{align}
Since the above integrand $ \mathbb P \BLACK N(u, p)$ has vanishing integral mean, we may replace the Stokes semigroup $e^{-\tau A}$ by 
$$e^{-\tau A}-I = \int_0^\tau -Ae^{-sA} \ds = - A\int_0^\tau e^{-sA} \ds $$
so that $u_\perp$ in \eqref{u(t)} satisfies
\begin{align}\label{u_perp2}
\begin{aligned}
u_\perp(t)  & = \Big(\sum_{k=1}^\infty (-A) e^{-kT A}\Big) \int_0^T \Big(\int_0^\tau e^{-sA} \ds\Big)\mathbb P N(u, p) (t-\tau)\dtau\\
& \quad\; + \int_0^T e^{-\tau A} N(u, p)(t-\tau) \dtau. \end{aligned}
\end{align} 
In view of \eqref{u_perp2}  and the assumption $P_s N(u, p)=0$ we decompose $H$, see \eqref{equ:ns-us-per}, for $k\geq 1$ and $k=0$ \BLACK into $H_1$ and $H_2$, respectively,  with  
\begin{align}\label{H_1}
   \begin{aligned}
H_1[u,p](t) 
& = \Big(\sum_{k=1}^\infty (-A) e^{-kT A}\Big) \int_0^T \Big(\int_0^\tau e^{-sA} \ds\Big) \mathbb P N(u, p)(t-\tau) \dtau,\\
H_2[u ,p\BLACK](t) & =  \int_0^T e^{-\tau A} N(u, p)(t-\tau) \dtau. 
\end{aligned}
\end{align}

\section{Analysis of the Stationary Stokes Equations}

We consider on a smooth exterior domain $\Omega_0\subset \IR^n$ the linear stationary problem 
\begin{align}\label{stationary-problem}
-\Delta u +\nabla p= \div F +g,  \ \ \div u=0\; \mbox{ in }\Omega_0,\; \ u=0 \mbox{ on }\partial\Omega_0. 
\end{align}
In this section, our aim is to prove the following estimate. 

\vspace{1ex}

\begin{prop}\label{result-linear-stationary} 
Let $n< q_0 <q$ and $n\leq \ell$. Assume that 
$$
F\in L^{\infty}_{n-1}(\Omega_0), \ \ \div F \in L^\infty_{n}(\Omega_0)
$$
and
$$
g=g_1+g_2,  \ \ g_1 \in L^{q}_{\ell}(\Omega_0),  \ \ g_2 \in L^{q_0}_\ell (\Omega_0) \cap L^{q,\infty}(\Omega_0).
$$ 
 Set 
\begin{align}\label{norm-F-g}
\begin{aligned}
\mathscr{L}(F) & = \|F\|_{L^\infty_{n-1}}+ \|\div F\|_{L^\infty_{n}},\\
\mathscr{L}'(g) & = \inf_{g=g_1+g_2} \big(\|g_1\|_{L^q_\ell} + \|g_2\|_{L^{q_0}_\ell \cap L^{q,\infty}}\big),
\end{aligned}
\end{align} 
where the infimum runs over all decompositions 
$g=g_1+g_2,  \ \ g_1 \in L^{q}_{\ell},  \ \ g_2 \in L^{q_0}_\ell \cap L^{q,\infty}.$ 
Then there exists a solution $(u,p) \in H^{1.q} (\Omega_0) \times L^q (\Omega_0)$ with 
 $\nabla^2 u \in L^{q,\infty}$ and $\nabla p \in L^{q,\infty}$ to \eqref{stationary-problem} satisfying that 
\begin{equation}\label{u-p-est}
\|\nabla^2 u\|_{L^{q,\infty}}  + \|u\|_{L^\infty_{n-2}\cap H^{1,q}} + \|\nabla u\|_{L^\infty_{n-1}} + \|\nabla p\|_{L^{q,\infty}}+\|p\|_{L^\infty_{n-1} \cap L^q} \leq 
C(\mathscr{L}(F)+\mathscr{L}'(g)). 
\end{equation}
\end{prop}

\vspace{2ex}


We note that $\mathscr{L}(F)$ and $\mathscr{L}'(g)$
stand for expressions like norms, but 
$\mathscr{L}(F) + \mathscr{L}'(g)$ is neither for a pair $(F,g)$ nor $(\div F,g)$ a norm in \eqref{stationary-problem}.

To prove Proposition \ref{result-linear-stationary}, 
we first state the following $L^\infty$ estimates.  We note that the proof of Lemma \ref{convolution2} (ii) follows the ideas of \cite[Lemma 5.4]{Eiter-Kyed-Shibata23}  adapted to the $n$-dimensional case. We also exploit H\"older's inequality to get the embedding \begin{equation}\label{Lq-Linfty_s}
  L^\infty_\ell(\IR^n) \hookrightarrow L^q(\IR^n),\quad \ell>\frac{n}{q}.
\end{equation} 
In particular, when $3\leq n<q<\infty$, then $\ell=n,n-1$ and $\ell=n-2$ are admitted.\BLACK

\vspace{2ex}

\begin{lem}\label{convolution2}
	{\rm (i)}
	Let $E(x)$ be a scalar function on $\IR^n$ satisfying 
	\begin{eqnarray}\label{spatial decay}
		|\nabla^{\alpha}_{x}E(x)|\leq \frac{C}{|x|^{|\alpha|+n-2}}, \ \ |\alpha|=0,1,2. 
	\end{eqnarray}
	Assume that $f$ is a scalar function satisfying $\|f\|_{L^{q}_\ell }<\infty$, where 
    $q>n$ and $\ell > \frac{n}{q'}$. Then
	\begin{eqnarray*}
		\|\nabla^{\alpha}_{x}E \ast f\|_{L^\infty_{ n-2+|\alpha| }} \leq C\|f\|_{L^{q}_\ell }, \quad |\alpha|=0,1.
	\end{eqnarray*}
	
	{\rm (ii)} 
	Let $E(x)$ be a scalar function satisfying \eqref{spatial decay} and assume that $f$ is a scalar function of the form $f=\partial_{x_j}F$ for some $1 \leq j \leq n$ satisfying $\|\partial_{x_j}F\|_{L^{\infty}_n} + \|F\|_{L^{\infty}_{n-1}}<\infty$. 
	Then there hold the following estimates for $|\alpha|=0,1$:
	\begin{eqnarray}\label{E*f-infty}
		\|\nabla^{\alpha}_{x}E \ast f\|_{L^\infty_{n-2+|\alpha|  }} \leq 
        C\big(\|\partial_{x_j}F\|_{L^{\infty}_n} + \|F\|_{L^{\infty}_{n-1}}\big).
	\end{eqnarray}
\end{lem}

\begin{proof} \BLACK
(i) We differ between the cases $|x|<1$ and $|x|\geq 1$, and use for $u(x) = E*f(x)$, $x\in\IR^n$, the estimate 
\begin{align}\label{Ef-ptw}
    |u(x)|\leq c \Bigg( \int_{|y|<\frac{|x|}{2}} + \int_{|y|>2|x|} + \int_{\frac{|x|}{2}\leq y\leq 2|x|}\Bigg) \frac{|f(y)| (1+|y|)^\ell}{|x-y|^{n-2} (1+|y|)^\ell} \dy .
\end{align}
If $|x|\geq 1$, \BLACK we refer to H\"older's inequality to get the term $\|f\|_{L^q_\ell}$, exploit $\ell>\frac{n}{q'}$ and $q>n\geq 3$, and estimate the three remaining $f$-independent integrals by elementary techniques. However, if $|x|<1$, we use \eqref{Ef-ptw}, but take special care  when $y\sim x$ in the third integral. By analogy, the convolution integral $\nabla u$ is estimated.

    (ii)
For simplicity, set $\gamma_F = \|\partial_j F\|_{L^\infty_n} + \|F\|_{L^\infty_{n-1}}$. 
By Gauss' divergence theorem, we write 
\begin{align*}
u(x) & = \int_{|y| \leq \frac{|x|}{2}} E(y)(\partial_j F)(x-y)\dy\, 
 + \int_{|y|=\frac{|x|}{2}} E(y)\frac{y_j}{|y|}\cdot F(x-y)\domega \\[1ex]
& \qquad + \int_{\frac{|x|}{2} \leq |y| \leq 2|x|}
\partial_j E(y)F(x-y)\dy + \int_{|y| \geq 2|x|}\partial_j E(y)F(x-y)\dy.
\end{align*}
Noting that $|x-y| \geq |x|/2$ for $|y| \leq |x|/2$, $|x-y| \leq 3|x|$ 
for $|x|/2 \leq |y|
\leq 2|x|$, and $|x-y| \geq |y|/2$ for $|y| \geq 2|x|$, 
by \eqref{spatial decay} we have
\begin{align*}
|u(x)|  \leq C\gamma_F\Bigl\{ & (1+|x|)^{-n}
\int_{|y| \leq \frac{|x|}{2}} |y|^{-(n-2)}\dy 
+ |x|^{-(n-2)}(1+|x|)^{-(n-1)} \int_{|y| = |\frac{x|}{2}}\domega  \\[1ex]
&
+|x|^{-(n-1)} \int_{|z| \leq 3|x|}|z|^{-(n-1)}\dz + \int_{|y| \geq 2|x|} |y|^{-2(n-1)}\dy\Bigr\}
\;\leq\; \frac{C\gamma_F}{|x|^{n-2}}
\end{align*}
for $x\neq0$.
When $|x| \leq 1$, we see that $|(\partial_j F)(x-y)| \leq \gamma_F$ for $|y| \leq 2$ and 
$|(\partial_j F)(x-y)| \leq C\gamma_F |y|^{-n}$ for $|y| \geq 2$, and get additionally that 
\begin{align*}
|u(x)| & \leq \int_{|y| \leq 2}|E(y)(\partial_j F)(x-y)|\dy + \int_{|y| \geq 2}
|E(y)(\partial_j F)(x-y)|\dy \\[1ex]
& \leq C\gamma_F\Bigl\{ \int_{|y| \leq 2}|y|^{-(n-2)}\dy 
+ \int_{|y| \geq 2} |y|^{-n-(n-2)}\dy\Bigr\}\,\leq\, C\gamma_F.
\end{align*}
In total, we obtain that $\|u\|_{L^\infty_{n-2}} \leq C\gamma_F$. 

Similarly, we proceed with the estimate of $\nabla u$ and write 
\begin{align*}
\nabla u(x) & = \int_{|y| \leq \frac{|x|}{2}} \nabla E(y)(\partial_j F)(x-y)\dy\, 
 + \BLACK \int_{|y|=\frac{|x|}{2}} \nabla E(y)\frac{y_j}{|y|}\cdot F(x-y)\domega
\\[1ex]
& + \int_{\frac{|x|}{2} \leq |y| \leq 2|x|}
\nabla\partial_j E(y)\BLACK F(x-y)\dy + \int_{|y| \geq 2|x|} \nabla\partial_j E(y)\BLACK F(x-y)\dy.
\end{align*}
Then, we have
\begin{align*}
|\nabla u(x)|  \leq C\gamma_F & \Bigl\{(1+|x|)^{-n}\!
\int_{|y| \leq \frac{|x|}{2}} |y|^{-(n-1)}\dy
+ |x|^{-(n-1)}(1+|x|)^{-(n-1)}\! \int_{|y|=\frac{|x|}{2}}\! \domega \\
& + |x|^{-n} \int_{|z| \leq 3|x|} |z|^{-(n-1)}\dz 
+ \int_{|y| \geq 2|x|}|y|^{-n-(n-1)}\dy\Bigr\}
\leq \frac{C\gamma_F}{|x|^{n-1}}
\end{align*}
for all $x\neq 0$. When $|x| \leq 1$, arguing as above, we have 
\begin{align*}
|\nabla u(x)| & \leq \int_{|y| \leq 2}|\nabla E(y)(\partial_j F)(x-y)|\dy 
+ \int_{|y| \geq 2}
|\nabla E(y)(\partial_j F)(x-y)|\dy \\[1ex]
& \leq C\gamma_F\Bigl\{ \int_{|y| \leq 2}|y|^{-(n-1)}\dy 
+ \int_{|y| \geq 2} |y|^{-(2n-1)}\dy\Bigr\}\leq C\gamma_F.
\end{align*}
Summing up, we have $\|\nabla u\|_{L^\infty_{n-1}} \leq C\gamma_F$.
\end{proof}

\vspace{1ex}


Next we consider the Stokes equation in $\mathbb{R}^n$,  $n\geq 3$, 
\begin{align}\label{stationary-problem-whole-space-n}
-\Delta u +\nabla p= f,  \ \ \div u=0.  
\end{align}
We set 
$$
u(x)= U\ast f(x),  \ \ p(x)=Q \ast f (x), 
$$
where $(U,\,Q)$ is the fundamental solution to \eqref{stationary-problem-whole-space-n}, given by  
\begin{equation} \label{UQ} 
U_{ij}(x) = \frac{1}{2\omega_n}\Big(\frac{\delta_{ij}}{(n-2)|x|^{2-n}}+\frac{x_i x_j}{|x|^n}\Big), \ \ Q_j(x) = \frac{x_j}{\omega_n|x|^n};
\end{equation}\BLACK
here $\omega_n$ denotes the surface area of the $(n-1)$-dimensional unit sphere in $\mathbb{R}^n$; {\em e.g.,} see \cite[IV.2]{Galdi-steady}. 
Obviously,  for any multi-index $\alpha\in\IN_0$, \BLACK $U$ and $Q$ satisfy the estimates
\begin{eqnarray}\label{spationa-decay-n}
		|\partial^{\alpha}_{x}U(x)|\leq \frac{C}{|x|^{|\alpha|+n-2}},  \ \ 
        |\partial^{\alpha}_{x} Q(x)|\leq \frac{C}{|x|^{|\alpha|+n-1}}. 
	\end{eqnarray}

\vspace{1ex}

\begin{lem}\label{result-stationary-problem-whole-space}
{\rm (i)} Let $u= U\ast f$ and $p=Q\ast f$ for $f =\div F$ with $\div F \in L^\infty_{ n }(\IR^n)$ and $F \in L^\infty_{n-1}(\IR^n)$. Then for each $n<q<\infty$ we have 
$$
\|u\|_{H^{2,q}} + \|u\|_{L^\infty_{n-2}}+\|\nabla u\|_{L^\infty_{n-1}}+ \|p\|_{H^{1,q}}+\|p\|_{L^\infty_{n-1}} \leq 
C\mathscr{L}(F).
$$

{\rm (ii)} Let $u= U\ast g$ and $p=Q\ast g$ for  $g \in L^q_\ell(\IR^n)$   with $n<q$, $n \leq \ell$.  Then we have  
\begin{align}
\begin{aligned}\label{potential-est1}
\|u\|_{H^{2,q}} + \|u\|_{L^\infty_{n-2}}+\|\nabla u\|_{L^\infty_{n-1}}+ \|p\|_{H^{1,q}}+\|p\|_{L^\infty_{n-1}} \leq 
C  \|g\|_{L^q_\ell}. 
\end{aligned}
\end{align}

{\rm (iii)} Let $u= U\ast g$ and $ p=Q\ast g$ for  $g \in L^{q_0}_\ell \cap L^{q,\infty}$ with $n< q_0 <q$ and  $n \leq \ell$. Then  
\begin{align}\label{potential-est2}
\begin{aligned}
\|u\|_{H^{1,q}} + \|u\|_{L^\infty_{n-2}}+\|\nabla u\|_{L^\infty_{n-1}}+ \|p\|_{L^q}+\|p\|_{L^\infty_{n-1}} & \leq 
C \|g\|_{L^{q_0}_\ell} \\[1ex]
\|\nabla^2 u\|_{L^{q,\infty}} + \|\nabla p\|_{L^{q,\infty}} & \leq 
C  \|g\|_{L^{q,\infty}}.
\end{aligned}\end{align} 

{\rm (iv)} 
Let $q>n$, $u= U\ast g$ and $p=Q\ast g$ where $g \in L^q(\IR^n)$ is a vector field with compact support. \BLACK
Then we have 
\begin{align}
\begin{aligned}\label{potential-est3}
\|u\|_{H^{2,q}} + \|u\|_{L^\infty_{n-2}}+\|\nabla u\|_{L^\infty_{n-1}}+ \|p\|_{H^{1,q}}  + \|p\|_{L^\infty_{n-1}} \BLACK\leq 
C  \|g\|_{L^q}. 
\end{aligned}
\end{align}

{\rm (v)} 
Let $u= U\ast g$ and $p=Q\ast g$ for  $g \in L^{q,\infty}(\IR^n)$, $q>n$,  with compact support. \BLACK
Then 
\begin{align}
\begin{aligned}\label{potential-est11}
\|u\|_{H^{2;q,\infty}} + \|u\|_{L^\infty_{n-2}}+\|\nabla u\|_{L^\infty_{n-1}}+ \|p\|_{H^{1;q,\infty}}+ \|p\|_{L^\infty_{n-1}} \BLACK\leq 
C  \|g\|_{L^{q,\infty}}.  
\end{aligned}
\end{align}
\end{lem}


\begin{proof} {\rm (i)} The weighted $L^\infty$-estimates of $u$ are proved by Lemma \ref{convolution2} {\rm (ii)} which is easily extended to the vector-valued case when $F$ is matrix-valued. 
For $n< q <\infty$ the theory of singular integrals on $\IR^n$ and \eqref{Lq-Linfty_s} imply that 
\begin{align*}
\begin{aligned}
\|\nabla^2 u\|_{L^q} + \|\nabla p\|_{L^q}
& \leq C\|\dv F\|_{L^q}
\leq C_q\|\dv F\|_{L^\infty_n},  \\[1ex]
\|\nabla u\|_{L^q } + \|p\|_{L^q}
& \leq C\|F\|_{L^q } 
\leq C_q\|F\|_{L^\infty_{n-1}}.
\end{aligned}
\end{align*}
As for $\|u\|_{L^q}$ we use that  $\|u\|_{L^q} \leq C\|u\|_{L^\infty_{n-2}}$, see \eqref{Lq-Linfty_s}, and apply \eqref{E*f-infty}. The estimate of $p\in L^\infty_{n-1}$ follows from the decay properties of $Q$, see \eqref{spationa-decay-n}, and the ideas used for $\nabla u$. 
\BLACK

{\rm (ii)} 
For the weighted $L^\infty$ estimates we note that $\ell\geq n>\frac{n}{q'}$ and refer to Lemma \ref{convolution2} {\rm (i)}. The $L^q$ estimates are based on the theory of singular integral operators which implies that 
\begin{equation}\label{proof:3.1a}
\|\nabla^2 u\|_{L^q}
+ \|\nabla p\|_{L^q}  \leq C\|g\|_{L^q_\ell}. 
\end{equation}
From the $L^\infty$ estimates of $u,\nabla u,p$ and \eqref{Lq-Linfty_s} we conclude that 
$$\|\nabla u\|_{L^q} + \|u\|_{L^q} + \|p\|_{L^q} \leq C\|g\|_{L^q_\ell}. $$
This estimate together with \eqref{proof:3.1a} completes the proof of (ii). 
\BLACK

{\rm (iii)} The weighted $L^\infty$ estimates are derived similarly  by Lemma \ref{convolution2} {\rm (i)} for $n< q_0$. 
Moreover, from the $L^q$ estimates \eqref{proof:3.1a} and real interpolation  we derive that 
\begin{equation}
\|\nabla^2 u\|_{L^{q,\infty}}
+ \|\nabla p\|_{L^{q,\infty}}  \leq C\|g\|_{L^{q,\infty}}.  
\end{equation}
Taking in the pointwise estimate $|u(x)| \leq C\langle x\rangle^{-(n-2)} \|g\|_{L^{q_0}_\ell}$, see Lemma \ref{convolution2} (i), the $q$th power and integrating on $\IR^n$, we obtain that $\|u\|_{L^q}$ is bounded by $\|g\|_{L^{q_0}_\ell}$. 
By analogy, we estimate $\nabla u,p$ in $L^q$.  

{\rm (iv)}  If $g\in L^q$  has compact support, \BLACK 
then $g\in L^q_\ell$ for any $\ell\geq n$. Thus \eqref{potential-est1} yields \eqref{potential-est3}. 

{\rm (v)} follows directly from {\rm (iv)} by real interpolation theory. 
\end{proof}

\vspace{1ex}

We introduce some notation to construct a solution operator for the stationary Stokes problem \eqref{stationary-problem} on $\Omega_0$. 
Let $f_0$ be the zero extension of 
$f$ to the complement of $\Omega_0$. For simplicity we write $U f_0 = U*f_0$ and $ Q f_0 = Q *f_0$. 
Let $\Omega_b = \Omega_0 \cap B_b$ with a constant $b>0$ satisfying 
$\Omega_0^c \subset B_{b}$,
$f_b$ be the restriction of $f$ to $\Omega_{4b}$,
and let 
$\CA_0$ and $\CB_0$ be the
operators acting 
on $f_b$  such that 
$\CA_0 f_b $,
$\CB_0 f_b $ 
satisfy the equations 
\begin{equation}\label{eq:3.5} 
- \Delta \CA_0 f_b + \nabla \CB_0 f_b = f_b,  \quad\dv\CA_0 f = 0
\quad\text{in $\Omega_{4b}$}, \quad
 \CA_0 f_b \big|_{\del \Omega_{4b}} =  0.  
\end{equation}


\begin{cor}\label{result-stationary-problem-A-B}
{\rm (i)} For $f \in L^q(\Omega_0)\BLACK$ with $n<q$ it holds that 
\begin{align}
\begin{aligned}\label{est:3.5}
& \|\CA_0 f_b\|_{H^{2,q}(\Omega_{4b})}+ \|\CA_0 f_b\|_{L^{\infty}_{n-2}(\Omega_{4b})}+ \|\nabla \CA_0 f_b\|_{L^{\infty}_{n-1}(\Omega_{4b})} \\ 
& \quad + \|\CB_0 f_b\|_{ H^{1,q}(\Omega_{b})} + \|\CB_0 f_b\|_{L^{\infty}_{n-1}(\Omega_{b})}\\
& \leq C\| f_b\|_{L^q(\Omega_{b})}.
\end{aligned}
\end{align}

{\rm (ii)} For $f \in L^{q,\infty}(\Omega_0)\BLACK$ with $n<q$ it holds that 
\begin{align}
\begin{aligned}\label{est:3.5i}
&\|\CA_0 f_b\|_{H^{2;q,\infty}(\Omega_{4b})} + \|\CA_0 f_b\|_{L^{\infty}_{n-2}(\Omega_{4b})}+ \|\nabla \CA_0 f_b\|_{L^{\infty}_{n-1}(\Omega_{4b})} \\
&\quad + \|\CB_0 f_b\|_{ H^{1;q,\infty}(\Omega_{4b})}+ \|\CB_0 f_b\|_{L^{\infty}_{n-1}(\Omega_{b})}\\
& \leq C\| f_b\|_{L^{q,\infty}(\Omega_{b})}.
\end{aligned}
\end{align}
\end{cor}

\begin{proof}
 (i) is a consequence of standard Stokes {\em a priori} estimates on the bounded
domain $\Omega_{4b}$, see {\em e.g.} \cite [Chapter IV.6]{Galdi-steady}.
Here the $L^\infty$ estimates follow from Sobolev
embeddings on $\Omega_{4b}$.  On $\Omega_{4b}$ the weight functions are uniformly bounded. 

Part (ii) is based on real interpolation and (i).
\end{proof}

\vspace{1ex}

We may assume that 
\begin{equation}\label{eq:3.6}
\int_{\Omega_{4b}}(Q f_0 - \CB_0 f_b) \dx = 0
\end{equation}
because $\CB_0 f_b$ is only defined up to a constant which is to be chosen suitably.

In addition, let $\varphi$ be a function in  
$C^\infty_0(\IR^n)$ 
that equals $1$ for $x \in B_{2b}$ and $0$ for $x\not\in B_{3b}$, and let
$\BB$ be the Bogovski\u\i \, operator on $B':= B_{3b} \setminus \overline{B_{2b}}$ such that functions $\BB[\cdot]$ can be extended by $0$ in suitable Sobolev spaces to functions on $\Omega$. 
In particular, $\BB$ is bounded from $L^q_m(B')$ to $W^{1,q}_{0}(B')$ and from $W^{k,q}_{0,m}(B')$ to $W^{k+1,q}_0(B')$, $k=1,2$, where the subscript $m$ indicates that the integral mean on $B'$ vanishes.
With $U,Q$ defined in \eqref{UQ} \BLACK let  
\begin{align}\label{proof:3.1}
\begin{aligned}
\CY_0f & = (1-\varphi)U f_0 + \varphi \CA_0f_b + 
\BB[(\nabla\varphi)(U f_0 - \CA_0f_b)],\\
\CZ_0f & = (1-\varphi)Q f_0 + \varphi \CB_0 f_b.
\end{aligned}
\end{align}
 Note that $\int_{B'} (\nabla\varphi)(U f_0 - \CA_0\bff_b)\dx = -\int_{\partial B_{2b}}  (U f_0 - \CA_0\bff_b)\cdot \textsl{n} \,{\rm d}\omega =0$ since $\dv (U f_0) =0$ in $B_{2b}$ and $\dv (\CA_0 f_b) =0$ in $\Omega_0\cap B_{2b}$.  Moreover, $(1-\varphi)f_0 +\varphi f_b = f$ in $\Omega_0$.  \BLACK

Inserting \eqref{proof:3.1} into the Stokes system \eqref{stationary-problem}, 
we have 
\begin{equation}\label{proof:3.2}
-\Delta \CY_0f + \nabla\CZ_0f = f + \CR_1f, 
\quad \dv \CY_0 f = 0 \quad\text{in $\Omega_0$}, \quad \CY_0f|_{\del \Omega_0}=0, 
\end{equation}
where the operator $\CR_1$ is defined by 
\[
\begin{aligned}
\CR_1f & = 2(\nabla\varphi)\cdot(\nabla U f_0 - \nabla\CA_0f_b) 
+ (\Delta\varphi)(Uf_0 - \CA_0f_b)\\
& \qquad - \Delta\BB[(\nabla\varphi)
\cdot (Uf_0 - \CA_0f_b)] - (\nabla\varphi)(Qf_0 - \CB_0f_b).
\end{aligned}
\]
We note that, due to cut-off properties,
$\supp \CR_1 f \subset B'\subset \overline{\Omega_{3b}}$. 

\vspace{1ex}

\begin{lem}\label{R1estimate} 
Let $q>n$. Then the operator $\CR_1$ is compact on $L^q(\Omega_{3b})$ and on $L^{q,\infty}(\Omega)$ for functions $f$ with  $\supp f\subset\overline{\Omega_{3b}}$. \BLACK

{\rm (i)} It holds the estimate
\begin{equation}\label{est:6.3.3-0}
\|({\rm I} + \CR_1)^{-1} f\|_{ L^q(\Omega_0)} \leq C\|f\|_{ L^q(\Omega_0)}
\end{equation}
for $f \in L^q(\Omega)$  with $\supp f\subset\overline{\Omega_{3b}}$. 

{\rm (ii)} For $f \in  L^{q,\infty}(\Omega)$ satisfying  $\supp f\subset\overline{\Omega_{3b}}$ there holds the estimate
\begin{equation}\label{est:6.3.3}
\|({\rm I} + \CR_1)^{-1} f\|_{  L^{q,\infty}(\Omega_0)} \leq C\|f\|_{ L^{q,\infty}(\Omega_0)}.
\end{equation}
\end{lem}


\begin{proof} {\rm (i)} is obtained directly by \cite[Lemma 5.5]{Eiter-Kyed-Shibata23} when extending the $3$-dimensional case to the $n$-dimensional case by using \eqref{spationa-decay-n}.  

{\rm (ii)} 
By the definition of $\CR_1 f$ Lemma \ref{result-stationary-problem-whole-space} and Corollary \ref{result-stationary-problem-A-B}, 
$\CR_1 f \in H^{1; q,\infty}$ and 
${\rm supp}\, \CR_1 f \subset \overline{B'}$, where $B' = B_{3b}\setminus \overline{B_{2b}} $. Then Rellich's compactness
theorem in Lorentz spaces \cite{Cw92} implies that  $\CR_1$ is a compact operator on $L^{q,\infty}(\Omega_{3b})$.  
 Now by Fredholm's alternative it suffices to show that the kernel ${\rm Ker}\, ({\rm I} + \CR_1) = \{0\}$  in order to get that ${\rm I} + \CR_1$ is invertible.  

 Given any $f \in {\rm Ker}({\rm I}+\CR_1)$, 
we will show that $f\equiv 0$.  The assumption $({\rm I}+ \CR_1)f=0$ derives that $f = -\CR_1f $ and ${\rm supp}\,f \subset \overline{B'}$.
Let $u= \CY_0 f$ and $p = \CZ_0 f$. 
Then \eqref{proof:3.2} implies that 
\begin{equation}\label{eq:3.9}
- \Delta u + \nabla p = 0, \quad\dv u=0 \quad\text{in $\Omega_0$},\quad
u|_{\del \Omega_0 }=0,
\end{equation}
where $u \in H^{2;q,\infty}$ and $p \in H^{1;q,\infty}$. 
Let $\psi$ be a $C^\infty_0(\BR^n)$ 
function which equals $1$ for $|x| < 1$ and $0$ for $|x| > 2$ and set
$\psi_{L}(x) = \psi(x/L)$ for $L > 4b$. We note that $L^{q,\infty}=(L^{q',1})^*$ and $\psi_L u \in L^{q',1}$ because 
$$
\|\psi_L u\|_{L^{q',1}(\Omega_{4b})} \leq \|\psi_L \|_{L^{\widetilde{q},1}}\|u\|_{L^{q,\infty}} \leq 
C \|u\|_{L^{q,\infty}}, 
$$
where $\frac{1}{\widetilde{q}}= \frac{1}{q'}-\frac{1}{q}$. Similarly, $\psi_L \nabla u \in L^{q',1}$ and $(\nabla \psi_L )  u \in L^{q',1}$.  Hence 
\eqref{eq:3.9} verifies 
\begin{equation}\label{proof:3.4}
0 = (-\Delta u + \nabla p, \psi_L u) = (\nabla u, \psi_L\nabla u)
+(\nabla u, (\nabla\psi_L) u) - ( p, (\nabla\psi_L)\cdot u).
\end{equation}
Applying Lemma \ref{result-stationary-problem-whole-space} (iv) 
we get the pointwise estimates
$$|u(x)| \leq C|x|^{-(n-2)}, \quad|\nabla u(x)| \leq C|x|^{-(n-1)}, \quad 
\quad 
|p| \leq C|x|^{-(n-1)}
$$
for $|x| > 4b$, and so we obtain  that  
\begin{align*}
|(\nabla u, (\nabla\psi_L) u)| 
& \leq C\|\nabla\psi\|_{L^\infty(\BR^n)}L^{-1}\int_{L \leq |x| \leq 2L} |x|^{-n}
\dx
= O(L^{-1}) \to 0, \\
|( p, (\nabla\psi_L)u)| 
& \leq C\|\nabla\psi\|_{L^\infty(\BR^n)}L^{-1}\int_{L \leq |x| \leq 2L} |x|^{-n}\ dx
= O(L^{-1}) \to 0
\end{align*}
as $L\to \infty$. Taking $L\to\infty$ in \eqref{proof:3.4}, we see that 
$\|\nabla u\|_{L^2(\Omega_0)} = 0$, which implies that $u$ is a constant vector. 
However, since $u|_{\del \Omega_0}=0$, even $u=0$. Therefore, the first equation
of \eqref{eq:3.9} yields $\nabla p=0$, which shows that $p$ is a constant.  
Similarly, 
$p=0$ due to $p(x) = O(|x|^{-(n-1)})$ as $|x| \to \infty$. 

Consequently, 
\eqref{proof:3.1} verifies that 
\begin{equation}\label{proof:3.5} \begin{aligned}
&(1-\varphi)U f_0 + \varphi\CA_0 f_b
+ \BB[(\nabla\varphi)\cdot(U f_0 - \CA_0 f_b)]=0, \\
&(1-\varphi)Q f_0 + \varphi \CB_0 f_b=0
\end{aligned}\end{equation}
in $\Omega_{4b}$. 
Since $\BB[(\nabla\varphi)\cdot(U f_0 - \CA_0 f_b)]$ 
vanishes for $x \not\in B'$ and $\varphi(x) = 0$ for 
$|x| > 3b$ and $1-\varphi(x)=0$ for $|x| < 2b$, we have
\begin{equation}\label{proof:3.6}
\CA_0 f_b= 0, \enskip \CB_0 f_b = 0 \quad\text{for $|x| < 2b$},
\quad U f_0 = 0, \enskip Q f_0 = 0\quad\text{for $|x| > 3b$}.
\end{equation}
Here we consider two functions 
$\zeta(x)$ and $\gamma (x)$ by 
$$\zeta (x) = \begin{cases} (\CA_0 f_b)(x) &\quad\text{for $x \in \Omega_{4b}$}, \\
0 &\quad\text{for $x \not\in \Omega_0$}, 
\end{cases}
\quad 
\gamma (x) = \begin{cases} (\CB_0 f_b)(x) &\quad\text{for $x \in \Omega_{4b}$}, \\
0 &\quad\text{for $x \not\in \Omega_0$}. 
\end{cases}
$$
Then \eqref{est:3.5i} and \eqref{proof:3.6} show that $\zeta  \in H^{2;q,\infty}(B_{4b})$, $\gamma  \in H^{1;q,\infty}(B_{4b})$, and the pair $(\zeta,\,\gamma)$ satisfies the equations
\begin{equation}\label{proof:3.7}
-\Delta \zeta + \nabla \gamma = f_0, \quad \dv \zeta = 0 
\quad\text{in $B_{4b}$}, \quad \zeta|_{\del B_{4b}}=0.
\end{equation}
Moreover, \eqref{proof:3.6} implies that 
$U f_0$ and $Q f_0$ also satisfy \eqref{proof:3.7};  in particular, on $\partial B_{4b}$, we have $U f_0 = - \varphi\CA_0 f_b
- \BB[(\nabla\varphi)\cdot(U f_0 - \CA_0 f_b)]=0$. \BLACK 
This together with the uniqueness of solutions in the bounded domain $B_{4b}$ derive that 
$\zeta = U f_0$ and $\nabla(\gamma -Q f_0) = 0$ in $B_{4b}$. 
Since $\gamma =\CB_0 f_b$  in $\Omega_{4b}$,  by \eqref{eq:3.6} we have 
$\gamma  = Q f_0$.  Especially,  $(\nabla\varphi)\cdot(U f_0-
\CA_0 f_b) = 0$. We thus see from \eqref{proof:3.5} that in $\Omega_{4b}$ \BLACK
\begin{align*}
0 & = U f_0 - \varphi(U f_0 - \CA_0 f_b) = U f_0, \\
0 & = Q f_0 - \varphi(Q f_0 - \CB_0 f_b) = Q f_0,
\end{align*}
which verify that
$f = - \Delta U f_0 + \nabla Q f_0 = 0$ in $\Omega_{4b}$.  Since $\supp f\subset \overline{B'}$, even 
  $f \equiv 0$.    
\end{proof}


Finally, we set 
\begin{align*}
\CU f & := \CY_0({\rm I} + \CR_1)^{-1}f \\
& =
(1-\varphi)U (({\rm I} + \CR_1)^{-1}f)_0
+ \varphi \CA_0(({\rm I} + \CR_1)^{-1}f)_b \\
& \qquad+ 
\BB[(\nabla\varphi)(U(({\rm I} + \CR_1)^{-1}f)_0 
- \CA_0(({\rm I} + \CR_1)^{-1}f))_b],
\\
\CQ f &:= \CZ_0 ({\rm I} + \CR_1)^{-1}f 
= (1-\varphi)  Q (({\rm I} + \CR_1)^{-1}f)_0 
+ \varphi\CB_0(({\rm I} + \CR_1)^{-1}f)_b 
\end{align*}
and derive from \eqref{proof:3.1}, \eqref{proof:3.2} \BLACK
that $(\CU f,\CQ  f)$ is a solution to  the stationary Stokes problem  with right-hand side $f$ on $\Omega_0$.
It directly follows from Lemma \ref{result-stationary-problem-whole-space} {\rm (ii)}, (iii), Corollary \ref{result-stationary-problem-A-B} and Lemma \ref{R1estimate} that $(u,p)$ satisfies the following estimates. 
\vspace{2ex}

\begin{lem}\label{U-Q-estimates}
{\rm (i)} Let $n<q<\infty$ and $f \in L^q(\Omega_0)$ with $\supp f\subset\overline{\Omega_{3b}}$. Then $(u',p') = (\CU f,\CQ  f)$ \BLACK satisfies the estimate
\begin{align}
\begin{aligned}\label{potential-est1-2}
\|u'\|_{H^{2,q}} + \|u'\|_{L^\infty_{n-2}}+\|\nabla u'\|_{L^\infty_{n-1}}+ \|p'\|_{H^{1,q}} +\|p'\|_{L^\infty_{n-1}} \leq 
C  \|f\|_{L^q}. 
\end{aligned}
\end{align}

{\rm (ii)} Let $f \in L^{q_0}(\Omega_0) \cap L^{q,\infty}(\Omega_0)$ satisfy $\supp f\subset\overline{\Omega_{3b}}$
with $n< q_0 <q$.  Then 
\begin{align}
\begin{aligned}\label{potential-est12-2}
\|u'\|_{H^{1,q}} + \|u'\|_{L^\infty_{n-2}}+\|\nabla u'\|_{L^\infty_{n-1}}+ \|p'\|_{L^q}+\|p'\|_{L^\infty_{n-1}} \leq 
C  \|f\|_{L^{q_0}} \BLACK
\end{aligned}
\end{align}
and 
\begin{align}
\begin{aligned}\label{potential-est2-12}
\|\nabla^2 u'\|_{L^{q,\infty}} + \|\nabla p'\|_{L^{q,\infty}} \leq 
C  \|f\|_{L^{q,\infty}}.
\end{aligned}
\end{align} 
\end{lem}

\vspace{2ex}

Now we are in a position to prove Proposition \ref{result-linear-stationary}. 

\vspace{2ex}

\begin{proof}[Proof of Proposition \ref{result-linear-stationary}.]
Suppose that  
$ f =\dv F + g$, $g=g_1+g_2$, where 
$$
F\in L^{\infty}_{n-1}, \ \ \div F \in L^\infty_{n},\;\textrm{ and }\; g_1 \in L^{q}_{\ell},\;\;  g_2 \in  L^{q_0}_\ell \cap L^{q,\infty}.
$$
We rewrite $f$ as $f = \dv((1-\varphi) F)+ h$
with $h = \varphi\,\dv F + (\nabla\varphi)\cdot F + g$,  where $\varphi\in C^\infty_0(\IR^n)$ is a cut off function as above. \BLACK Let  
\begin{align*}
\bar{u} & := (1-\varphi)U  f_0 + \BB[(\nabla\varphi)U f_0],
\quad \bar{p} := (1-\varphi)Q f_0, 
\end{align*}
where $ f_0 = \dv((1-\varphi) F)_{ 0\BLACK} + h_0$.  
Then $\bar{u}$ and $\bar{p}$ satisfy the perturbed Stokes system
\begin{align}\label{perutubed-Stokes}
-\Delta \bar{u} + \nabla \bar{p} = (1-\varphi) f_0 +\CR_2 f, \quad \dv \bar{u} = 0
\quad\text{in $\Omega_0$}, \quad \bar{u}|_{\del \Omega_0} = 0, 
\end{align}
where the operator $\CR_2$ is defined by 
$$\CR_2 f = 2(\nabla\varphi)\cdot\nabla U f_0 + (\Delta\varphi)
U f_0 -\Delta\BB[(\nabla\varphi)\cdot U f_0]
-(\nabla\varphi)Q  f_0.
$$
We note  that ${\rm supp}\,\CR_2 f \subset B_{3b}$.   In addition, we see from Lemma \ref{result-stationary-problem-whole-space} that 
\begin{equation}\label{estimate-R2}
    \|\CR_2 f\|_{L^q(\Omega_0)} \leq 
 C(\mathscr{L}(F) + \mathscr{L}'(g) ).
 \end{equation}
Lemma \ref{result-stationary-problem-whole-space}, \eqref{estimate-R2} and the extension to the whole space imply that the solution  $(\bar u,\bar p)$  
\BLACK to \eqref{perutubed-Stokes} satisfies that 
\begin{align}\label{estimate-perturbed-Stokes}
\|\nabla^2 \bar{u}\|_{L^{q,\infty}}  + \|\bar{u}\|_{L^\infty_{n-2}\cap H^{1,q}} + \|\nabla \bar{u}\|_{L^\infty_{n-1}} + \|\nabla \bar{p}\|_{L^{q,\infty}}+\|\bar{p}\|_{L^\infty_{n-1} \cap L^q} \leq 
C(\mathscr{L}(F) + \mathscr{L}'(g) ). 
\end{align}

Finally define $S_1$ and 
$S_2$ acting on $f$  by 
\begin{align*}
S_1 f& = (1-\varphi) U f_0 
+ \BB[(\nabla\varphi)\cdot U f_0]
+ \CU(\varphi f  - \BLACK \CR_2 f), \\
S_2 f & = (1-\varphi)Q f_0 + \CQ(\varphi f - \BLACK \CR_2 f). 
\end{align*}
Then $u=S_1 f$ and $p = S_2 f$ satisfy \eqref{stationary-problem} by \eqref{perutubed-Stokes}.  
We see from Lemmata \ref{result-stationary-problem-whole-space} and  \ref{U-Q-estimates}, \eqref{estimate-R2}
 and \eqref{estimate-perturbed-Stokes} that $(u,p)$ satisfies the  {\em a priori} estimate \eqref{u-p-est}.
This completes the proof of Proposition \ref{result-linear-stationary}. \BLACK
\end{proof}
\BLACK

\section{Analysis of the Purely Oscillatory Part} 

\vspace{1ex}

For the purely oscillatory part, we use the following key estimate.   

\vspace{1ex}

\begin{prop}  
Let $n\geq 3$, $1<p<q< \infty$ and $-\frac{n}{q} < \ell < \frac{n}{p'}-2$  with $\ell\leq 0$, 
and assume 
$\frac{1}{p}-\frac{1}{q} \leq \frac{2}{n}$.
Let $A = -\mathbb P\Delta$ denote the Stokes operator on the exterior domain $\Omega_0$. Then, for $\varphi\in L^{p,1}_{\ell}(\Omega_0)$, it holds that 
\begin{align} 
\begin{aligned}\label{key-Yamazaki-neg-weight} 
\int_0^\infty t^{\frac{n}{2}\big(\frac{1}{p}-\frac{1}{q}\big)} \|A e^{-tA}\varphi\|_{L^{q,1}_{\ell}(\Omega_0)} \dt \leq  C\|\varphi\|_{L^{p,1}_{\ell}(\Omega_0)}.
\end{aligned}
\end{align}
\end{prop}

\vspace{1ex}

\begin{rem}
{\rm Choosing $\ell=-(n-1)$, the condition 
$-\frac{n}{q} < \ell$ yields $q'>n$, {\em i.e.} $q<n'$. In Proposition \ref{result-second-order-derivative} which is based on a duality argument to \eqref{key-Yamazaki-neg-weight} this reads $q>n$ as wanted.} 
 \end{rem}
\BLACK

\begin{proof}
Following the proof of \cite[Corollary 2.3]{Yamazaki}  let $\delta=\frac{n}{2} \big(\frac{1}{p}-\frac{1}{q}\big)$, choose any $p_1<p<p_2$, let $\delta_j=\frac{n}{2} \big(\frac{1}{p_j}-\frac{1}{q}\big)$ such that $|\delta-\delta_j|<1$, 
and define $v(t)=t^\delta \|Ae^{-tA}\varphi\|_{L^{q,1}_\ell}$. 
By \eqref{A-1-delta} with $p,\delta$ replaced by $p_j,\delta_j$, $j=1,2$, we get that
$$ v(t) \leq C t^{\delta-\delta_j-1}\|\varphi\|_{L^{p_j,1}_\ell} \in L^{s_j,\infty}(0,\infty),\quad \frac{1}{s_j} = 1-(\delta-\delta_j),$$
for $j=1,2$. Note that $0<s_1<1<s_2$. 
Now, by real interpolation in quasi-Banach spaces, see \cite[Chapter 3.11]{BerghL}, and in weighted Lorentz spaces, see \cite[Theorem 2]{Freitag}, there holds 
\begin{align*}
\big(L^{s_1,\infty}(0,\infty),L^{s_2,\infty}(0,\infty)\big)_{\theta,1}& = L^1(0,\infty), \\
\big(L^{p_1,1}_\ell(\Omega_0), L^{p_2,1}_\ell(\Omega_0)\big)_{\theta,1} & = L^{p,1}_\ell(\Omega_0),
\end{align*}
where $1=\frac{1-\theta}{s_1} + \frac{\theta}{s_2}$, $\frac1p=\frac{1-\theta}{p_1} + \frac{\theta}{p_2}$, $0<\theta<1$. 
Then an application to the sublinear operator  $T: L^{p_j,1}_\ell(\Omega_0) \to L^{s_j,\infty}(0,\infty),\; \varphi\mapsto v$, implies that $v\in L^1(0,\infty)$ and
$$ \int_0^\infty t^\delta \|Ae^{-tA}\varphi\|_{L^{q,1}_\ell}\dt =  \|T\varphi\|_{L^1(0,\infty)} 
\leq C\|\varphi\|_{L^{p,1}_\ell}.$$
Thus \eqref{key-Yamazaki-neg-weight} is proved. 
\end{proof} 

\vspace{1ex}

By \eqref{key-Yamazaki-neg-weight}  we obtain the following estimate for the second order derivative of solutions of the purely oscillatory part. In the application the condition \eqref{ass-coefficient-for-Yamazaki-est} below is based on \eqref{assumpt-phi} in Assumption \ref{ass} with $\delta=\mu\in(0,1]$. The crucial part of \eqref{ass-coefficient-for-Yamazaki-est} is the behavior of $a(\cdot,t)$ as $t\to 0+$;  recall \eqref{equ:alpha-beta} and that $a(\cdot,0) =0$. \BLACK

\vspace{1ex}

\begin{prop}\label{result-second-order-derivative} 
Let $F =\mathbb{P} (a f)  \in L^\infty(0,\infty;L^{q_0,\infty}_{\sigma,\ell}(\Omega_0))$,  where  
$f \in L^\infty(0,\infty;L^{q,\infty}_{\ell}(\Omega_0))$ and $a \in L^\infty(0,\infty;L^{q_1,1}_{\ell}(\Omega_0))$ with 
$\frac{1}{q_0} = \frac{1}{q_1} + \frac{1}{q}$ where $\frac{nq}{q-n}< q_1$, $n <q_0 < q<\infty$ 
and $\frac{1}{q_0} - \frac{1}{q} \leq \frac{2}{n}$ 
with $\max(0,2-\frac{n}{q}) \leq \ell< \frac{n}{ q'_0}$.  
Moreover, assume that $a=a(t, x)$ satisfies the decay estimate 
\begin{align}
\begin{aligned}\label{ass-coefficient-for-Yamazaki-est}
\|a(t,\cdot) \|_{L^{q_1,1}_{\ell}}\leq ct^{\delta},\quad t>0, 
\end{aligned}
\end{align}
where $\delta=\frac{n}{2}\Big(\frac{1}{q'}-\frac{1}{q_0'}\Big)  = \frac{n}{2}\Big(\frac{1}{q_0}-\frac{1}{q}\Big)\in (0,1]$. 
 \BLACK  Then, for any $T \in (0, \infty]$,  
\begin{align}\label{int_eAf}
\Big\|\nabla^2 \int_0^T  e^{-\tau A} F(\tau) d \tau \Big\|_{ L^{q,\infty}_{\ell}(\Omega_0)} \leq C\|f\|_{L^\infty(0,\infty; L^{q,\infty}_{\ell}(\Omega_0))}
\end{align} 
with a constant $C>0$ independent of $T>0$.
\end{prop}

\begin{proof} 
For a solenoidal smooth vector field $v$ on $\Omega_0$ we note that 
\begin{align}\label{Deltau}
\begin{aligned}
\|\nabla^2 v\|_{L^q_{\ell} } 
\leq C\|\Delta v\|_{L^q_{\ell} }.
\end{aligned}
\end{align} 
which, {\em e.g.}, follows by the resolvent estimate \eqref{equ:rse-w2}  in the limit $\lambda\to 0$ \BLACK with $n<q$ satisfying the condition
$2-\frac{n}{q} < \ell < \frac{n}{q_0'}$.  Then real interpolation verifies that 
\begin{align}\label{Deltau2}
\begin{aligned}
\|\nabla^2 v\|_{L^{q,\infty}_{\ell}} 
\leq C\|\Delta v\|_{L^{q,\infty}_{\ell}},
\end{aligned}
\end{align} 
and we thus see that 
\begin{align}
\begin{aligned}
\Big\|\nabla^2 \int_0^T  e^{-\tau A} F(\tau)  \dtau \Big\|_{ L^{q,\infty}_{\ell}} & \leq 
C\Big\|\Delta \int_0^T  e^{-\tau A} F(\tau)  \dtau \Big\|_{ L^{q,\infty}_{\ell}},
\end{aligned}
\end{align} 
provided the integral term is well-defined under $\Delta$. To this aim \BLACK  we use a typical duality argument. 
For  $\varphi \in C^\infty_{0,\sigma}(\Omega_0)\subset L^{q',1}_{\sigma, -\ell}(\Omega_0) $ \eqref{key-Yamazaki-neg-weight} with $q$ replaced by $q'$ verifies
\begin{align}\label{int-Delta-phi2}
\begin{aligned}
\Big|\int_0^T \langle \Delta e^{-\tau A}F(\tau) , \varphi\rangle  \dtau \Big| 
& = \Big|\int_0^T \langle  af(\tau),  
e^{-\tau A} \mathbb P_{q'} \Delta \varphi\rangle  \dtau \Big|\\
& \leq C\int_0^T \|a(\tau)f (\tau)\|_{L^{q_0,\infty}_{\ell}} \| Ae^{-\tau A}  \varphi\|_{L^{q_0',1}_{-\ell}}\dtau \\
& \leq C\int_0^T \tau^\delta \|f(\tau)\|_{L^{q,\infty}_{\ell}} \| Ae^{-\tau A}  \varphi\|_{L^{q_0',1}_{-\ell}}\dtau \\
& \leq C\|f\|_{L^\infty(0,\infty;L^{q,\infty}_\ell} \int_0^T \tau^\delta\| Ae^{-\tau A}  \varphi\|_{L^{q_0',1}_{-\ell}}\dtau\\[1ex]
&  \leq C \|f\|_{L^\infty(0,\infty;L^{q,\infty}_{\ell})}
\|\varphi\|_{L^{q',1}_{-\ell}};\BLACK
\end{aligned}
\end{align}
\BLACK for the last step based on \eqref{key-Yamazaki-neg-weight} note that $q>n$ satisfies the inequalities $-\frac{n}{q_0'} < -\ell < \frac{n}{(q')'}  -2 = \frac{n}{q} -2$. 
 Since $C^\infty_{0,\sigma}(\Omega_0) \subset  L^{q,1}_{\sigma, -(n-1)}(\Omega_0)\BLACK$ is dense (\cite{FS}), a duality argument in \eqref{int-Delta-phi2} implies \eqref{int_eAf}. 
\BLACK
\end{proof}

\vspace{2ex}

We recall the mapping $H[u,p]=H_1[u,p]+H_2[u,p]$, see $\eqref{equ:ns-us-per}_2$,  \eqref{H_1}, \BLACK defined by  
\BLACK
\begin{align}\label{def-H_1}
    \begin{aligned}
H_1[u,p](t) & = \Big(\sum_{k=1}^\infty e^{-kT A}\Big) \int_0^T e^{-\tau A} \mathbb P N(u, p)(t-\tau) \dtau \\ 
& = \Big(\sum_{k=1}^\infty (-A) e^{-kT A}\Big) \int_0^T \Big(\int_0^\tau e^{-sA} \ds\Big) \mathbb P N(u, p)(t-\tau) \dtau 
\end{aligned}
\end{align}
and when $k=0$, 
$$
H_2[u,p](t)=  \int_0^T e^{-\tau A} \mathbb P N(u, p)(t-\tau) \dtau. 
$$ 
For simplicity, we write $H=H_1+H_2$ 
where we will consider instead of $\mathbb P N(u, p)(t-\tau)$ a function $f(\cdot) \in  L^\infty(0,\infty;L^{q,\infty}_{\sigma,n-1}(\Omega_0))$ and $H(f)$ instead of $H[u,p]$. Note that the condition $P_s N(u,p)=0$ is not needed when considering $H$ in the form of \eqref{u-perp-int}.

\vspace{2ex}

By virtue of Proposition \ref{result-second-order-derivative} we will estimate second order derivatives $\nabla^2 H(f)$ as follows:\BLACK 

\begin{prop}\label{result-oscillatory-part-2nd-derivative} 
\begin{enumerate}
    \item[{\rm (i)}]
Under the assumptions of Proposition \ref{result-second-order-derivative} on $a$ and $f$ there holds for $F=\mathbb P(af)$, $0<t<T$ and finite $T>0$ \BLACK
%
%
\begin{align}\label{nabla2_H}
\big\|\nabla^2 HF (t)\BLACK \big\|_{ L^{q,\infty}_{n-1}(\Omega_0)} & \leq C \|f\|_{L^\infty(0,\infty;L^{q,\infty}_{n-1}(\Omega_0))} 
\end{align} 
uniformly in $T$. 
\item[{\rm (ii)}] 
Let $f \in L^\infty(0,\infty;H^{1;q_2, \infty}_{ \sigma \BLACK,n-1}(\Omega_0))$ where $n< q_2 <q< \frac{n+1}{n} q_2$. Then there exists a constant $C_T>0$ such that for $0<t<T<\infty$\BLACK   
\begin{equation}\label{nabla2_H2}
\big\|\nabla^2 Hf(t) \BLACK\big\|_{ L^{q,\infty}_{n-1}(\Omega_0)} \leq C_T \|f\|_{L^\infty(0,\infty;H^{1;q_2, \infty}_{n-1}(\Omega_0))}. 
\end{equation}  
\end{enumerate}
\end{prop}


\begin{proof}
{\rm (i)} With $HF = \int_0^T e^{-\tau A} F(T-\tau) \dtau$ and $\ell=n-1$ Proposition \ref{result-second-order-derivative} directly yields (i).

{\rm (ii)} 
Let $0< \theta < \frac{1}{2q}$. Then for any $f\in L^\infty(0,\infty;H^{1;q_2, \infty}_{\sigma,n-1}(\Omega_0))$, 
the classical decay estimate $\|A^{1-\theta} e^{-t A}g\|_{L^q_\ell} \leq Ct^{-(1-\theta)} \|e^{-t A/2}g\|_{L^q_\ell}$ and Theorem \ref{weightedLpLq^decay-estimates} (i), $\alpha=0$, 
imply with $q_3$ defined by $\frac{1}{q_3} = \frac{1}{q_2}- \frac{1}{q}$ and $\ell=n-1$ satisfying $2-\frac{n}{q}< \ell < \frac{n}{q'}$
\begin{align*}
\|\nabla^2 e^{-t A}f\|_{L^q_\ell} & \leq C
\|A  e^{-t A}f\|_{L^q_\ell} = C 
\|A^{1-\theta} e^{-t A}A^{\theta}f\|_{L^q_\ell} \\
& \leq C t^{-(1-\theta)} \|e^{-tA/2} A^{\theta}f\|_{L^q_\ell} \\
& \leq C t^{-(1-\theta)-\frac{n}{2q_3}} \| A^{\theta}f\|_{L^{q_2}_\ell}.
\end{align*}
%
%
We note that by Proposition \ref{fractional-power-without-boundary-condition},  for $0< \theta < \frac{1}{2q}$  without Dirichlet boundary condition,  
$$
\D(A^\theta)=[L^q_{\sigma,\ell}, \D(A)]_{\theta} = [L^q_\ell, \D(-\Delta)]_{\theta} \cap L^q_\sigma = H^{2\theta,q}_{\sigma,\ell}(\Omega_0)
$$
and thus 
$\|A^{\theta}f\|_{L^{q_2}_\ell} \leq C\|f\|_{H^{2\theta,q_2}_{\ell}(\Omega_0)} \leq C\|f\|_{H^{1,q_2}_{\ell}(\Omega_0)}.$
Then by real interpolation the above estimate can be extended to 
$$ \|\nabla^2 e^{-t A}f\|_{L^{q,\infty}_\ell} \leq C t^{-(1-\theta)-\frac{n}{2q_3}} \| f\|_{L^\infty(0, \infty; {H}^{1;q_2,\infty}_{\ell}(\Omega_0))} \quad \textrm{for a.e. } t\in J. $$

Since $n< q_2 <q< \frac{n+1}{n} q_2$, we define $q_3\in (q_2,\infty)\BLACK$ by $\frac{1}{q_2}=\frac{1}{q_3}+ \frac{1}{q} $ and see that $\frac{n}{2 q_3} +\big(1-\frac{1}{2 q}\big)<1$. 
Hence there exists $0<\theta < \frac{1}{2q}$  such that  
$$
\frac{n}{2q_3} +(1-\theta)<1 $$
and consequently $C_T' := \int_0^T t^{-(1-\theta)-\frac{n}{2q_3}} \dt <\infty$. 
Thus the time integral on $(0,T)$ of the above estimate implies with $\ell=n-1$  
\begin{align}\label{Ae2-consequence}
\begin{aligned}
\|\nabla^2 H f(t)\|_{ L^{q,\infty}_{n-1}} & \leq C \int_0^T \tau^{-\frac{n}{2q_3}-(1-\theta)} \|A^{\theta} f(\tau) \|_{ L^{q_2,\infty}_{\ell} } \dtau \\[1ex]
& \leq C_T \| f  \|_{L^\infty(0, \infty; {H}^{1;q_2,\infty}_{\ell}(\Omega_0))}. 
\end{aligned}
\end{align}
\end{proof}

\BLACK

For $H$ and $\nabla H$ 
we get the following estimates. 

\vspace{1ex}

\begin{prop}\label{result-oscillatory-part-0-1-order} {\rm (i)} 
For $f \in L^\infty(0,\infty;L^{q}_{\sigma,n-1+\delta}(\Omega_0))$  with $0<\delta<1-\frac{n}{q}$ there holds for any $0<t<T<\infty$ the estimate \BLACK
$$
\|Hf\|_{ L^q_{n-1}(\Omega_0)\BLACK } \leq 
C_T\|f\|_{L^\infty(0,\infty; L^{ q}_{n-1+\delta}(\Omega_0))} 
$$
with a constant $C_T$ blowing up as $T\to\infty$.

{\rm (ii)} 
For $f \in L^\infty(0,\infty;L^{q}_{\sigma,n-1}(\Omega_0))$  and $q>n$
we have that 
$$
\|\nabla Hf(t)\|_{ L^q_{n-1}(\Omega_0)\BLACK } \leq 
C_T\|f\|_{L^\infty(0,\infty; L^{q}_{n-1}(\Omega_0))}.
$$

{\rm (iii)} 
For  $f \in L^\infty(0,\infty;L^{q_0 }_{\sigma,n-1}(\Omega_0))$, 
where $n < q_0 <q$, the slightly different estimates  
\begin{align*}
\|Hf(t)\|_{ L^q_{n-1}(\Omega_0)} + \|\nabla Hf(t)\|_{ L^q_{n-1}(\Omega_0)} 
& \leq 
C_T\|f\|_{L^\infty(0,\infty; L^{ q_0}_{n-1}(\Omega_0))}
\end{align*}
hold for all $0<t<T<\infty$. \BLACK
\end{prop}


\begin{proof} 
We apply weighted $L^q$-$L^p$ decay estimates of the Stokes semigroup. 

(i) Theorem \ref{weightedLpLq^decay-estimates} (ii) for $H_1$ and $H_2$ with taking $t_0=T>0$ in Remark \ref{constant} imply that 
\begin{align}\label{u_perp-est0-H1H2}
\begin{aligned}
\|H_1f\|_{ L^q_{n-1}(\Omega_0)} 
& = \Big\|\Big(\sum_{k=1}^\infty (-A) e^{-kT A}\Big) \int_0^T \Big(\int_0^\tau e^{-sA} \ds\Big) f(\tau) \dtau \Big\|_{ L^q_{n-1}(\Omega_0)}\\ 
& \leq C_T\sum_{k=1}^\infty (kT)^{-1-\frac{\delta}{2}} \, T^2\BLACK \,\|f\|_{ L^\infty(0,\infty; L^{q}_{n-1+\delta}(\Omega_0))} \\
& \leq 
C_T \|f\|_{L^\infty(0,\infty;L^{ q}_{n-1+\delta}(\Omega_0))}, \\
\|H_2f\|_{ L^q_{n-1}(\Omega_0)} 
&  = \Big\| \int_0^T e^{-\tau A} f(\tau) \dtau \Big\|_{ L^q_{n-1}(\Omega_0)}\\ 
& \leq 
C_T \|f\|_{L^\infty(0,\infty;L^{ q}_{n-1}(\Omega_0))}. 
\end{aligned}
\end{align}

(ii) If $f \in L^\infty(0,\infty; L^{q}_{\sigma,n-1}(\Omega_0))$, then by Theorem \ref{A1/2-nabla_A1/2-all} (ii)   
\begin{align}\label{u_perp-est1-H1H2-case1}
\|\nabla H_1f\|_{ L^q_{n-1}(\Omega_0)} 
 &  = \BLACK \Big\|\Big(\sum_{k=1}^\infty \nabla  e^{-kT A} (-A)\Big) \int_0^T \Big(\int_0^\tau e^{-sA} \ds\Big) f(\tau) \dtau \Big\|_{ L^q_{n-1}(\Omega_0) }\nonumber\\ 
& \leq C_T \sum_{k=1}^\infty  \big\{(kT)^{-1-\frac{1}{2}}+(kT)^{-1-\frac{n}{2q}-\frac{n-1}{2}} \big\}  \,T^2\,\BLACK \|f\|_{ L^\infty(0,\infty; L^{q}_{n-1}(\Omega_0)) } \nonumber\\
& \leq 
C_T \|f\|_{L^\infty(0,\infty; L^{q}_{n-1}(\Omega_0))},  \\
\|\nabla H_2f\|_{ L^q_{n-1}(\Omega_0)} 
& \leq C_T \int_0^T  \tau^{-\frac{1}{2}}\|f\|_{ L^\infty(0,\infty; L^{q}_{n-1}(\Omega_0)) } \dtau\nonumber\\
& \leq 
C_T \|f\|_{L^\infty(0,\infty; L^{q}_{n-1}(\Omega_0))}.\nonumber 
\end{align}

(iii) On the other hand, if $f \in L^\infty(0,\infty; L^{q_0}_{ \sigma,n-1}(\Omega_0))$, we use Theorem \ref{A1/2-nabla_A1/2-all} (ii) and Corollary \ref{Lp-Lq At-all}  with $\ell=\ell'=n-1$ and $\kappa=\frac{n}{q_0} - \frac{n}{q}$ which imply that  
\begin{align}\label{u_perp-est1-H1H2}
\begin{aligned}
\|H_1f\|_{L^q_{n-1}(\Omega_0)} 
&  = \BLACK \Big\|\Big(\sum_{k=1}^\infty (-A) e^{-kT A}\Big) \int_0^T \Big(\int_0^\tau e^{-sA} \ds\Big) f(\tau) \dtau \Big\|_{ L^q_{n-1}(\Omega_0) }\\ 
& \leq C_T \sum_{k=1}^\infty (kT)^{-1-\frac{n}{2}\big(\frac{1}{q_0}-\frac{1}{q}\big) }  \,T^2\,\BLACK \|f\|_{ L^\infty(0,\infty; L^{q_0 }_{n-1}(\Omega_0))} \\
& \leq 
C_T\|f\|_{L^\infty(0,\infty; L^{ q_0 }_{n-1}(\Omega_0))},  \\
\|H_2f\|_{ L^q_{n-1}(\Omega_0)} 
& \leq C_T \int_0^T  \tau^{-\frac{n}{2}\big(\frac{1}{q_0}-\frac{1}{q}\big)}\|f\|_{ L^\infty(0,\infty; L^{ q_0}_{n-1}(\Omega_0)) } \dtau\\
& \leq 
C_T\|f\|_{L^\infty(0,\infty; L^{ q_0 }_{n-1}(\Omega_0))}.
\end{aligned}
\end{align}
%
%
Similarly, 
\begin{align}\label{u_perp-est1-H1H2-case2}
\|\nabla H_1f\|_{ L^q_{n-1}(\Omega_0)} 
 &  = \Big\|\Big(\sum_{k=1}^\infty \nabla  e^{-kT A} (-A)\Big) \int_0^T \Big(\int_0^\tau e^{-sA} \ds\Big) f(\tau) \dtau \Big\|_{ L^q_{n-1}(\Omega_0) }\nonumber \\ 
& \leq C_T\sum_{k=1}^\infty  (kT)^{-\frac32-\frac{n}{2}\big(\frac{1}{q_0}-\frac{1}{q}\big)} 
 \, T^2\, \|f\|_{ L^\infty(0,\infty; L^{q_0}_{n-1}(\Omega_0)) } \nonumber\\ 
& \leq 
C_T\|f\|_{L^\infty(0,\infty; L^{q_0}_{n-1}(\Omega_0))},  
\\[1ex]
\|\nabla H_2f\|_{ L^q_{n-1}(\Omega_0)\BLACK } 
& \leq C_T \int_0^T \tau^{-\frac{n}{2}\big(\frac{1}{q_0}-\frac{1}{q}\big)-\frac{1}{2}}\|f\|_{ L^\infty(0,\infty; L^{q_0}_{n-1}(\Omega_0)) } \dtau \ \nonumber \\
& \leq 
C_T\|f\|_{L^\infty(0,\infty; L^{q_0}_{n-1}(\Omega_0))}. \nonumber
\end{align}\vspace*{-8mm}
\end{proof}

\section{Nonlinear Estimates}

In this section we present estimates of nonlinear terms and finally prove the existence of time periodic solutions by Banach's fixed point argument. 

\vspace{1ex}

Recalling the nonlinear term $K(u,p)$, see $\eqref{KNM}_1$, we also define 
\begin{align*}
 M_0(u) & =  a_0 :u + a_1 :\nabla u +  b_0:u + b_1:\nabla u,\\
M_1(u,p) & = a_2 :\nabla^2  u +  (\nabla^{\phi(0)}-\nabla^{\phi(t)})p,\\
K_1 (u) & = - P_s (u_s \cdot \nabla u_s) = - u_s \cdot \nabla u_s,\\
K_2(u,p) & = K(u,p) - K_1 (u)  - P_s (M_0(u))\BLACK - P_s (M_1(u,p)),  
\end{align*} 
where $a_2:\nabla ^2 u$ is a short notation for all terms $a_\alpha\partial^\alpha u$ with $|\alpha|=2$. Similarly, we define $a_1:\nabla u$, $a_0:u$, $b_1:\nabla u$ and $b_0:u$.   

With the power $q_1$ given in Assumption \ref{ass} {\rm (iv)} we find $n<q_0<q$ satisfying  
$$
\frac{1}{q_0}=\frac{1}{q_1} + \frac{1}{q}.
$$ 
Set 
\begin{align*}
 \mathscr{L}_\mathscr{K} (K(u,p)) & = \|\div (u_s \otimes u_s) \|_{L^\infty_{n}} +
\|u_s \otimes\BLACK  u_s \|_{L^\infty_{n-1}}+ \|K_2(u)\|_{L^q_n}\\[1ex] 
& \quad +  \|P_s M_0(u)\|_{L^{q_0}_{n}\cap L^{q, \infty}_{n-1}} \BLACK + 
\|P_s  M_1(u, p)\|_{L^{q_0}_{n}\cap L^{q, \infty}_{n-1}}  + \|f_s\|_{L^q_n}. \BLACK
\end{align*}
Similarly, let 
\BLACK
$$ N_1(u,p) = N(u, p) - P_\perp M(u) -  P_\perp\big(\nabla^{\phi(0)}-\nabla^{\phi(t)}\big) p$$ 
and set 
\begin{align*}
 \mathscr{L}_\mathscr{N} (N(u, p))\BLACK & = \|N_1(u,p) \|_{L^\infty(\IR;L^{q}_{n-1+\delta} \cap L^q_{n-1} \BLACK)} + \|P_\perp  M(u)\|_{ L^\infty (\IR;L^{q_0}_{ n} \cap L^{q,\infty}_{n-1})} \\
& \quad + \big\| P_\perp \big(\nabla^{\phi(t)}-\nabla^{\phi(0)}\big)p\big\|_{ L^\infty(\IR;L^{q_0}_{n} \cap L^{q,\infty}_{n-1})} +  \|f_\perp\|_{L^\infty(\IR:\L^q_{n-1+\delta})}. \BLACK
\end{align*}
Note that the control terms $\mathscr{L}_\mathscr{K}$ and $\mathscr{L}_\mathscr{N}$ do not define norms since the decomposition of $K(u,p)$ and $N(u,p)$ into a sum of terms as in \eqref{KNM} is by far not unique; see also Remark \ref{rem-norms-f} (i) below. 

Concerning the external force field $f=f_s+f_\perp$ we assume for simplicity that $f_s(=g)\in L^q_n(\Omega_0)$
and $f_\perp \in L^\infty_{per}(\IR;L^q_{n-1+\delta}(\Omega_0))$ 
so that Proposition \ref{result-oscillatory-part-0-1-order} (i), (ii) and Proposition \ref{result-oscillatory-part-2nd-derivative} for the instationary part and Proposition \ref{result-linear-stationary} for $f_s$ can be applied. For further results see Remark \ref{rem-norms-f} (i).
\vspace{1ex}

\begin{prop}\label{est-K-N}
For any period $T>0$ there exists a constant $C_T>0$ \BLACK such that 
\begin{align*}
 \mathscr{L}_\mathscr{K} & (K(u,p))  + \mathscr{L}_\mathscr{N} (N(u,p)) \\
 & \leq C_T\big(\|(u_s, p_s)\|_{X_s}^2 +  \|(u_\perp, p_\perp)\|_{X_\perp}^2 \big) + \|f_s\|_{L^q_{n}} + \|f_\perp\|_{L^\infty (\IR;L^q_{n-1+\delta}) 
 }\\
  &  + \epsilon C_T\big(\|u\|_{L^\infty (\IR;H^{1,q})} + \|\nabla^2 u\|_{L^\infty (\IR;L^{q,\infty})} + \|\nabla p\|_{L^\infty (\IR;L^{q,\infty})} \big). \BLACK
\end{align*}
\end{prop}

\begin{proof}
By Assumption \ref{ass} (iv), \eqref{equ:alpha-beta=0}, and the H\"older inequality in Lorentz spaces we see that for $M_1(u, p)$, omitting $P_s$,   
\begin{align}
\begin{aligned}\label{pressure-est}
\big\|\big(\nabla^{\phi(t)}-\nabla^{\phi(0)}\big)p\big\|_{L^{q_0}_{n}} + \big\|\big(\nabla^{\phi(t)}-\nabla^{\phi(0)}\big)p\big\|_{L^{q,\infty}_{n-1}}
& \leq \epsilon C_T \BLACK \|\nabla p\|_{L^{q,\infty}}, \\
\|a_2 \nabla^2  u\|_{L^{q_0}_{n}}+ \|a_2 \nabla^2  u\|_{L^{q, \infty}_{n-1}}
& \leq  \epsilon C_T \|\nabla^2  u\|_{L^{q,\infty}}.  
 \end{aligned}
\end{align} 
where $\mu=\delta=\frac{n}{2}\big(\frac{1}{q_0}-\frac{1}{q}\big)$ and we used the fact $a_2=0$ at $t=0$. Here $\epsilon>0$ is a constant depending on $c_\phi$ in \eqref{phi-mu}, \eqref{assumpt-phi} and on $T$, and can be chosen suitably small if $c_\phi$ is sufficiently small.   \BLACK 
Concerning similar estimates for $a_0 : u$, $a_1 : \nabla u$, $b_1:\nabla u$ and $b_0:u$, we refer to \eqref{a0a1} below.
These estimates imply that 
\begin{align}\label{M_1stat}
\|M_1(u,p)(t)\|_{L^{q,\infty}_{n-1} \cap L^{q_0}_n} & \leq \epsilon C_T \|\nabla^2 u,\nabla p\|_{L^\infty_{per}(\IR;L^{q,\infty})}\quad \textrm{for a.a. } t\in J,\\
\label{M_1-final}
\|M_1(u,p)\|_{L^\infty_{per}(\IR;L^{q,\infty}_{n-1} \cap L^{q_0}_n)} & \leq \epsilon C_T \|\nabla^2 u,\nabla p\|_{L^\infty_{per}(\IR;L^{q,\infty})}. 
\end{align}  
Later, \eqref{M_1stat} is used in Proposition \ref{result-linear-stationary} for a term of type $g_2$, and \eqref{M_1-final} is exploited in \eqref{nabla2_H} to estimate $\nabla^2 H$ in $X_\perp$. 

Concerning $a_\alpha\partial^\alpha u$ with $|\alpha|\leq 1$ and $b_\beta\partial^\beta$ we get from \eqref{assumpt-phi}, \eqref{equ:alpha-beta=0} that
\begin{align}\label{a0a1}
\|M_0(u)(t)\|_{L^{q_0}_n \cap L^{q,\infty}_{n-1}} & \leq C\epsilon \big(\|u(t)\|_{L^q} + \|\nabla u(t)\|_{L^q}\big) \quad t-a.e.
\end{align}
Applying $P_s$ or $P_\perp$, the right hand side must be replaced $C\epsilon \|u(t)\|_{L^\infty_{per}(H^{1,q})}$. 
This estimate is used together with \eqref{M_1-final} in Proposition \ref{result-oscillatory-part-0-1-order} (iii) and  Proposition \ref{result-oscillatory-part-2nd-derivative}.\BLACK

Further terms are estimated directly. In particular, for the stationary part, omitting the bounded projection $P_s$ and taking into account that $n\geq 3$, 
\begin{align}
\begin{aligned}
\|u_s \cdot \nabla u_s \|_{L^\infty_{n}} & = \|\div (u_s \otimes u_s) \|_{L^\infty_{n}} \leq C\|u_s\|_{L^\infty_{n-2}}\|\nabla u_s\|_{L^\infty_{n-1}}, \\
\|u_s \otimes  u_s \|_{L^\infty_{n-1}} & \leq C\|u_s\|_{L^\infty_{n-2}}\| u_s\|_{L^\infty_{n-2}}, \\
\big\|u_\perp \cdot \nabla^{\phi(t)} u_\perp \big\|_{L^q_{n}} & \leq C \|u_\perp \|_{L^\infty_{n-2}} \|\nabla u_\perp \|_{L^q_{n-1}}\leq C \|u_\perp \|_{H^{1,q}_{n-2}} \BLACK \|\nabla u_\perp \|_{L^q_{n-1}} , \\
\big\|u_\perp \cdot \nabla^{\phi(t)} u_s \big\|_{L^q_{n}} & \leq C\|u_\perp \|_{L^q_{n-2}}\|\nabla u_s \|_{L^\infty_{n-1}}, \\
\big\|u_s \cdot \nabla^{\phi(t)} u_\perp \big\|_{L^q_{n}} & \leq C\|u_s \|_{L^\infty_{n-2}} \|\nabla u_\perp \|_{L^q_{n-}},\\ 
\big\|u_s \cdot \big(\nabla^{\phi(t)} - \nabla^{\phi(0)}\big) u_s \big\|_{L^q_{n}} & \leq C\epsilon\|u_s \|_{L^q}\|\nabla u_s \|_{L^\infty_{n-1}}. 
\end{aligned}
\end{align}
In the third line we applied the embedding $H^{1,q}(\IR^n)\subset L^\infty$, $n<q$,  which is easily adapted to the weighted setting. 
These estimates are helpful in Proposition 3.1 on the stationary Stokes system with right-hand sides $\div F$, where $F=u_s\otimes u_s$, and $g_1$, $\ell=n$. 

For the instationary equation we will exploit the following estimates where in most cases we omitted $P_\perp$ for simplicity: 
We note that, since $P_\perp (u_s \cdot \nabla u_s)=0$,   
\begin{align}\label{P_perp-terms-est}
\big\|P_\perp \big(u_s  \cdot \nabla^{\phi(t)}u_s\big)\big\|_{L^{{q}}_{n-1}} 
& \leq   2\big\|u_s  \cdot\big(\nabla^{\phi(t)}-\nabla^{\phi(0)}\big) u_s\big\|_{ L^\infty_{per}(\IR;\BLACK L^{{q}}_{n-1})} \leq C \|u_s\|_{L^\infty}\|\nabla u_s\|_{L^{q}},\nonumber \\
\BLACK\big\|P_\perp \big(u_s  \cdot \nabla^{\phi(t)}u_s\big)\big\|_{L^{q}_{n-1+\delta}} 
& \BLACK\leq  2\big\|u_s  \cdot \big(\nabla^{\phi(t)}-\nabla^{\phi(0)}\big) u_s\big\|_{ L^\infty_{per}(\IR;\BLACK L^{q}_{n-1+\delta})} \leq C \|u_s\|_{L^\infty_{n-2}}\|\nabla u_s\|_{L^{q}}, \nonumber \\
\BLACK\big\|u_s  \cdot \nabla^{\phi(t)}  u_\perp\big\|_{L^{q}_{n-1+\delta}} & \BLACK\leq C \|\nabla u_\perp\|_{L^q_{n-1}}\| u_s\|_{L^{\infty}_{n-2}},\nonumber \\ 
\BLACK\|u_\perp  \nabla^{\phi(t)} \cdot u_s\|_{L^{q}_{n-1+\delta}} & \BLACK\leq C \|u_\perp\|_{L^q_{n-1}} \|\nabla u_s\|_{L^{\infty}_{n-1}},\\  
\BLACK\big\|u_\perp  \cdot \nabla^{\phi(t)}u_{\perp}\big\|_{L^{q}_{n-1+\delta}} & \BLACK\leq C  \|u_\perp\|_{L^\infty_{n-1}}  \|\nabla u_\perp\|_{L^{q}_{n-1}}\leq C\|u_{\perp}\|_{H^{1,q}_{n-1}}^2. \nonumber
\end{align}
Obviously, analogous terms in the norm of $L^q_{n-1}$ admit the same estimates. 

Norms of product terms estimated in $H^{1;q_2,\infty}_{n-1}(\Omega_0)$ need a more careful analysis. 
Given $n<q_2<q < \frac{n+1}{n}q_2$, as in the proof of Proposition \ref{result-oscillatory-part-2nd-derivative} (ii), we write $\frac{1}{q_2} = \frac{1}{q_3} + \frac{1}{q}$ where $n<q< q_3$ and $\frac{1}{q} - \frac{1}{q_3} < \frac{1}{n}$. 
Taking into account that partial derivatives of $\phi(\xi,t)$ are uniformly bounded by Assumption \ref{ass} (ii), \BLACK we obtain that
$$\big\|u_s  \cdot \nabla^{\phi(t)}u_{\perp}\big\|_{H^{1;q_2,\infty}_{n-1}} \leq C\big(\|u_s\|_{L^{q_3}}\|\nabla u_{\perp}\|_{H^{1;q,\infty}_{n-1}} + \|\nabla u_s\|_{L^{q_3,\infty}} \|\nabla u_{\perp}\|_{L^{q}_{n-1}}\big),$$ 
where the term $\|u_s\|_{L^{q_3}}$ is estimated by the Gagliardo-Nirenberg inequality 
on $\Omega_0$, {\em i.e.,} with $\delta=n\big(\frac{1}{q}-\frac{1}{q_3}\big)\in (0,1)$ there holds
\begin{equation}\label{Lq_3}  
\|u_s\|_{L^{q_3}} \leq C\big(\|\nabla u_s\|_{L^q}^\delta \|u_s\|_{L^q}^{1-\delta} + \|u_s\|_{L^q} 
\big) \leq C\|u_s\|_{H^{1,q}}. 
\end{equation}
By analogy, real interpolation and the embedding $L^q\subset L^{q,\infty}$ imply that 
\begin{equation}\label{Lq_3-infty}
\|\nabla u_s\|_{L^{q_3,\infty}} \leq C 
\|\nabla u_s\|_{H^{1; q,\infty\BLACK}} \leq C\big( \|\nabla^2 u_s\|_{L^{q, \infty\BLACK}} + \|\nabla u_s\|_{L^{q}}\big) . 
\end{equation}
Therefore, we obtain that
\begin{align}\label{u_su_p-H}
\big\|u_s & \cdot \nabla^{\phi(t)}u_{\perp}\big\|_{H^{1;q_2, \infty}_{n-1}} \nonumber\\ 
& \leq C  
\big(\|u_s\|_{H^{1,q}} \|\nabla^2 u_{\perp}\|_{L^{q, \infty}_{n-1}} 
  + \big(\|u_s\|_{H^{1,q}} + \|\nabla^2 u_s\|_{L^{q,\infty}}\big) \|\nabla u_{\perp}\|_{L^{q}_{n-1}}\big).
\end{align}

%
As for the product term $\|u_\perp  \cdot \nabla^{\phi(t)}u_{s}\big\|_{H^{1;q_2, \infty}_{n-1}}$, Proposition \ref{DS}
is exploited for $u_\perp$  yielding the upper bound $\|u_\perp\|_{H^{1,q}_{n-1}}$, {\em cf.}  \eqref{Lq_3} without weights. Then, together with \eqref{Lq_3-infty} applied to $\nabla u_s$, we get that
\begin{align}\label{u_pu_s-H}
\big\|u_\perp  \cdot \nabla^{\phi(t)}u_{s}\big\|_{H^{1;q_2, \infty}_{n-1}} 
& \leq C\big(\|u_\perp\|_{L^{q_3}_{n-1}} \|\nabla u_s\|_{H^{1;q,\infty}} + \|\nabla u_s\|_{L^{q_3,\infty}} \|\nabla u_\perp\|_{L^{q}_{n-1}}\big) \nonumber \\ 
& \leq C \|u_\perp\|_{H^{1,q}_{n-1}}
\big(\|\nabla^2 u_s\|_{L^{q, \infty}} + \|\nabla u_s\|_{L^{q}} \big) .\BLACK  
\end{align}
%
By analogy, 
there holds
\begin{align}\label{u_pu_p-H}
\big\|u_\perp  \cdot \nabla^{\phi(t)}u_{\perp}\big\|_{H^{1;q_2, \infty}_{n-1}}  
& \leq C  \big(\|u_\perp\|_{L^{q_3}_{n-1}} \|\nabla u_\perp\|_{H^{1;q,\infty}} + \|\nabla u_\perp\|_{L^{ q_3,\infty\BLACK}} \|\nabla u_\perp\|_{L^{q}_{n-1}}\big)\; 
\nonumber\\
& \leq C \|\nabla u_\perp\|_{H^{1;q,\infty}} \| u_\perp\|_{H^{1;q,}_{n-1}}. 
\end{align}
%
The estimates \eqref{P_perp-terms-est} and \eqref{u_su_p-H}, \eqref{u_pu_s-H}, \eqref{u_pu_p-H} are used to control $N_1(u,p)(t)$ in the space $L^\infty(\IR; L^{q}_{n-1+\delta} \cap H^{1,q_2,\infty}_{ n-1} )$
and hereby the solution operator $H=H(N)$. 
\end{proof}

\vspace{1ex}

\begin{proof}[Proof of Theorem \ref{main}]
Let the solution operator of the stationary Stokes equation \eqref{stationary-problem} with right-hand side $F, g_1, g_2$ be called $G=G(F,g_1,g_2)$, {\em i.e.,} $(u_s,p_s) = G(F,g_1,g_2)$ satisfying the estimates of Proposition \ref{result-linear-stationary}. 
In the present case where
\begin{align*} 
\begin{aligned}
F & = -u_s\otimes u_s,\\
g_1 & = -P_s \big(u_s \cdot (\nabla^{\phi(t)} - \nabla^{\phi(0)}) u_s + u_s\cdot \nabla^{\phi(t)} u_\perp + u_\perp\cdot \nabla^{\phi(t)} u_\perp + u_\perp\cdot \nabla^{\phi(t)} u_s\big)+f_s,\\
g_2 & = -P_s \big( (\nabla^{\phi(t)}-\nabla^{\phi(0)})p - M_0(u)  -M_1(u,p)\BLACK\big)
\end{aligned}
\end{align*}
we get from the proof of Proposition \ref{est-K-N} the estimate
\begin{align*}
\|G & (F, g_1,g_2)\|_{X_s} \\
& \leq C\big(\|u_s,p_s\|_{X_s}^2 + \|u_\perp, p_\perp\|_{X_\perp}^2\big) + C\epsilon (\|\nabla^2u,\nabla p\|_{L^{q,\infty}} + \|u\|_{H^{1,q}}) + C\|f_s\|_{L^q_n}.
\end{align*}

Moreover,  $H[u,p]=H(\mathbb P N)$ \BLACK satisfies by Propositions \ref{result-oscillatory-part-2nd-derivative}, \ref{result-oscillatory-part-0-1-order} (iii) and   \ref{est-K-N} the estimate
\begin{align*} 
\|H(\mathbb P N)\|_{X_\perp} 
& \leq C\big(\|u_s,p_s\|_{X_s}^2  + \|u_\perp, p_\perp\|_{X_\perp}^2\big) + C\epsilon (\|\nabla^2u,\nabla p\|_{L^{q,\infty}} + \|u\|_{H^{1,q}})\\
& \quad + C\|f_\perp\|_{L^\infty (\IR;L^q_{n-1+\delta}
)}.
\end{align*}
Finally, by \eqref{u-perp-eq-p}, 
\begin{equation}
\nabla p_\perp = (I- \mathbb{P})\Delta u_\perp + (I-\mathbb{P}) N(u,p),   
\end{equation}
we see that 
$$
\|\nabla p_\perp\|_{L^{q,\infty}} 
\leq C_1\|\nabla^2 u_\perp\|_{L^{q,\infty}} +  C_1\mathscr{L}_\mathscr{N} (N(u,p)).
$$
Adding this estimate times $1/2C_1$ to the above estimates and taking $\epsilon_f$ in the smallness assumption \eqref{f-small} \BLACK  and the ball $\mathcal B$ in the fixed point argument sufficiently small, the contraction mapping principle can be applied.
 
To apply Banach's fixed point theorem we use the iteration in 
$$ W= (u,p)= (u_s, p_s) + (u_\perp, p_\perp)  = W_s+W_\perp \cong (W_s,W_\perp) \in X_s \times X_\perp $$ 
defined by 
$W^{(1)}\mapsto W^{(2)} = (G,H)[W^{(1)}] $ for any $W^{(1)} \in 
X_s \times X_\perp$, where the notation $[W]$ stresses the point that the right-hand side is mainly controlled by $W=(u,p)$, but may depend on $(f_s,f_\perp)$ as well. Actually, there holds the estimate
\begin{align}\label{W1W2}
\begin{aligned}
\|W^{(2)}\|_{X_s \times X_\perp} & \leq C\|W^{(1)}\|^2_{X_s \times X_\perp } + C\epsilon \|W^{(1)}\|_{X_s \times X_\perp} \\[1ex]
& \quad + C\big( \|f_s\|_{L^q_n} + \|f_\perp\|_{L^\infty (\IR;L^q_{n-1+\delta} )}
\big
).
\end{aligned}
\end{align}
Moreover, given any $V^{(1)}, W^{(1)} \in X_s \times X_\perp$ with solution  $V^{(2)}, W^{(2)} \in X_s \times X_\perp$  of \eqref{stationary-problem} and \eqref{u-perp-m}, respectively, we get the estimate
\begin{align}\label{W1-W2}
\|V^{(2)} - W^{(2)} \|_{X_s \times X_\perp}  
& \leq C\big(\|V^{(1)}\|_{X_s \times X_\perp}  + \|W^{(1)}\|_{X_s \times X_\perp} + \epsilon \big)
   \|V^{(1)} - W^{(1)} \|_{X_s \times X_\perp} .
\end{align}

By \eqref{W1W2}, \eqref{W1-W2} there exists a small closed ball $\mathcal B$ in the Banach space $X_s \times X_\perp$ on which the map $(G,H)[\cdot]$ is a strict contraction. Hence there exists a locally unique fixed point $W=(u,p)$ of $(G,H)$ in $\mathcal B \subset X_s \times X_\perp$. 
Then $(u, p)$ is a (locally unique) $T$-periodic solution of the modified Navier-Stokes system \eqref{equ:ns2}. 
\end{proof}
\BLACK

\vspace{1ex}

\begin{rem} \label{rem-norms-f}
{\rm 

(i) For the right-hand side in Proposition \ref{result-linear-stationary} and for $K(u,p)$ and $N(u,p)$ we did not consider norms, but a control term consisting of different norms for different types of functions. However, in the final application of the fixed point iteration, the right-hand side and $K(u,p),\,N(u,p)$ depend linearly and quadratically on $(u,p)$ to be estimated by norms in $X^1,\,X^2$ and terms defined by the external force which is constant in the iteration, but without control by a norm. Nevertheless, Banach's fixed point theorem is applicable.

(ii) There are several choices for function spaces of  external forces. Besides $f_s=g\in L^q_n(\Omega_0)$ and $f_\perp \in L^\infty (\IR;L^q_{n-1+\delta}(\Omega_0))$, see Theorem \ref{main},  we may allow for terms like $f_s=\div F$ where $F\in L^\infty_{n-1}$, $\div F\in L^\infty_n$ or $f=g_2\in L^{q_0}_n\cap L^{q,\infty}$, {\em cf.} Proposition \ref{result-linear-stationary}; 
moreover, concerning $F_\perp$, we can consider 
$f_\perp \in L^\infty_{per}(\IR;L^{q_0}_{n-1})  \cap L^\infty_{per}(\IR; L^{q,\infty}_{n-1})$, {\em cf.} Proposition \ref{result-oscillatory-part-2nd-derivative}, Proposition \ref{result-oscillatory-part-0-1-order} (iii).
}
\end{rem}

\vspace{2ex}

\noindent {\bf Acknowledgements.}
The second author is supported by JSPS grant whose number is 22K13946.\vspace{1ex} 

\noindent {\bf Statements and Declarations}\vspace{1ex}
 
\noindent {\bf Conflicts of interest statement.}  There is no conflict of interest. \vspace{1ex} 

\noindent {\bf Data Availability statement.} No datasets were generated or analysed during the current study.


\begin{thebibliography}{99}



\bibitem{BM-ALE} 
A. M. Benselama and J. Monnier, Navier-Stokes ALE free surface flow with generalized Navier slip conditions. Droplet impact
 and attempt using Comsol Multiphysics 3.2., INRIA $n^\circ$ 6175 (2007)

\bibitem{BerghL} J. Bergh and J. L\"ofstr\"om, Interpolation Spaces: An Introduction. Springer-Verlag,
Berlin Heidelberg New York 1976

\bibitem{BoMi95} W. Borchers and T. Miyakawa,
On stability of exterior stationary Navier-Stokes flows,
Acta Math. {\bf 174} (1995), 311--382

\bibitem{Celik-K} A. Celik and M. Kyed, Nonlinear acoustics: Blackstock-Crighton equations with a periodic forcing term, J. Math. Fluid Mech. {\bf 21:45} (2019) 

\bibitem{Chua} S. Chua, 
Extension theorems on weighted Sobolev spaces, Indiana Univ. Math. J. {\bf 41} (1992), 1027--1076


\bibitem{Cw92} M. Cwikel,
Real and complex interpolation and extrapolation of compact operators, Duke Math. J. {\bf 65}, no. 2, (1992), 333--343

\bibitem{DHMT}  R. Danchin, M. Hieber, P. Mucha, P. Tolksdorf, 
Free boundary problems via
Da Prato–Grisvard theory, arXiv:2011.07918bv1, to appear in: Memoirs of the AMS. \BLACK

\bibitem{DHP} W. Desch, M. Hieber and J. Pr\"u\ss,
$L^p$-Theory of the Stokes equation in a half space, J. Evol. Equ. {\bf 1} (2001), 115--142

\bibitem{DHP2} R. Denk, M. Hieber and J.  Pr\"uss, 
    $\mathcal R$-Boundedness, Fourier Multipliers and Problems of Elliptic and Parabolic Type, Mem. Amer. Math. Soc. {\bf 166} (2003), no. 788

\bibitem{DGN-ALE} 
F. Duarte, R. Gormaz and S. Natesan, Arbitrary Lagrangian–Eulerian method for Navier–Stokes equations with moving boundaries, Comput. Methods. Appl. Mech. Engrg. {\bf 193} (2004), 4819--4836 

\bibitem{Du-Si} 
R. Duarte and J.D. Silva, 
Weighted Gagliardo-Nirenberg interpolation
inequalities, 
J. Funct. Anal. {\bf 285} (2023),  110009

\bibitem{Eiter-Shibata}
T. Eiter and Y. Shibata, 
Viscous flow past a translating body with oscillating boundary, J. Math. Soc. Japan {\bf 77} (2024), 103--134 
 
\bibitem{Eiter-Kyed-Shibata21} T. Eiter, M. Kyed and Y. Shibata,
On periodic solutions for one-phase and two-phase problems of the Navier–Stokes equations, J. Evol. Equ. {\bf 21} (2021), 2955–3014 

\bibitem{Eiter-Kyed-Shibata23}
T. Eiter, M. Kyed and Y. Shibata, Periodic $L_p$ 
 estimates by ${\mathcal R}$-boundedness: Applications to the Navier-Stokes equations,  Acta Appl. Math. {\bf 188:1} (2023) 

\bibitem{Farwig-Kozono-Wegmann} R. Farwig, H. Kozono and D. Wegmann, Maximal regularity of the Stokes operator in an exterior
domain with moving boundary and application to the
Navier-Stokes equations, Math. Ann. {\bf 375} (2019), 949--972

\bibitem{FKTW}
R. Farwig, H. Kozono, K. Tsuda and D. Wegmann, The time periodic problem of the Navier-Stokes equations in a bounded domain with moving boundary,
Nonlinear Anal. Real World Appl. {\bf 61} (2021), 103339

\bibitem{Farwig-Sohr-1994} 
R. Farwig and H. Sohr, Generalized resolvent estimates for the Stokes
 system in bounded and unbounded domains. J. Math. Soc. Japan
 {\bf 46} 1994, 607--643

\bibitem{FS} R. Farwig and H. Sohr, 
Weighted $L^q$-theory for the Stokes resolvent in exterior domains, J. Math. Soc. Japan {\bf 49} (1997), 251--288
	
\bibitem{FT-unif} R. Farwig and K. Tsuda, Uniform estimates for fractional operators, Partial Differ. Equ. Appl. {\bf 2}, 27 (2021)

\bibitem{FT-initial-value-prob}  R. Farwig and  K. Tsuda, The Fujita-Kato approach for the Navier-Stokes equations with moving boundary and its application,  J. Math. Fluid Mech. {\bf 24:77} (2022).
	
\bibitem{FT-lin-half} R. Farwig and K. Tsuda, The time periodic problem for the Navier-Stokes equations on half spaces with moving boundary: Linear theory,  J. Differential Equations {\bf 411} (2024), 531--603

\bibitem{FT-lin-half2} R. Farwig and K. Tsuda, 
The time periodic problem for the Navier-Stokes equations on half spaces with moving boundary:~Nonlinear theory, Math. Ann. {\bf 39} (2025), 4791--4846

\bibitem{Farwig-Tsuda-exterior} R. Farwig and K. Tsuda, 
The time periodic problem for the Navier-Stokes equations
in exterior domains in weighted spaces, arXiv:2409.17590
    
\bibitem{Farwig-Tsuda-stab} R. Farwig and K. Tsuda, 
Critical decay rate of stability for stationary solutions
to the Navier-Stokes equations in exterior domains, submitted (2025)


\bibitem{Farwig-Tsuda-Hinfty} R. Farwig and K. Tsuda,
The Stokes operator on exterior domains in homogeneous weighted function spaces: From weak theory to $\mathscr H^\infty$-calculus to fractional domains, submitted 2025
\BLACK

\bibitem{Freitag} D. Freitag,
Real interpolation  of  weighted $L_p$-spaces,
Math. Nachr. {\bf 86} (1978), 15--18

\bibitem{FroehII} A. Fr\"ohlich, 
The Stokes operator in weighted $L^p$-spaces II: Weighted resolvent estimates and maximal $L^p$-regularity, Math. Ann. {\bf 339} (2007), 287--316

\bibitem{Galdi-steady} G.P. Galdi,
An Introduction to the Mathematical Theory of the Navier-Stokes Equations.
Steady-State Problems.
Springer New York, 2nd ed. 2011  


\bibitem{Giga} Y. Giga, 
Domains of fractional powers of the Stokes operator in $L_r$ spaces, Arch. Rational Mech. Anal. {\bf 89} (1985), 251–265

 
\bibitem{Giga-Sohr} Y. Giga and H. Sohr, 
On the Stokes operator on exterior domains, J. Fac. Sci. Univ. Tokyo. Sect. IA. Math. {\bf 36} (1989), 103--130

\bibitem{Hishida}
T. Hishida, Spatial pointwise behavior of time-periodic Navier–Stokes flow induced by oscillation of a moving obstacle, J. Math. Fluid Mech. {\bf 24} (2022) 

\bibitem{HYZ}
H.  Hajaiej, X. Yu and Z. Zhai, Fractional Gagliardo–Nirenberg and Hardy inequalities under Lorentz norms, J. Math. Anal. Appl. {\bf 396} (2012), 569--577 

\bibitem{Iwashita} H. Iwashita,
$L_q-L_r$, estimates for solutions of the nonstationary Stokes equations in an exterior domain and the Navier-Stokes initial value problems in $L_q$ spaces,
Math. Ann. {\bf 285}  (1989), 265--288

\bibitem{Kalton} N. Kalton, P.  Kunstmann and L. Weis, Perturbation and interpolation theorems for the $H^\infty$-calculus with applications to differential
	operators, Math. Ann. {\bf 336} (2006), 747--801

\bibitem{Kaniel-Shinbrot}
S. Kaniel and M. Shinbrot, A reproductive property of the Navier-Stokes equations, 
Arch. Rational Mech. Anal. {\bf 24} (1967),  363--369    
		
\bibitem{Kobayashi-Kubo-half} 
T. Kobayashi and T. Kubo, 
Weighted $L^p-L^q$ estimates of Stokes semigroup in half-space and its application to the Navier-Stokes equations, Adv. Math. Fluid Mech., Recent developments of mathematical fluid mechanics (2016),  337--349

\bibitem{Kozono-Nakao} H. Kozono and M. Nakao, Periodic solutions of the Navier-Stokes equations in unbounded domains, Tohoku Math. J. {\bf 48} (1996), 33--50

\bibitem{KuWe04} P. Kunstmann and L. Weis, Maximal $L_p$-regularity for parabolic equations,
    Fourier multiplier theorems and
    $H^\infty$-functional calculus,  Functional analytic methods for evolution equations,
    M. Iannelli et al. (eds.), Lecture Notes in Math. {\bf 1855} (2004), 65--311

\bibitem{KurtzWheeden} D.S. Kurtz and R.L. Wheeden, Results on weighted norm inequalities for multi-
pliers, Trans. Amer. Math. Soc. {\bf 255} (1979), 343--362

\bibitem{Miyakawa-Teramoto} T. Miyakawa and Y. Teramoto, Existence and periodicity of weak solutions of the Navier-Stokes equations in a time dependent domain,
	Hiroshima Math. J. {\bf 12} (1982), 513--528   
\bibitem{Mohsenipour-Sadeghi} M. Mohsenipour and G. Sadeghi,  Interpolation between weighted Lorentz spaces
 with respect to a vector measure,  Hacet. J. Math. Stat.
 {\bf 48} (2019), 1590--1600 

\bibitem{Morimoto}
H. Morimoto, 
On existence of periodic weak solutions of the Navier-Stokes equations in regions with periodically moving boundaries, 
J. Fac. Sci. Univ. Tokyo Sect. IA Math. {\bf 18} (1971/72), 499–-524

\bibitem{NollSaal} A. Noll and J. Saal, $H^\infty$-calculus for the Stokes operator on $L_q$-spaces, Math. Z. {\bf 244} (2003), 651--688  

\bibitem{OgawaShimizu}  T. Ogawa, S. Shimizu, Maximal $L^1$-regularity and free boundary problems for the incompressible Navier–Stokes equations in critical spaces, J. Math. Soc. Japan {\bf 76(2)} (2024), 593--672 
 
\bibitem{Oishi-Shibata} K. Oishi and Y. Shibata, On the global well-posedness and decay of a free boundary problem of the Navier–Stokes equation in unbounded domains, Mathematics {\bf} (2022), 10, 774

\bibitem{OkabeTsutsui} T. Okabe and Y. Tsutsui,  Time periodic strong solutions to the incompressible Navier–Stokes equations with external forces of non-divergence form, J. Differential Equations
{\bf 263} (2017), 8229--8263

\bibitem{Saal} J. Saal, Maximal regularity for the Stokes system on noncylindrical
space-time domains, J. Math. Soc. Japan {\bf 58} (2006),  617--641
	
\bibitem{Saal6thMiss} J. Saal, Strong solutions for the Navier-Stokes equations on bounded and unbounded domains with a moving
boundary, 
Electron. J. Differ. Equ. Conf.  {\bf 15} (2007), 365--375

\bibitem{Serrin} 
J. Serrin, 
A note on the existence of periodic solutions of the Navier-Stokes equations,  
Arch. Rational Mech. Anal. {\bf 3} (1959),  120--122

\bibitem{Shibata-lecturer} 
Y. Shibata, Boundedness, maximal regularity and free boundary problems for the Navier-Stokes equations, Lecture Notes in Math. 
{\bf 2254} (2020), 193--462

\bibitem{Teramoto} Y. Teramoto, On the stability of periodic solutions of the Navier-Stokes equations in a noncylindrical domain, Hiroshima Math. J. {\bf 13} (1983), 607--625



\bibitem{Triebel}
H. Triebel, Interpolation theory, function spaces, differential operators, Verlag der Wissenschaften,
Berlin 1978 


\bibitem{Triebel2010} H. Triebel, 
Theory of function spaces, Modern Birkh\"auser Classics, Springer, Basel 2010, reprint of Akademische Verlagsgesellschaft Geest \& Portig, Leipzig 1983


\bibitem{Tsuda2016}
K. Tsuda, On the existence and stability of time periodic solution to the compressible Navier-Stokes equation on the whole space, Arch. Ration. Mech. Anal. {\bf 219} (2016), 637--678


\bibitem{Yamazaki} M. Yamazaki, 
The Navier-Stokes equations in the weak-$L^n$ space with time-dependent external force, 
Math. Ann. {\bf 317}  (2000), 635--675

\end{thebibliography}
\end{document}